\documentclass[a4paper,10pt]{article}
\usepackage[utf8]{inputenc}
\usepackage[left=4cm, right=3cm]{geometry}

\usepackage{amsmath, amssymb, amsthm}
\usepackage{mathtools}
\usepackage{mathrsfs}
\usepackage{bbm}

\usepackage{tikz-cd}
\usetikzlibrary{decorations.pathmorphing}
\tikzcdset{arrow style = tikz,diagrams={>={Classical TikZ Rightarrow}}}

\usepackage{enumitem}
\usepackage[all]{xy}
\usepackage{url}
\usepackage{xcolor}
\usepackage{mdframed}

\usepackage{stmaryrd}

\usepackage{marginnote}
\usepackage{comment}

\usepackage{letltxmacro}

\usepackage{listings}
\theoremstyle{definition}

\newtheorem{defn}{Definition}[section]
\newtheorem{aux-defn}[defn]{Auxiliary Definition}
\newtheorem{ex}[defn]{Example}
\newtheorem{non-ex}[defn]{Non-Example}
\newtheorem{sitn}[defn]{Situation}
\newtheorem{var-def}[defn]{Variant}

\newtheorem{nota}[defn]{Notation}

\newtheorem{constr}[defn]{Construction}
\newtheorem{rem}[defn]{Remark}

\newtheorem{prob-upright}[defn]{Problem}

\theoremstyle{plain}

\newtheorem{prop}[defn]{Proposition}
\newtheorem{cor}[defn]{Corollary}
\newtheorem{lem}[defn]{Lemma}
\newtheorem{univ-prop}[defn]{Universal Property}
\newtheorem{thm}[defn]{Theorem}
\newtheorem{lem-def}[defn]{Lemma and Definition}

\newtheorem{clm}[defn]{Claim}

\newenvironment{enum:arabic}{\enumerate[label=\textnormal{(\arabic*)}]}{\endenumerate}
\newenvironment{enum:alph}{\enumerate[label=\textnormal{(\alph*)}]}{\endenumerate}
\newenvironment{enum:roman}{\enumerate[label=\textnormal{(\roman*)}]}{\endenumerate}

\DeclareFontFamily{U}  {MnSymbolC}{}
\DeclareSymbolFont{MnSyC}         {U}  {MnSymbolC}{m}{n}
\SetSymbolFont{MnSyC}       {bold}{U}  {MnSymbolC}{b}{n}
\DeclareFontShape{U}{MnSymbolC}{m}{n}{
	<-6>  MnSymbolC5
	<6-7>  MnSymbolC6
	<7-8>  MnSymbolC7
	<8-9>  MnSymbolC8
	<9-10> MnSymbolC9
	<10-12> MnSymbolC10
	<12->   MnSymbolC12}{}
\DeclareFontShape{U}{MnSymbolC}{b}{n}{
	<-6>  MnSymbolC-Bold5
	<6-7>  MnSymbolC-Bold6
	<7-8>  MnSymbolC-Bold7
	<8-9>  MnSymbolC-Bold8
	<9-10> MnSymbolC-Bold9
	<10-12> MnSymbolC-Bold10
	<12->   MnSymbolC-Bold12}{}

\DeclareMathSymbol{\invneg}{\mathord}{MnSyC}{181}
\DeclareMathSymbol{\blackdiamond}{\mathord}{MnSyC}{109}
\DeclareMathSymbol{\largestar}{\mathord}{MnSyC}{131}
\DeclareMathSymbol{\largediamond}{\mathord}{MnSyC}{111}
\DeclareMathSymbol{\splus}{\mathord}{MnSyC}{20}
\DeclareMathSymbol{\stimes}{\mathord}{MnSyC}{21}
\DeclareMathSymbol{\ssqcup}{\mathord}{MnSyC}{64}

\DeclareFontFamily{U}{mathb}{\hyphenchar\font45}
\DeclareFontShape{U}{mathb}{m}{n}{
	<5> <6> <7> <8> <9> <10> gen * mathb
	<10.95> mathb10 <12> <14.4> <17.28> <20.74> <24.88> mathb12
}{}
\DeclareSymbolFont{mathb}{U}{mathb}{m}{n}
\DeclareFontSubstitution{U}{mathb}{m}{n}

\DeclareMathSymbol{\pluscirc}{\mathbin}{mathb}{'011}
\DeclareMathSymbol{\sqbullet}{\mathbin}{mathb}{'015}

\gdef\notesoff{\gdef\note##1{}}
\notesoff

\DeclareFontFamily{U}{mathx}{}
\DeclareFontShape{U}{mathx}{m}{n}{ <-> mathx10 }{}
\DeclareSymbolFont{mathx}{U}{mathx}{m}{n}
\DeclareFontSubstitution{U}{mathx}{m}{n}
\DeclareMathAccent{\widecheck}{0}{mathx}{"71}

\newcommand{\kk}{\Bbbk}
\newcommand{\kq}{\mathbbm{q}}

\newcommand{\itUpsilon}{\mathit{\Upsilon}}

\newcommand{\IA}{\mathbb{A}}

\newcommand{\CC}{\mathbb{C}}

\newcommand{\GG}{\mathbb{G}}

\newcommand{\NN}{\mathbb{N}}

\newcommand{\PP}{\mathbb{P}}
\newcommand{\QQ}{\mathbb{Q}}
\newcommand{\RR}{\mathbb{R}}

\newcommand{\ZZ}{\mathbb{Z}}

\newcommand{\cA}{\mathcal{A}}
\newcommand{\cB}{\mathcal{B}}
\newcommand{\cC}{\mathcal{C}}

\newcommand{\cE}{\mathcal{E}}
\newcommand{\cF}{\mathcal{F}}

\newcommand{\cH}{\mathcal{H}}
\newcommand{\cI}{\mathcal{I}}

\newcommand{\cK}{\mathcal{K}}
\newcommand{\cL}{\mathcal{L}}
\newcommand{\cM}{\mathcal{M}}
\newcommand{\cN}{\mathcal{N}}
\newcommand{\cO}{\mathcal{O}}
\newcommand{\cP}{\mathcal{P}}

\newcommand{\cS}{\mathcal{S}}

\newcommand{\cU}{\mathcal{U}}

\newcommand{\cW}{\mathcal{W}}

\newcommand{\scA}{\mathscr{A}}
\newcommand{\scB}{\mathscr{B}}
\newcommand{\scC}{\mathscr{C}}
\newcommand{\scD}{\mathscr{D}}
\newcommand{\scE}{\mathscr{E}}

\newcommand{\scK}{\mathscr{K}}

\newcommand{\scS}{\mathscr{S}}
\newcommand{\scT}{\mathscr{T}}

\newcommand{\scV}{\mathscr{V}}
\newcommand{\scW}{\mathscr{W}}
\newcommand{\scX}{\mathscr{X}}

\newcommand{\scZ}{\mathscr{Z}}

\newcommand{\fa}{\mathfrak{a}}

\newcommand{\fm}{\mathfrak{m}}
\newcommand{\fn}{\mathfrak{n}}

\newcommand{\fp}{\mathfrak{p}}
\newcommand{\fq}{\mathfrak{q}}
\newcommand{\fr}{\mathfrak{r}}

\DeclareMathOperator{\Spec}{Spec}

\DeclareMathOperator{\Proj}{Proj}

\mathchardef\hy"2D 

\newcommand{\eps}{\varepsilon}
\newcommand{\cp}{\varphi}

\newcommand{\lookup}{\colorbox{yellow}{Look it up!}}

\newcommand{\m}[1]{\mathrm{#1}}

\newcommand{\ms}[1]{\mathsf{#1}}

\newcommand*{\qedloz}{%
	\leavevmode\unskip\penalty9999 \hbox{}\nobreak\hfill
	\quad\hbox{$\lozenge$}%
}

\newcommand*{\qeddef}{%
	\leavevmode\unskip\penalty9999 \hbox{}\nobreak\hfill
	\quad\hbox{$\blacklozenge$}%
}

\newcommand*{\qedex}{%
	\leavevmode\unskip\penalty9999 \hbox{}\nobreak\hfill
	\quad\hbox{$\lozenge$}%
}

\DeclareFontFamily{U}{mathx}{\hyphenchar\font45}
\DeclareFontShape{U}{mathx}{m}{n}{
	<-6> mathx5 <6-7> mathx6 <7-8> mathx7
	<8-9> mathx8 <9-10> mathx9
	<10-12> mathx10 <12-> mathx12
}{}
\DeclareSymbolFont{mathx}{U}{mathx}{m}{n}
\DeclareFontSubstitution{U}{mathx}{m}{n}
\DeclareMathAccent{\widebar}{0}{mathx}{"73}

\title{Generically log smooth families \\ via generators and relations}
\author{Simon Felten}

\begin{document}

\maketitle

\begin{abstract}
	Let $f\colon X \to \IA^1_t$ be an affine flat morphism of finite type, and let $V = f^{-1}(0)$. Then, we obtain a morphism of log schemes $f\colon (X|V) \to (\IA^1_t|0)$. In this article, we develop algorithmic tools to study the log-geometric properties of $f$ by means of a presentation \[\Gamma(X,\cO_X) = \kk[t,x_1,\ldots,x_n]/(f_1,\ldots,f_r).\]
	We obtain similar tools for projective flat morphisms when the homogeneous coordinate ring is given by generators and relations. We provide an implementation of our algorithms in Macaulay2. In a slightly different direction, we give some results on the sheaf $\cL\cS_V$ of log smooth structures on a toroidal crossing scheme $(V,\cP,\bar\rho)$.
\end{abstract}

\newpage

\tableofcontents

\section{Introduction}

In this introduction, we motivate our results and give an overview. We discuss them in detail in Part I (which can and should be read as an extended introduction), and the proofs are provided in Part II.

\subsection{An algorithmic criterion for generic log smoothness}

We discuss the main result of this article.

\bigskip

\noindent\textbf{Degenerating families of algebraic varieties.} We fix an algebraically closed field $\kk$ of characteristic $0$. Let $B/\kk$ be a smooth curve, and let $f\colon X \to B$ be a degenerating family of algebraic varieties. Let $0 \in B$ be a point, and suppose that $V \coloneq f^{-1}(0)$ is the only singular fiber. 

\begin{ex}\label{ex:intro-quartic}\note{ex:intro-quartic}
	We take $\overline B = \Spec \kk[t]$ and $\overline X = \Proj \kk[t,x,y,z,w]/(f_1)$ with $\m{deg}(t) = 0$ and $\m{deg}(x) = \m{deg}(y) = \m{deg}(z) = \m{deg}(w) = 1$, and with 
	\[f_1 = xyzw - t(x^4 + y^4 + z^4 + w^4).\]
	This yields a degeneration $\bar f\colon \overline X \to \overline B$ of a smooth quartic surface in $\PP^3$ to $V = \bar f^{-1}(0) = \{xyzw = 0\}$. After restricting to a suitable open subset $0 \in B \subseteq \overline B$, we obtain a family $f\colon X \to B$ such that $V = f^{-1}(0)$ is the only singular fiber. For an illustration of the central fiber, see \cite[Fig.~1.2, Fig.~1.6]{FeltenGLDT}. \qedex
\end{ex}
\begin{ex}\label{ex:intro-Fano}\note{ex:intro-Fano}
	We take $\overline B = \Spec \kk[t]$ and $\overline X = \Proj \kk[t,x,y,z,u,v,w]/(f_1,f_2)$ with $\m{deg}(t) = 0$ and the other variables of degree $1$, and with
	\[f_1 = xy - z^2 + t^2(u^2 + v^2 + w^2) \quad \text{and} \quad f_2 = zw - t(uv - x^2 - y^2).\]
	This yields a degeneration $\bar f\colon \overline X \to \overline B$ of a smooth Fano threefold in $\PP^5$. After restricting to a suitable open subset $0 \in B \subseteq \overline B$, we obtain a family $f\colon X \to B$ whose only singular fiber is $V = f^{-1}(0)$. We will come back to these equations in Example~\ref{ex:projective-112}, where the reader can also find an illustration. \qedex
\end{ex}

We turn $f$ into a morphism of logarithmic schemes $f\colon (X|V) \to (B|0)$ by endowing source and target with the divisorial log structure, defined in the \'etale topology. For a description of this construction, see \cite[\S III.1.6]{LoAG2018}. In some cases, for example when $f\colon X \to B$ is semistable, this morphism is log smooth. Then, its formal properties are similar to those of smooth morphisms, and we can use log-geometric tools to analyze the degeneration. For example, when $f\colon X \to B$ is proper, the log Hodge--de Rham spectral sequence degenerates at $E_1$, and the log Hodge sheaves are vector bundles \cite{kkatoFI}. We also have a logarithmic Gromov--Witten theory \cite{Gross2013,ACGS2020,ACGS2025} that relates curve counts on $V = f^{-1}(0)$ to curve counts on $X_t = f^{-1}(t)$ for $t \not= 0$.

In practice, $f\colon (X|V) \to (B|0)$ is not log smooth for most degenerating families of algebraic varieties. For instance, neither the family in Example~\ref{ex:intro-quartic} nor the family in Example~\ref{ex:intro-Fano} is log smooth. However, similar to normal varieties, there is often a closed subset $Z_X \subset X$ of relative codimension $\geq 2$ such that $f\colon \big(X \setminus Z_X\big|V \setminus Z_X\big) \to (B|0)$ is log smooth. We say that $f\colon (X|V) \to (B|0)$ is \emph{generically log smooth}, and that $Z_X \subset X$ is the \emph{log singular locus}. We recall details of this notion in Definition~\ref{defn:gls}. 

Our broader goal is to develop tools that do not only apply to log smooth families but to generically log smooth families. For example, in \cite{FFR2021}, we showed that the log Hodge--de Rham spectral sequence still degenerates at $E_1$ when the log singularities have a certain simple form. The log singularities in Example~\ref{ex:intro-quartic} have this form.

In contrast, the log singularities in Example~\ref{ex:intro-Fano} do not have this specific form. For more general log singularities like these, very little is known. For instance, we do not currently have a definition of the log de Rham complex for this family with the property that the log Hodge sheaves are vector bundles. The first step toward a broader understanding is to analyze concrete examples. In this article, we provide algorithmic tools to establish that a family $f\colon (X|V) \to (B|0)$ given by generators and relations is indeed generically log smooth.

\bigskip

\noindent\textbf{Computing the log singular locus.} Suppose that we want to compute the log singular locus $Z_X \subset X$ from the equations. In Example~\ref{ex:intro-quartic}, the central fiber $V = \{xyzw = 0\}$ is a normal crossing scheme. By \cite[Lemma~1.44]{FeltenGLDT}, the morphism $f\colon X \to B$ is semistable in a point $p \in V$ if and only if $X$ is smooth in $p$. A direct computation yields 
\[\m{Sing}(X) \cap V = \m{Sing}(V) \cap \{x^4 + y^4 + z^4 + w^4 = 0\},\]
a collection of $24$ points, $4$ on each of the six lines of $D = \m{Sing}(V)$. Then, $f\colon X \to B$ is semistable outside these $24$ points, and hence generically log smooth with log singular locus $Z_X = \m{Sing}(X) \cap V$ since $f\colon X \setminus V \to B \setminus \{0\}$ is smooth.

Let us try this approach in Example~\ref{ex:intro-Fano}. The central fiber $V$ has three irreducible components 
\begin{align*}
	V(x,y,z) &= \{t = w = xy - z^2 = 0\} & &\cong \Proj \kk[x,y,z,u,v]/(xy - z^2), \\
	V(x,w) &= \{t = y = z = 0\} & &\cong \Proj \kk[x,w,u,v], \\
	V(y,w) &= \{t = x = z = 0\} &  &\cong \Proj \kk[y,w,u,v].
\end{align*}
Any two irreducible components intersect in a copy of $\PP^2$. Concretely, let us write 
\[D(x) = V(x,w) \cap V(x,y,z),\ D(y) = V(y,w) \cap V(x,y,z),\ D(w) = V(x,w) \cap V(y,w).\]
Then, $D \coloneq D(x) \cup D(y) \cup D(w)$ is both the singular and the non-normal locus of $V$. Furthermore, all three irreducible components of $V$ intersect in $T \cong \PP^1$. 

Now, $V$ is not a normal crossing scheme along $T$. This is no problem as we only want to show log smoothness in codimension $1$. Thus, we record $T$ as part of the locus $Z_X \subset X$ where $f\colon (X|V) \to (B|0)$ is not (necessarily) log smooth. 

We consider $D(x)$. Among points in $D(x) \setminus T$, the total space $X$ is singular in $Z(x) \coloneq D(x) \cap \{x^2 - uv = 0\}$. This is a curve, so we can add it to $Z_X$. At points in $D(x) \setminus (T \cup Z(x))$, the family is semistable. Similarly, it is semistable at points in $D(y) \setminus (T \cup Z(y))$ with $Z(y) = D(y) \cap \{y^2 - uv = 0\}$.

The situation is different for $D(w)$; namely, $X$ is \emph{nowhere} smooth along $D(w)$. Thus, $D(w)$ is a closed subset of relative codimension $1$ where $f\colon X \to B$ is not semistable, and we cannot yet conclude that $f\colon (X|V) \to (B|0)$ is generically log smooth.

Being semistable is not the only way in which $f\colon (X|V) \to (B|0)$ can be log smooth at a point $p \in V$. Let us consider the morphism 
\[\mu_{\ell;d}\colon \enspace M_{\ell;d} = \Spec \kk[t,x,y,z_1,\ldots,z_{d - 1}]/(xy - t^\ell) \to \Spec \kk[t] = B.\]
We say that $\ell \geq 1$ is the \emph{kink}.\footnote{This name is motivated by the shape of the monoid defining the affine toric variety $M_{\ell;d}$, or rather of a piecewise linear function which gives rise to that monoid.} When we endow source and target with the divisorial log structure defined by $t = 0$, then this morphism becomes log smooth. The case $\ell = 1$ is semistable, but for $\ell \geq 2$, this is a different way to be log smooth. Here, the total space is indeed singular along $\{t = x = y = 0\}$.

To show that the family $f\colon (X|V) \to (B|0)$ in Example~\ref{ex:intro-Fano} is generically log smooth, we wish to compute a locus $Z(w) \subset D(w)$ such that, for each closed point $p \in D(w) \setminus (T \cup Z(w))$, the ring homomorphism $f^\sharp\colon \cO_{B,0}^\m{h} \to \cO_{X,p}^\m{h}$ between Henselian stalks is isomorphic to the ring homomorphism $\mu_{\ell;d}^\sharp\colon \cO_{B,0}^\m{h} \to \cO_{M_{\ell;d},0}^\m{h}$ at $0 \in M_{\ell;d}$. We call such a point a \emph{double toroidal crossing (dtc) degeneration point} (of kink $\ell \geq 1$).

\bigskip

\noindent\textbf{Double normal crossing schemes.} Before coming back to dtc degeneration points, we say more about the central fiber. In Example~\ref{ex:intro-Fano}, it is easy to compute directly that $V \setminus T$ is a normal crossing scheme. For more complicated equations, this might not be as easy. In this article, we establish a new computational criterion to determine the \emph{double normal crossing locus} of the central fiber $V$, i.e., the locus where $V$ is \'etale locally equivalent to $\Spec \kk[x,y,z_1,\ldots,z_{d - 1}]/(xy)$. For details, we refer the reader to Section~\ref{sect:dnc-points}.

\begin{prop}\label{prop:intro-dnc}\note{prop:intro-dnc}
	Let $\kk$ be an algebraically closed field, and let $V/\kk$ be reduced, separated, of finite type, and assume that every irreducible component has the same dimension $d \geq 1$. Let $\nu\colon \widetilde V \to V$ be the normalization map, and let $C \subset V$ be the conductor locus. Let $\mu\colon \widetilde C \to C$ be the restriction of the normalization map to $C \subset V$. Then, the double normal crossing locus of $V$ is the largest open subset $W \subseteq V$ with the following properties:
	\begin{enum:alph}
		\item $\widetilde W \coloneq \nu^{-1}(W) \subseteq \widetilde V$ is smooth over $\Spec \kk$;
		\item $C \cap W$ is smooth over $\Spec \kk$, and every irreducible component has dimension $d - 1$;
		\item $\nu\colon \widetilde W \to W$ is unramified;
		\item $\mu_*\cO_{\widetilde C \cap \widetilde W}$ is a locally free $\cO_{C \cap W}$-module of rank $2$.
	\end{enum:alph}
\end{prop}

\bigskip

\noindent\textbf{Double toroidal crossing degeneration points.} After determining the double normal crossing locus of $V$, we wish to determine its dtc degeneration points and hence the locus $Z(w) \subset D(w)$ we hoped for above. For dtc degeneration points, we have the following criterion, which we give in Corollary~\ref{cor:dtc-computable-condition} below. For the notion of a \emph{pre--generically log smooth family} used in the statement, see Definition~\ref{defn:gls}. Briefly, this means that $f$ is a flat morphism of schemes satisfying some extra hypotheses. Note that the criterion assumes that we already know that $p \in V$ is a double normal crossing point, i.e., a point with $\cO_{V,p}^\m{h} \cong \cO_{\{xy = 0\},0}^\m{h}$ as $\kk$-algebras.

\begin{prop}\label{prop:intro-dtc}\note{prop:intro-dtc}
	Let $0 \in B \subseteq \Spec \kk[t]$, and let $f\colon X \to B$ be a pre--gls family with central fiber $V = f^{-1}(0)$. Consider the short exact sequence
	\[0 \to \cO_X \xrightarrow{\iota} \Omega^1_X \xrightarrow{\pi} \Omega^1_{X/B} \to 0\]
	with $\iota(1) = f^*dt$. Let $p \in V$ be a closed double normal crossing point. Then, $p$ is a dtc degeneration point of kink $\ell = 1$ if and only if $\cE xt^1(\Omega^1_X|_{V},\cO_{V})_p = 0$ for the stalk at $p$. It is a dtc degeneration point of kink $\ell \geq 2$ if and only if the following conditions hold:
	\begin{enum:alph}
		\item the map 
		\[\xi_{\ell - 2}\colon \enspace \cE xt^1(\Omega^1_{X_{\ell - 2}/B_{\ell - 2}},\cO_{X_{\ell - 2}}) \to \cE xt^1(\Omega^1_X|_{X_{\ell - 2}},\cO_{X_{\ell - 2}})\]
		is injective at $p$;
		\item the multiplication map 
		\[\sigma_{\ell - 1}\colon \enspace \cE xt^1(\Omega^1_X|_{X_{\ell - 1}},\cO_{X_{\ell - 1}}) \xrightarrow{(-) \cdot f^\sharp(s)^{\ell - 1}} \cE xt^1(\Omega^1_X|_{X_{\ell - 1}},\cO_{X_{\ell - 1}})\]
		is the zero homomorphism at $p$.
	\end{enum:alph}
	Here, $B_k = \Spec \kk[t]/(t^{k + 1})$, and $f_k\colon X_k \to B_k$ is the base change of $f\colon X \to B$ along $B_k \to B$.
\end{prop}

In Example~\ref{ex:intro-Fano}, every point in $D(w) \setminus T$ is a double normal crossing point. The map $\xi_0$ is injective at every $p \in D(w) \setminus T$, and among these points, the image of $\sigma_1$ is non-zero precisely in 
\[Z(w) \coloneq D(w) \cap \{u^2v^2 - u^2w^2 - v^2w^2 - w^4 = 0\}.\]
This is a curve as well, and $f\colon (X|V) \to (B|0)$ is log smooth at every point in $D(w) \setminus (T \cup Z(w))$. Since $V$ and hence $f$ is smooth at points in $V \setminus D$, we find that $f\colon (X|V) \to (B|0)$ is log smooth in every point of $V \setminus Z_V$ for the curve $Z_V \coloneq T \cup Z(x) \cup Z(y) \cup Z(w)$. Now, the restricted map $f\colon X \setminus V \to B \setminus \{0\}$ is smooth, so $f\colon (X|V) \to (B|0)$ is generically log smooth with log singular locus $Z_X = Z_V$.\footnote{The computation shows that the log singular locus (inside $V$) is \emph{at most} $Z_V$. In fact, $T$ is superfluous in the definition of $Z_V$, but we cannot show this with the general computational methods developed in this article. Cf.~Remark~\ref{rem:precise-log-smooth-locus-II}.}

\bigskip 

\noindent\textbf{The algorithmic criterion.} Based on Proposition~\ref{prop:intro-dnc} and Proposition~\ref{prop:intro-dtc}, we develop an algorithmically verifiable sufficient criterion for generic log smoothness. The precise formulation is quite lengthy. We refer the reader to Section~\ref{sect:criterion-explicit-affine} for details in the affine setup. An analogous result for projective families (such as Example~\ref{ex:intro-Fano}) is discussed in Section~\ref{sect:criterion-explicit-projective}. In Examples~\ref{ex:112} and \ref{ex:projective-112}, we illustrate how to check the conditions step by step, in the affine respectively projective case. For both the affine and the projective case, we provide an implementation in Macaulay2 which automatizes the analysis.

\subsection{Toroidal crossing schemes}

We discuss another result that we also prove in this article. Although it is not a necessary step toward our main theorem, we decided to include it because the methods we use in the proof are closely related to the methods we use for our main result.

\bigskip

\noindent\textbf{Background and motivation.} A \emph{toroidal crossing scheme} $(V,\cP,\bar\rho)$ is a generalization of a normal crossing scheme. Here, $V/\kk$ is a reduced scheme of finite type, $\cP$ is a sheaf of sharp monoids in the \'etale topology of $V$, and $\bar\rho \in \Gamma(V,\cP)$ is a global section. We impose certain local models analogous to but more general than the local models $\{z_1 \cdot \ldots \cdot z_r = 0\}$ for a normal crossing scheme. These local models come from the full toric boundary of certain Gorenstein affine toric varieties. 

A toroidal crossing scheme $(V,\cP,\bar\rho)$ is intermediate between a mere scheme $V$ and a morphism of log schemes $f_0\colon (V,\cM) \to S_0$ to the standard log point $S_0 = \Spec(\NN \to \kk)$. In fact, we have a sheaf $\cL\cS_V$ of \emph{log smooth structures}, i.e., log structures $\alpha\colon \cM \to \cO_V$ with ghost sheaf $\cP$ and a log smooth morphism to the standard log point $S_0 = \Spec(\NN \to \kk)$ such that $\bar\tau_0 \mapsto \bar\rho$. Here, $\bar\tau_0$ is the generator of $\overline\cM_{S_0} \cong \NN$.

Toroidal crossing schemes arise for instance in the Mumford degeneration of a toric variety and in the toric Gross--Siebert mirror construction \cite{GrossSiebertI,GrossSiebertII}. In these and similar situations, the goal is usually to find a well-behaved section of $\cL\cS_V$ which is defined on an open subset $U \subseteq V$ as large as possible ($U = V$ is often not possible), and then to use logarithmic deformation theory to construct a deformation of $V$ to a smooth or mildly singular scheme. The frequent impossibility of $U = V$ makes it necessary to study log singularities in this context, ideally phrased in an explicit description of $\cL\cS_V$. For a survey on toroidal crossing schemes, see \cite[\S1, \S9]{FeltenGLDT}.

\begin{ex}\label{ex:intro-tc}\note{ex:intro-tc}
	Let $X = \Spec \kk[t,x,y,z,w,u]/(xy - z^2, zw - t)$ and $B = \Spec \kk[t]$, and consider the map $f\colon X \to B$. Let $V = f^{-1}(0)$. Then, $f\colon (X|V) \to (B|0)$ is log smooth, and $\cP \coloneq \overline\cM_X|_V$ together with $\bar\rho \coloneq f^\splus(\bar\tau)|_V$ turn $V$ into a toroidal crossing scheme. 
	
	Note that $V$ has three irreducible components $V(x,y,z)$, $V(x,w)$, and $V(y,w)$ analogous to Example~\ref{ex:intro-Fano}. We denote their intersections by $D(x) = V(x,w) \cap V(x,y,z) \cong \Spec \kk[x,u]$, $D(y) = V(y,w) \cap V(x,y,z) \cong \Spec \kk[y,u]$, and $D(w) = V(x,w) \cap V(y,w) \cong \Spec \kk[w,u]$. All three irreducible components of $V$ intersect in $T \cong \Spec \kk[u]$. For an illustration of $V$, see the figure in Example~\ref{ex:112}. \qedex
\end{ex}

In certain cases like Example~\ref{ex:intro-tc}, Gross and Siebert computed $\cL\cS_V$ as a sheaf of sets in \cite{GrossSiebertI}. In Example~\ref{ex:intro-tc}, the result is as follows: For an open subset $W \subseteq V$, a section in $\Gamma(W,\cL\cS_V)$ is given by three invertible functions $f_x \in \Gamma(W,\cO_{D(x)}^\times)$, $f_y \in \Gamma(W,\cO_{D(y)}^\times)$, and $f_w \in \Gamma(W,\cO_{D(w)}^\times)$ such that $f_x|_T = f_y|_T$ and $(f_x|_T)^2 = f_w|_T$ in $\Gamma(W,\cO_T^\times)$.

When we remove the condition that $f_x$, $f_y$, $f_w$ should be invertible, we obtain a larger sheaf $\cL\cS_V^+$. Then, if $f_x,f_y,f_w \not= 0$, we can set $Z(x) = D(x) \cap \{f_x = 0\}$, $Z(y) = D(y) \cap \{f_y = 0\}$, and $Z(w) = D(w) \cap \{f_w = 0\}$, and $\lambda^+ = (f_x,f_y,f_w)$ yields a section $\lambda \in \Gamma(V \setminus Z_V,\cL\cS_V)$, where $Z_V = Z(x) \cup Z(y) \cup Z(w)$. This yields a generically log smooth morphism $f_0\colon (V,\lambda^+) \to S_0$, where now the log structure is defined only on $V \setminus Z_V$.

In view of our broader goal to construct (generic) log structures which allow us to find a smoothing in the case $V$ proper, we say that $\lambda^+ = (f_x,f_y,f_w)$ is \emph{deformable} if it arises from a degenerating family, i.e., if there is a generically log smooth family $f\colon X \to B$ with central fiber $V \cong f^{-1}(0)$ and $Z_V = Z_X \cap V$ such that we can find an isomorphism $\overline\cM_{X \setminus Z_X}|_{V \setminus Z_V} \cong \cP|_{V \setminus Z_V}$ compatible with $f^\splus(\bar\tau)|_V$ and $\bar\rho$ which exhibits the log structure on the central fiber of $f\colon (X|V) \to (B|0)$ as the section $\lambda = \lambda^+|_{V \setminus Z_V} \in \Gamma(V \setminus Z_V,\cL\cS_V)$.

Not all $\lambda^+ = (f_x,f_y,f_w)$ are deformable. In current work in progress with Andr\'es David Gomez Villegas \cite{FeltenVillegas2026}, we address the question which ones are deformable for Example~\ref{ex:intro-tc} and some closely related specific examples.

Suppose that a generically log smooth family $f\colon X \to B$ with central fiber $V = f^{-1}(0)$ as in Example~\ref{ex:intro-tc} is given via generators and relations, and that we already have an isomorphism of the ghost sheaf on the central fiber with $\cP$. Then, it is not straightforward to compute $(f_x,f_y,f_w)$. In this article, we provide a technique that we will develop into a tool to compute $(f_x,f_y,f_w)$ in \cite{FeltenVillegas2026}.

\bigskip

\noindent \textbf{The result.} To explain the technique, let us go back to a normal crossing scheme $V/\kk$. This is a special case of a toroidal crossing scheme, where we set $\cP = \nu_*\underline\NN_{\widetilde V}$ for the normalization $\nu\colon \widetilde V \to V$. In particular, we have a sheaf of log smooth structures $\cL\cS_V$. In \cite{FFR2021,FeltenGLDT}, we studied a map of sheaves of sets 
\begin{equation}\label{eqn:eta-intro}
	\eta_V\colon \enspace \cL\cS_V \to \cE xt^1(\Omega^1_V,\cO_V)
\end{equation}
which turned out to be injective. This solves our above problem in the case of a normal crossing scheme. Namely, on the one hand, when $\lambda \in \Gamma(V \setminus Z_V,\cL\cS_V)$ is represented by the analog of $(f_x,f_y,f_w)$, the image $\eta_V(\lambda)$ is readily computable, and on the other hand, when $\lambda$ is represented as the central fiber of an explicit morphism $f\colon X \to B$, the very definition of $\eta_V$ allows us to compute $\eta_V(\lambda)$ directly.\footnote{Somewhat more importantly than the application presented here, the map $\eta_V$ is useful to construct sections of $\cL\cS_V$, as we did in \cite{FFR2021}. Moreover, the equivalence between the existence of a global log smooth structure and $d$-semistability, i.e., $\cE xt^1(\Omega^1_V,\cO_V) \cong \cO_D$ for $D = \m{Sing}(V)$, originally due to Friedman \cite{Friedman1983}, is best explained by this map.}

For a toroidal crossing scheme $(V,\cP,\bar\rho)$, we have a map $\eta_V$ as well. However, it is in general no longer injective, in fact often constant. In this article, we introduce for every $\ell \geq 2$ an injective map of sheaves of sets 
\[\eta_V^{(\ell)}\colon \enspace \cL\cS_V|_{\cU_1^{(\ell)}V} \to \cE xt^1(\Omega^1_V,\cO_V)|_{\cU_1^{(\ell)}V}\]
which is defined only on a suitable open subset $\cU_1^{(\ell)}V$ of $V$. Nonetheless, together with an open subset $\cU_1^{(1)}V \subseteq V$ where the original $\eta_V$ is injective, they jointly cover the double normal crossing locus $\cU_1V \subseteq V$. For more details, see Section~\ref{sect:classifying-map}. 

\subsection{Differential log smoothness}

Our above algorithm only determines the exact log smooth locus inside the double normal crossing locus of $V$; it does not determine log smoothness where $V$ has a more complicated local structure (such as $T$ in Example~\ref{ex:intro-Fano}). The notion of \emph{differential log smoothness} allows us to study log smoothness inside this more complicated locus heuristically by studying local freeness of log differential forms. For details, we refer the reader to Section~\ref{sect:differential-log-smoothness}.

\subsection{Applications}

The results of this article and in particular the implementation will allow us to study the log singularities of explicit degenerations efficiently. Concretely, we plan to use the present results in joint work in progress with Matthias Zach \cite{FeltenZach2026} and Andr\'es David Gomez Villegas \cite{FeltenVillegas2026}. Moreover, it closes a gap in our description of some of the examples in \cite{FeltenGLDT}, where we claimed but did not prove a specific log singular locus, for instance \cite[Ex.~1.158]{FeltenGLDT}.

\subsection{Acknowledgments}

The author thanks his collaborators on two related projects Andr\'es David Gomez Villegas and Matthias Zach for suggestions that are also relevant to this article. The author thanks Alessio Corti for the equations of Example~\ref{ex:projective-112}. The author thanks Helge Ruddat and Matthias Zach for suggestions on a draft of this article.

This work has been financially supported by the Royal Society via Lukas Brantner's grant URF\textbackslash R1\textbackslash 211075.

\newpage

\part{Results and examples}

\section{Generic log smoothness: main results}

We explain our effective criterion to determine whether a family $f\colon X \to B$ of schemes over a curve $B$ defines a generically log smooth family. In the case $0 \in B \subseteq \IA^1_t$, and $f$ either affine or projective, we discuss the computational approach to our criterion. The criterion also has an infinitesimal version, which requires knowledge only of some infinitesimal thickenings $f_k\colon X_k \to B_k$ of $f_0\colon X_0 \to B_0$ rather than of the whole family $f\colon X \to B$.

\subsection{Generically log smooth families}

We use the following notations and conventions for log schemes.

\begin{nota}
	We consider log structures $\cM_X$ on a scheme $X$ as sheaves in the \'etale topology. For a morphism of log schemes $f\colon (X,\cM_X) \to (Y,\cM_Y)$, we write $f^\sharp\colon \cO_Y \to f_*\cO_X$ and $f^\flat\colon \cM_Y \to f_*\cM_X$. We write $f^\splus\colon \overline\cM_Y \to f_*\overline\cM_X$ for the induced homomorphism on ghost sheaves.\footnote{We believe that this is the first place where this notation is used. We use it instead of the common notation $\bar f^\flat$ to avoid confusion with a possible other morphism $\bar f$ and to harmonize the notation with $f^\sharp$ and $f^\flat$. In our upcoming article \cite{Felten2026-extending}, $\overline\cM_X$ will play a more fundamental role than $\cM_X$, so from this perspective, it deserves its own notation in any case.} \qeddef
\end{nota}

In this article, we study log structures which are not defined everywhere. Ample for motivation for doing so can be found in the introduction of \cite{FeltenGLDT}. A \emph{partial log scheme} $(X,U,\cM_U)$ consists of a scheme $X$, an open subset $U \subseteq X$, and a log structure $\alpha_U\colon \cM_U \to \cO_U$ defined on $U$. A morphism of partial log schemes $f\colon (X,U,\cM_U) \to (Y,V,\cM_V)$ consists of a morphism of schemes $f\colon X \to Y$ such that $U \subseteq f^{-1}(V)$ and a morphism of log schemes $f\colon (U,\cM_U) \to (V,\cM_V)$. We consider every log scheme $(X,\cM_X)$ as a partial log scheme by means of $U = X$. The logarithmic analog of a family of normal varieties is as follows.

\begin{defn}\label{defn:gls}\note{defn:gls}
	A morphism of schemes $f\colon X \to S$ is \emph{pre--generically log smooth (pre--gls)} if the following conditions hold:
	\begin{enum:alph}
		\item $S$ is locally Noetherian;
		\item $f\colon X \to S$ is flat, separated, and of finite type;
		\item for every $s \in S$, the fiber $X_s = f^{-1}(s)$ is pure of some fixed dimension $d$, geometrically reduced, and satisfies Serre's condition $(S_2)$, i.e., $\m{depth}(\cO_{X_s,x}) \geq \m{min}\{2,\m{dim}(\cO_{X_s,x})\}$ for every $x \in X_s$.
	\end{enum:alph}
	A \emph{generically log smooth (gls) family} is a morphism of partial log schemes $f\colon (X,U) \to S$ such that:
	\begin{enum:alph}
		\item the underlying morphism of schemes is pre--gls;
		\item for $Z = X \setminus U$, we have $\m{codim}(Z_s,X_s) \geq 2$ for every $s \in S$;
		\item the log structure on $S$ is fine and saturated and defined everywhere;
		\item the log structure on $U$ is fine and saturated; the morphism of log schemes $f\colon U \to S$ is log smooth and saturated. \qeddef
	\end{enum:alph}
\end{defn}
\begin{rem}
	Being pre--gls is stable under base change. Namely, purity of dimension $d$ is stable under field extensions by \cite[Cor.~4.2.8]{EGAIV-2}, and Serre's condition $(S_2)$ is stable under field extensions by the discussion in \cite[\S8.1]{FeltenGLDT}. Similarly, being gls is stable under base change. \qedex
\end{rem}

\subsection{The criterion for generic log smoothness}

We present our criterion to show that a family $f\colon X \to B$ over a smooth curve $B$ is generically log smooth. The hard part in this criterion is log smoothness on an appropriate open subset. Below, we will discuss how to check the easier scheme-theoretic conditions, which we assume for now.

\begin{sitn}\label{sitn:explicit-gls-candidate-kk-intro}\note{sitn:explicit-gls-candidate-kk-intro}
	Let $\kk$ be an algebraically closed field of characteristic $0$. Let $B$ be a smooth affine curve over $\kk$, and let $0 \in B$ be a $\kk$-valued point. After shrinking $B$ in the \'etale topology, we can assume that there is a function $s \in \Gamma(B,\cO_B)$ with $\{0\} = \{s = 0\}$ and such that the induced map $\mathsf{b}_\minuso\colon B \to \IA^1_t$ given by $\mathsf{b}_\minuso^\sharp(t) = s$ is \'etale. 
	
	We endow $B$ with the log structure pulled back from $(\IA^1_t|0) = \Spec(\NN \to \kk[t])$. This coincides with the divisorial log structure defined by $0 \in B$ in the \'etale topology, so we denote the log scheme by $(B|0)$. The element $1 \in \NN$ induces an element $\tau = \mathsf{b}_\minuso^\flat(1) \in \Gamma(B,\cM_{(B|0)})$.
	
	For $k \geq 0$, we write $B_k \subset B$ for the strict closed subscheme cut out by $s^{k + 1} = 0$. Then, we obtain a map $\mathsf{b}_k\colon B_k \to \Spec(\NN \to \kk[t])$ which identifies $B_k$ with $\Spec(\NN \to \kk[t]/(t^{k + 1}))$. We write $\tau_k = \mathsf{b}_k^\flat(1) \in \Gamma(B_k,\cM_{B_k})$.
	
	Let $f\colon X \to B$ be a flat and separated morphism of finite type. For $b \in B$, we write $X_b = f^{-1}(b)$ for the fiber. For $k \geq 0$, we write $f_k\colon X_k \to B_k$ for the base change along $B_k \to B$.\footnote{Note that our two definitions of $X_0$ agree.} Furthermore, we write $B^\star = B \setminus \{0\}$ and $X^\star = X \setminus X_0$.
	
	We assume that the special fiber $V \coloneq X_0$ is reduced and satisfies Serre's condition $(S_2)$. We fix a number $d \geq 1$, which we consider as the \emph{intended relative dimension} of $f\colon X \to B$. We assume that each irreducible component of $V$ has dimension $d$. At this point, we do not impose any similar conditions on the fibers $X_b$ for $b \in B^\star$.
	
	We endow $X$ with the divisorial log structure defined by $V \subset X$ in the \'etale topology. Then, we obtain a morphism of log schemes $f\colon (X|V) \to (B|0)$.
	\qeddef
\end{sitn}

Next, we introduce some notation concerning the central fiber $V$.

\begin{nota}\label{nota:times-stratification}\note{nota:times-stratification}
	Let $\kk$ be an algebraically closed field of characteristic $0$, and let $V/\kk$ be separated, of finite type, and reduced. Let $d \geq 1$, and assume that every irreducible component of $V$ has dimension $d$. First, we set $\cU_0^\times V = \m{Reg}(V) \subseteq V$, the regular locus, which coincides with the smooth locus of $V \to \Spec \kk$. We denote its complement with the reduced induced scheme structure by 
	\[\cC_1^\times V = V \setminus \cU_0^\times V.\]
	We will also use the notation $D \coloneq \cC_1^\times V$ for brevity and refer to $D$ as the \emph{double locus}. Namely, $D$ often coincides with the locus where $V$ has at least two local irreducible components. By \cite[Tag~0C3K]{stacks}, the closed subset $D \subset V$ is cut out by the radical of the $d$-th Fitting ideal of $\Omega^1_{V/\kk}$.
	
	We will see in Definition~\ref{defn:dnc-locus} a maximal open subset $\cU_1^\times V \subseteq V$ such that every $\kk$-valued point in $\cU_1^\times V$ is either regular or a double normal crossing point, i.e., $\cU_1^\times V$ is \'etale locally given by $\Spec\kk[x,y,z_1,\ldots,z_{d - 1}]/(xy)$. Then, we set 
	\[\cC_2^\times V \coloneq V \setminus \cU_1^\times V,\] 
	and we have $\cC_2^\times V \subset \cC_1^\times V$. We endow $\cC_2^\times V$ with the reduced induced scheme structure.
	
	The difference
	\[\cS_1^\times V \coloneq \cC_1^\times V \setminus \cC_2^\times V = \cU_1^\times V \setminus \cU_1^\times V\]
	is a smooth scheme of dimension $d - 1$. We denote the set of irreducible (= connected) components of $\cS_1^\times V$ by $[\cS_1^\times V]$. 
	
	We write $T \subset D = \cC_1^\times V$ for the closed subset cut out by the radical of the $(d - 1)$-th Fitting ideal of $\Omega^1_{D/\kk}$. Since $\cS_1^\times V = D \setminus \cC_2^\times V$ is smooth of dimension $d - 1$, we have $T \subset \cC_2^\times V$. In many examples, we have $T = \cC_2^\times V$. We refer to $T$ as the \emph{triple locus} because it often coincides with the locus where $V$ has at least three local irreducible components.\footnote{The notation introduce here should be compared with our notation for toroidal crossing schemes in Section~\ref{sect:tc-main}.}
	\qeddef
\end{nota}

Finally, to state our theorem, we need the following notation.

\begin{nota}\label{nota:gls-criterion-preparation}\note{nota:gls-criterion-preparation}
	In Situation~\ref{sitn:explicit-gls-candidate-kk-intro}, we write $\varrho\colon \Omega^1_X \to \Omega^1_{X/B}$ for the canonical map. For $k \geq 0$, we denote the (set-theoretic) support of the kernel of 
	\[\xi_k \coloneq \cE xt^1(\varrho|_{X_k},\cO_{X_k})\colon \qquad \cE xt^1(\Omega^1_{X_{k}/B_{k}},\cO_{X_{k}}) \to \cE xt^1(\Omega^1_X|_{X_{k}},\cO_{X_{k}}),\]
	by $K^{(k)} \subset V$, and for $\ell \geq 1$, we denote the (set-theoretic) support of the image of the multiplication map 
	\[\sigma_{\ell - 1}\colon \quad \cE xt^1(\Omega^1_X|_{X_{\ell - 1}},\cO_{X_{\ell - 1}}) \xrightarrow{(-)\cdot f^\sharp(s)^{\ell - 1}}\cE xt^1(\Omega^1_X|_{X_{\ell - 1}},\cO_{X_{\ell - 1}})\]
	by $Z^{(\ell)} \subset V$. \qeddef
\end{nota}

Then, our criterion is as follows.

\begin{thm}\label{thm:explicit-gls-criterion-kk}\note{thm:explicit-gls-criterion-kk}
	Assume that we are in Situation~\ref{sitn:explicit-gls-candidate-kk-intro}, and suppose that $f\colon X \to B$ is pre--gls of relative dimension $d$. We consider the following conditions:
	\begin{enum:alph}
		\item for every $b \in B^\star$, the fiber $X_b$ is normal;
		\item\label{item:T0-codimension} we have $\m{dim}(\cC_2^\times V) \leq d - 2$;
		\item\label{item:ell-exists} for every $E \in [\cS_1^\times V]$, there is some $\ell = \ell(E) \geq 1$ such that $E \cap K^{(k)} = \varnothing$ for all $0 \leq k \leq \ell - 2$ (this condition is empty if $\ell = 1$), and such that $\m{dim}(E \cap Z^{(\ell)}) \leq d - 2$.
	\end{enum:alph}
	Then, the following statements hold:
	\begin{enum:arabic}
		\item The number $\ell(E)$ in Condition~\ref{item:ell-exists} is unique if it exists. We call it the \emph{kink} of $E \in [\cS_1^\times V]$.
		\item If there is a closed subset $Z_X \subset X$ such that $f\colon (X,X \setminus Z_X|V) \to (B|0)$ is generically log smooth, then all three conditions hold.
		\item If the conditions are satisfied, let $\ell = \ell(E)$ and $Z_E \coloneq E \cap Z^{(\ell)}$ (set-theoretically) for $E \in [\cS_1^\times V]$. Then,
		\begin{equation}\label{eqn:clubsuit}\note{eqn:clubsuit}
			Z_X = \m{Sing}(X^\star/B^\star) \cup \cC_2^\times V \cup Z_V^\circ \quad \text{with} \quad Z_V^\circ = \bigsqcup_{E \in [\cS_1^\times V]} Z_E \tag{$\clubsuit$}
		\end{equation} 
		is a choice for $Z_X$ which makes $f\colon (X,X \setminus Z_X|V) \to (B|0)$ generically log smooth. Here, $\m{Sing}(X^\star/B^\star) \subset X^\star$ is the closed subset where $f\colon X^\star \to B^\star$ is not smooth.
		\item If the conditions are satisfied and $f\colon (X|V) \to (B|0)$ is log smooth and saturated around a point $p \in \cU_1^\times V$, then $p \notin Z_V^\circ$. In other words, on $\cU_1^\times V$, the log singular locus is precisely $Z_V^\circ$.
	\end{enum:arabic}
\end{thm}

\begin{rem}\label{rem:precise-log-smooth-locus-I}\note{rem:precise-log-smooth-locus-I}
	We do not currently have an effective computational means to determine the precise log smooth locus inside $T = \cC_2^\times V$. In \cite{FeltenGLDT}, we imprudently claimed a precise log smooth locus in several examples (given by explicit equations) without proof, for instance \cite[Ex.~1.158]{FeltenGLDT}. To see that these log smooth loci are indeed correct, one has to use ad-hoc arguments as in Remark~\ref{rem:precise-log-smooth-locus-II} below. \qedloz
\end{rem}

\subsection{The criterion in explicit affine families}\label{sect:criterion-explicit-affine}\note{sect:criterion-explicit-affine}

We explain how to use Theorem~\ref{thm:explicit-gls-criterion-kk} to show that an affine family $f\colon X \to \IA^1_t$ given by generators and relations is gls. We arrive at an algorithm which we have implemented in Macaulay2. The implementation together with further comments on the code can be found in the file \texttt{methods-gls-families.m2}. In the file \texttt{test-equations.m2}, we provide equations to test the implementation.

\begin{sitn}\label{sitn:affine-gls-candidate}\note{sitn:affine-gls-candidate}
	Let $\kq$ be a field of characteristic $0$. We choose $\kq$ such that commutative-algebraic quantities can be computed algorithmically over $\kq$. In practice, we usually have $\kq = \QQ$. Let $\kq \to \kk$ be a field extension to an algebraically closed field.

	Let $R_{\widebar\scB} = \kq[t]$ and $\widebar\scB = \Spec R_{\widebar\scB}$. Let 
	\[R_{\widebar\scX} = \kq[t,x_1,\ldots,x_n]/(f_1,\ldots,f_r)\]
	be a finitely generated $\kq[t]$-algebra. We write $f^\sharp\colon R_{\widebar\scB} \to R_{\widebar\scX}$ for the ring homomorphism, and $f\colon \widebar\scX \to \widebar\scB$ for the morphism of schemes. We write $\scV = f^{-1}(0)$ for the central fiber.
	
	We fix a number $d \geq 1$ which serves as the \emph{intended relative dimension} of $f\colon \widebar\scX \to \widebar\scB$. 
	
	We write $\m{Sing}(\widebar \scX/\widebar \scB;d) \subset \widebar \scX$ for the closed subset cut out by the $d$-th Fitting ideal $\m{Fitt}_d(\Omega^1_{\widebar \scX/\widebar \scB})$. If $f\colon \widebar \scX \to \widebar\scB$ is flat and has relative dimension $d$, then its complement $\m{Reg}(\widebar \scX/\widebar \scB;d)$ is the open subset of points where $f\colon \widebar \scX \to \widebar \scB$ is smooth, by \cite[Tag~0C3K]{stacks}.
	
	We consider an affine open subset $0 \in \scB \subseteq \widebar\scB$ which we may shrink if necessary. Furthermore, we write $B = \scB \otimes_\kq \kk$, the base change along $\Spec \kk \to \Spec \kq$. We write $\scX = f^{-1}(\scB)$ and retain the symbol $f$ to denote the induced morphism $f\colon \scX \to \scB$. We write $X = \scX \otimes_\kq \kk$ and $f\colon X \to B$ for the base extensions along $\Spec \kk \to \Spec \kq$. If $f\colon X \to B$ is pre--gls, then it is an instance of Situation~\ref{sitn:explicit-gls-candidate-kk-intro}. \qeddef
\end{sitn}

In this situation, we have the following sufficient criterion for $f\colon X \to B$ being pre--gls.

\begin{lem}\label{lem:pre-gls-criterion}\note{lem:pre-gls-criterion}
	Assume that we are in Situation~\ref{sitn:affine-gls-candidate}, and assume the following conditions:
	\begin{enumerate}[label=\textnormal{(\alph*)}]
		\item the ring $R_{\widebar\scX}$ is integral, Cohen--Macaulay, and of dimension $d + 1$;
		\item the ring homomorphism $f^\sharp\colon \kq[t] \to R_{\widebar\scX}$ is injective;
		\item we have $\m{dim}(\m{Sing}(\widebar\scX/\widebar\scB;d)) \leq d - 1$.
	\end{enumerate}
	Then, $f\colon \widebar\scX \to \widebar\scB$ is flat, and its fibers are Cohen--Macaulay, pure of dimension $d$, and geometrically reduced. When taking $0 \in \scB \subseteq \widebar\scB$ small enough, then $\scX_b$ is normal for all $0 \not= b \in \scB$. In particular, $f\colon X \to B$ is pre--gls.
\end{lem}
\begin{defn}
	In Situation~\ref{sitn:affine-gls-candidate}, we say that $f\colon \widebar\scX \to \widebar\scB$ is \emph{algorithmically pre--gls} if the conditions of the lemma are satisfied. \qeddef
\end{defn}

To check the conditions in Theorem~\ref{thm:explicit-gls-criterion-kk} algorithmically, we have to check them over $\kq$ rather than $\kk$. We have an analog $\cU_1^\times\scV \subseteq \scV$ of $\cU_1^\times V$, and in Lemma~\ref{lem:U1x-base-change}, we will see that we have $\cU_1^\times V = (\cU_1^\times\scV) \otimes_\kq \kk$. Then, we also have $\cS_1^\times V = (\cS_1^\times\scV) \otimes_\kq \kk$ because reducedness is preserved under field extensions in characteristic $0$. Here is how we determine $\cC_2^\times\scV$ in practice.

\begin{lem}\label{lem:compute-C2}\note{lem:compute-C2}
	Let $\scV = \Spec \kq[x_1,\ldots,x_n]/(f_1,\ldots,f_r)$. Assume that $\scV$ is reduced, and that all irreducible components have the same dimension $d \geq 1$. Let $\scD = \m{Sing}(\scV;d) \subset \scV$ be the reduced closed subset cut out by the radical of the $d$-th Fitting ideal of $\Omega^1_{\scV/\kq}$, and assume that each irreducible component of $\scD$ has the same dimension $d - 1$. Let $\scT \subset \scD$ be the closed subset cut out by the radical of the $(d - 1)$-th Fitting ideal of $\Omega^1_{\scD/\kq}$. Then, we have $\scT \subset \cC_2^\times\scV$. 
	
	Let $\nu\colon \widetilde\scV \to \scV$ be the normalization, let $\scC \subset \scV$ be the conductor locus, i.e., the closed subset cut out by the annihilator of $\m{coker}(\cO_{\scV} \to \nu_*\cO_{\widetilde\scV})$, and let $\mu\colon \widetilde \scC \to \scC$ be the base change of $\nu$ along $\scC \to \scV$. Assume the following conditions:
	\begin{enum:alph}
		\item $\widetilde\scV \setminus \nu^{-1}(\scT)$ is smooth;
		\item $\scC \setminus \scT$ is smooth and pure of dimension $d - 1$;
		\item $\nu\colon \widetilde\scV \setminus \nu^{-1}(\scT) \to \scV \setminus \scT$ is unramified;
		\item $\mu_*\cO_{\widetilde\scC}$ is a locally free $\cO_{\scC}$-module of rank $2$ on $\scC \setminus \scT$.
	\end{enum:alph}
	Then, we have $\scT = \cC_2^\times\scV$.
\end{lem}

\begin{defn}
	In Situation~\ref{sitn:affine-gls-candidate}, we say that $\scV$ is \emph{algorithmically generically dnc} if the conditions of the lemma are satisfied, and if moreover $\m{dim}(\scT) \leq d - 2$. \qeddef
\end{defn}

We use the notations $\varrho$, $\xi_k$, $\sigma_{\ell - 1}$, $\scK^{(k)} \subset \scV$, and $\scZ^{(\ell)} \subset \scV$ analogous to the ones introduced in Notation~\ref{nota:gls-criterion-preparation}. We have $K^{(k)} = \scK^{(k)} \otimes_\kq \kk$ and $Z^{(\ell)} = \scZ^{(\ell)} \otimes_\kq \kk$. 

\begin{lem}\label{lem:gls-criterion-bar}\note{lem:gls-criterion-bar}
	In Situation~\ref{sitn:affine-gls-candidate}, assume that the conditions in Lemma~\ref{lem:pre-gls-criterion} are satisfied, and let $\scB \subseteq \widebar\scB$ be as in that lemma. Assume that the conditions in Lemma~\ref{lem:compute-C2} are satisfied so that $\scT = \cC_2^\times\scV$. Assume furthermore:
	\begin{enum:alph}
		\item we have $\m{dim}(\scT) \leq d - 2$;
		\item\label{item:ell-exists-bar} for every $\scE \in [\cS_1^\times \scV]$, there is some $\ell = \ell(\scE) \geq 1$ such that $\scE \cap \scK^{(k)} = \varnothing$ for all $0 \leq k \leq \ell - 2$, and such that $\m{dim}(\scE \cap \scZ^{(\ell)}) \leq d - 2$.
	\end{enum:alph}
	Then, the conditions in Theorem~\ref{thm:explicit-gls-criterion-kk} are satisfied for $f\colon X \to B$.
\end{lem}
\begin{proof}
	For the first condition, use \cite[Cor.~4.2.8]{EGAIV-2} to deduce the dimension of $\cC_2^\times V$ from the dimension of $\cC_2^\times\scV$. Apply similar reasoning to the third condition. Note that $\scE \otimes_\kq \kk$ is not necessarily irreducible or connected, but it decomposes into components of $\cS_1^\times V$.
\end{proof}

\begin{defn}
	In Situation~\ref{sitn:affine-gls-candidate}, we say that $f\colon \widebar\scX \to \widebar\scB$ is \emph{algorithmically gls} if it is algorithmically pre--gls, if $\scV$ is algorithmically generically dnc, and if Condition~\ref{item:ell-exists-bar} in Lemma~\ref{lem:gls-criterion-bar} is satisfied. \qeddef
\end{defn}

\begin{figure}\label{fig:112}
	\begin{mdframed}
		\begin{center}
			\includegraphics[scale=.88]{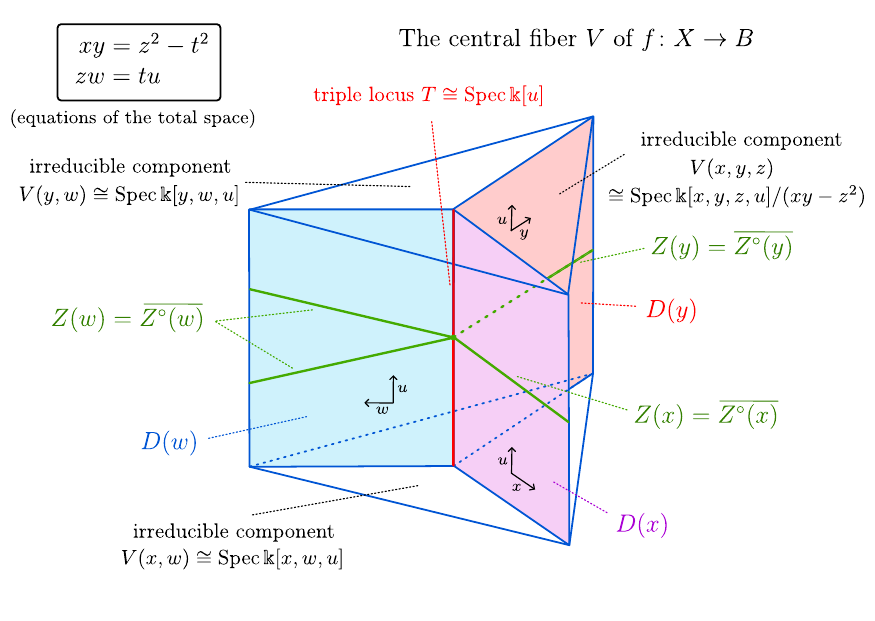}
		\end{center}
		\caption{The central fiber in Example~\ref{ex:112}}
	\end{mdframed}
\end{figure}

\begin{ex}\label{ex:112}\note{ex:112}
	We illustrate step by step how to check the conditions. The code for the computations in Macaulay2 is contained in the file \texttt{example-affine-gls-family-112.m2}. The numbers in the file name refer to the kinks as defined in Theorem~\ref{thm:explicit-gls-criterion-kk}.
	
	We set $\kq = \QQ$ and $\kk = \CC$. Let $R_{\widebar\scB} = \QQ[t]$ and $s = t$. We consider 
	\[f^\sharp\colon \enspace R_{\widebar\scB} = \QQ[t] \to \QQ[t,x,y,z,w,u]/(xy - z^2 + t^2, \: zw - tu) = R_{\widebar\scX}.\]
	The computation in M2 shows that $R_{\widebar\scX}$ is an integral domain of dimension $4$ and Cohen--Macaulay. The map $f^\sharp\colon R_{\widebar\scB} \to R_{\widebar\scX}$ is injective. We have  $\m{dim}(\m{Sing}(\widebar \scX/\widebar \scB;d)) = 2$, so Lemma~\ref{lem:pre-gls-criterion} shows that $f\colon X \to B$ is pre--gls with normal fibers outside $b = 0$ for a suitable $\scB \subseteq \widebar\scB$. Moreover, we have $\m{Sing}(\widebar\scX/\widebar\scB;d) \subset \scV$ so that we can take $\scB = \widebar\scB = \IA^1_t$. From now on, we write $f\colon \scX \to \scB$ for the family.
	
	The central fiber $\scV$ has three irreducible components given by
	\begin{align*}
		\scV(x,y,z) &= \{t = w = xy - z^2 = 0\} & &\cong \Spec \QQ[x,y,z,u]/(xy - z^2), \\
		\scV(x,w) &= \{t = y = z = 0\} & &\cong \Spec \QQ[x,w,u], \\
		\scV(y,w) &= \{t = x = z = 0\} &  &\cong \Spec \QQ[y,w,u].
	\end{align*}
	They all have dimension $3$ and are normal. The latter two are smooth whereas the (reduced) singular locus of $\scV(x,y,z)$ is $\{t = x = y = z = w = 0\}$. 
	
	The reduced singular locus of $\scV$ is given by $\scD = \{t = z = xy = xw = yw = 0\}$ and has three irreducible components
	\begin{align*}
		\scD(w) &= \{z = y = x = t = 0\} \cong \Spec \QQ[w,u], \\
		\scD(y) &= \{w = z = x = t = 0\} \cong \Spec \QQ[y,u], \\
		\scD(x) &= \{w = z = y = t = 0\} \cong \Spec \QQ[x,u].
	\end{align*}
	In particular, $\scD$ is pure of dimension $2$ as required in Lemma~\ref{lem:compute-C2}. Geometrically, we have $\scD(w) = \scV(x,w) \cap \scV(y,w)$, $\scD(y) = \scV(y,w) \cap \scV(x,y,z)$, and $\scD(x) = \scV(x,w) \cap \scV(x,y,z)$.
	
	The reduced singular locus of $\scD$ is given by $\scT = \{t = x = y = z = w = 0\} \cong \Spec \QQ[u]$. This coincides with the locus where all three irreducible components of $\scV$ meet. Moreover, this coincides with the singular locus of $\scV(x,y,z)$.
	
	We verify that $\scT = \cC_2^\times\scV$ by applying Lemma~\ref{lem:compute-C2}. The normalization $\widetilde\scV$ of $\scV$ is the disjoint union of the three irreducible components of $\scV$ because these components are normal. In particular, $\widetilde\scV$ is smooth outside $\nu^{-1}(\scT)$.
	
	The homomorphism $\nu^\sharp\colon \cO_{\scV} \to \nu_*\cO_{\widetilde\scV}$ is given as the diagonal 
	\[\cO_{\scV} \to \cO_{\scV(x,y,z)} \oplus \cO_{\scV(x,w)} \oplus \cO_{\scV(y,w)}.\]
	We check in the script that the annihilator of $\m{coker}(\nu^\sharp)$ cuts out $\scD$ so that $\scC = \scD$. In particular, $\scC \setminus \scT$ is smooth of dimension $d - 1 = 2$.
	
	The normalization map $\nu\colon \widetilde\scV \to \scV$ is clearly unramified. Finally, we have 
	\[\cM \coloneq \mu_*\cO_{\widetilde\scC} = \cO_{\scD \cap \scV(x,y,z)} \oplus \cO_{\scD \cap \scV(x,w)} \oplus \cO_{\scD \cap \scV(y,w)}.\]
	We check in the script that this is locally free of rank $2$ on $\scC \setminus \scT$. Concretely, we show that $\m{Fitt}_1(\cM)|_{\scC \setminus \scT} = \cO_{\scC \setminus \scT}$ and $\m{Fitt}_2(\cM)|_{\scC \setminus \scT} = 0$. This shows $\scT = \cC_2^\times\scV$. In particular, $\m{dim}(\cC_2^\times\scV) = 1 \leq d - 2$ in the conditions of Lemma~\ref{lem:gls-criterion-bar}.

	The three irreducible components of $\cS_1^\times\scV = \scD \setminus \scT$ are $\scD^\circ(v) \coloneq \scD(v) \setminus \scT$ for $v \in \{x,y,w\}$.
	
	In the script, we compute that 
	\[\scZ^{(1)} = \{t = x = y = z = 0\} \cup \{t = x = z = w = u = 0\} \cup \{t = y = z = w = u = 0\}.\]
	This shows that $\m{dim}(\scD^\circ(x) \cap \scZ^{(1)}) = \m{dim}(\scD^\circ(y) \cap \scZ^{(1)}) = 1$, and that $\scD^\circ(w)$ is contained in $\scZ^{(1)}$. Thus, Condition~\ref{item:ell-exists-bar} in Lemma~\ref{lem:gls-criterion-bar} is satisfied for both $\scD^\circ(x)$ and $\scD^\circ(y)$ with $\ell = 1$.
	
	Next, we compute that $\scK^{(0)} = \{t = z = w = xy = 0\}$. Thus, $\scK^{(0)} \cap \scD^\circ(w) = \varnothing$. Finally, we have 
	\[\scZ^{(2)} = \{t = x = y = z = w^2 - u^2 = 0\}\]
	so that $\m{dim}(\scD^\circ(w) \cap \scZ^{(2)}) = 1$. Therefore, Condition~\ref{item:ell-exists-bar} in Lemma~\ref{lem:gls-criterion-bar} is satisfied with $\ell = 2$.
	
	After base change to $\kk = \CC$, we compute the log singular locus $Z_X$ given by Formula~\eqref{eqn:clubsuit}. From the computations, we see that 
	\begin{align*}
		Z^\circ(x) \coloneq Z_{D^\circ(x)} &= \{t = y = z = w = 0, x \not= 0, u = 0\}, \\
		Z^\circ(y) \coloneq Z_{D^\circ(y)} &= \{t = x = z = w = 0, y \not= 0, u = 0\}, \\
		Z^\circ(w) \coloneq Z_{D^\circ(w)} &= \{t = x = y = z = 0, w \not= 0, u^2 = w^2\}.
	\end{align*}
	Furthermore, we have to add $\cC_2^\times V = T \coloneq \scT \otimes_\QQ \CC$ to the log singular locus. While Theorem~\ref{thm:explicit-gls-criterion-kk} guarantees that $f\colon (X|V) \to (B|0)$ is really not both log smooth and saturated at any point 
	\[p \in Z_V^\circ = Z^\circ(x) \cup Z^\circ(y) \cup Z^\circ(w),\]
	it makes no statement about whether the points $p \in T$ are log smooth and saturated or not. \qedloz
\end{ex}

\begin{rem}\label{rem:precise-log-smooth-locus-II}\note{rem:precise-log-smooth-locus-II}
	In the preceding Example~\ref{ex:112}, we did not yet determine which points inside $T$ are log smooth and saturated. Namely, as pointed out in Remark~\ref{rem:precise-log-smooth-locus-I}, Theorem~\ref{thm:explicit-gls-criterion-kk} does not provide any information about this.
	
	Let $\{0\} = \{t = x = y = z = w = u = 0\} \subset T$. We can see that $f\colon (X|V) \to (B|0)$ is log smooth and saturated as follows. We have an isomorphism
	\begin{multline*}
		\CC[t,x,y,z,w,u]_{u(u^2 - w^2)}/(xy - z^2 + t^2,zw - tu) \\ \to \enspace \CC[t,\tilde x,\tilde y,\tilde z,\tilde w,\tilde u]_{\tilde u(1 - \tilde w^2)}/(\tilde x\tilde y - \tilde z^2,\tilde z \tilde w - t) \\
		\psi(t) = t, \quad \psi(x) = \tilde x, \quad \psi(y) = \tilde y(1 - \tilde w^2), \quad \psi(z) = \tilde z, \quad \psi(w) = \tilde w\tilde u, \quad \psi(u) = \tilde u.
	\end{multline*}
	Since $T \setminus \{0\}$ is contained in $\{u \not= 0, u^2 \not= w^2\} \subset V$, it would be sufficient to have that the family given by $\tilde x\tilde y= \tilde z^2$ and $\tilde z \tilde w = t$ is log smooth and saturated. This is indeed the case, see \cite[Ex.~1.66]{FeltenGLDT}.
	
	The point $0 \in T$ cannot be log smooth and saturated because it is in the closure $\overline{Z_V^\circ}$ of $Z_V^\circ$. Thus, $\overline{Z_V^\circ}$ is the precise log smooth locus in Example~\ref{ex:112}. \qedloz
\end{rem}

\subsection{The criterion in explicit projective families}\label{sect:criterion-explicit-projective}\note{sect:criterion-explicit-projective}

We explain how to use Theorem~\ref{thm:explicit-gls-criterion-kk} to show that a projective family $f\colon X \to \IA^1_t$ given by generators and relations is gls. We incorporate the conditions of Lemma~\ref{lem:pre-gls-criterion} directly into our situation.

\begin{sitn}\label{sitn:projective-setup}\note{sitn:projective-setup}
	Let $\kq$ be a field of characteristic $0$ which is suitable for algorithmic computations, and let $\kq \to \kk$ be a field extension such that $\kk$ is algebraically closed.
	
	We set $R_{\widebar\scB} = \kq[t]$ and $\widebar\scB = \Spec \kq[t]$. Fix $n \geq 1$, and let $R^\bullet_\PP = \kq[t,x_0,\ldots,x_n]$ be the graded ring with $\m{deg}(x_i) = 1$ and $\m{deg}(t) = 0$.
	
	Let $r \geq 0$, and let $f_1,\ldots,f_r \in R^\bullet_\PP$
	be homogeneous polynomials. We have a graded quotient ring
	\[R_{\widebar\scX}^\bullet = \kq[t,x_0,\ldots,x_n]/(f_1,\ldots,f_r)\]
	and a closed subscheme $\widebar\scX = \Proj R_{\widebar\scX}^\bullet \subset \PP^n \times \widebar \scB$. Let $f\colon \widebar\scX \to\widebar\scB$ be the composition of this embedding with the projection to $\widebar\scB$; it is proper. We denote the central fiber by $\scV = f^{-1}(0)$.
	
	Let $\widebar\scX_\vartriangle = \Spec R_{\widebar\scX}^\bullet \subset \IA^{n + 1} \times \widebar\scB$ be the affine cone over $\widebar\scX$, and let $f_\vartriangle\colon \widebar\scX_\vartriangle \to \widebar\scB$ be the composition with the second projection.
	
	Fix an intended relative dimension $d \geq 1$. We assume the following conditions:
	\begin{enum:alph}
		\item\label{item:R-bullet-prop} the ring $R^\bullet_{\widebar\scX}$ is integral, Cohen--Macaulay, and of dimension $d + 2$;
		\item\label{item:bullet-injective} the ring homomorphism $R_{\widebar\scB} \to R^\bullet_{\widebar\scX}$ is injective;
		\item\label{item:Delta-singular} we have $\m{dim}(\m{Sing}(\widebar\scX_\vartriangle/\widebar\scB;d + 1)) \leq d$.
	\end{enum:alph}
	Then, Lemma~\ref{lem:pre-gls-criterion} shows that $f_\vartriangle\colon \widebar\scX_\vartriangle \to \widebar\scB$ is flat, and that its fibers are Cohen--Macaulay, pure of dimension $d + 1$, and geometrically reduced. For $0 \in \scB \subseteq \widebar\scB$ small enough, the fiber $\widebar\scX_{\vartriangle,b} = f_\vartriangle^{-1}(b)$ is normal for every $0 \not= b \in \scB$.
	
	Assumption~\ref{item:bullet-injective} implies that $R_{\widebar\scB} \to R_{\widebar\scX}^0$ is injective so that it is an isomorphism because $R_\PP^0 = R_{\widebar\scB}$. Therefore, the augmentation homomorphism $R_{\widebar\scX}^\bullet \to R_{\widebar\scX}^0 = R_{\widebar\scB}$ defines a closed immersion $\eps\colon \widebar\scB \to \widebar\scX_\vartriangle$. We set 
	\[\widebar\scX_\pluscirc = \widebar\scX_\vartriangle \setminus \eps(\widebar\scB) = \widebar\scX_\vartriangle \setminus (\{0\} \times \widebar\scB),\]
	the pointed affine cone, and we write $\iota\colon \widebar\scX_\pluscirc \to \widebar\scX_\vartriangle$ for the inclusion. By \cite[Cor.~8.3.6]{EGAII}, we have an affine surjective morphism $\pi\colon \widebar\scX_\pluscirc \to \widebar\scX$. For $h \in R_{\widebar\scX}^1$, the morphism $\pi^{-1}(D^+(h)) \to D^+(h)$ is given by $(R_{\widebar\scX}^\bullet)_{(h)} \to (R_{\widebar\scX}^\bullet)_{(h)} \otimes_\kq \kq[z^\pm]$ so that $\pi\colon \widebar\scX_\pluscirc \to \widebar \scX$ is smooth with fibers $\GG_m = \Spec \kq[z^\pm]$ because $R_{\widebar\scX}^\bullet$ is generated by $R_{\widebar\scX}^1$ as an $R_{\widebar\scX}^0$-algebra.
	
	Since $f \circ \pi\colon \widebar\scX_\pluscirc \to \widebar\scB$ is flat, also $f\colon \widebar\scX \to \widebar\scB$ is flat because $\pi$ is smooth and surjective. Then, $f\colon \widebar\scX \to \widebar\scB$ is dominant, and since $f$ is closed, dominance implies that $f\colon \widebar\scX \to \widebar\scB$ is surjective. 
	
	Since $\pi\colon \widebar\scX_\pluscirc \to \widebar\scX$ is smooth and surjective, \cite[Tag~0C0W]{stacks} shows that the fibers of $f\colon\widebar\scX \to \widebar\scB$ are Cohen--Macaulay. By similar results, they are geometrically reduced and pure of dimension $d$ (since the relative dimension of $\pi$ is one). Finally, for $0 \not= b \in \scB \subseteq \widebar\scB$, the fiber $\widebar\scX_b$ is normal.
	
	We write $B = \scB \otimes_\kq \kk$ as well as $f\colon X \to B$, $f_\vartriangle\colon X_\vartriangle \to B$, and $\pi\colon X_\pluscirc \to X$ for the base change along $B \to \scB \to \widebar\scB$. \qeddef
\end{sitn}

The key to determining whether $f\colon (X|V) \to (B|0)$ can be made gls is now the following observation.

\begin{prop}\label{prop:projective-gls-criterion}\note{prop:projective-gls-criterion}
	Assume that we are in Situation~\ref{sitn:projective-setup}. Then, $f\colon X \to B$ satisfies the conditions in Theorem~\ref{thm:explicit-gls-criterion-kk} if and only if $f_\vartriangle\colon X_\vartriangle \to B$ satisfies them. In this case, if $Z_X \subset X$ and $Z_{X,\vartriangle} \subset X_\vartriangle$ are defined by Formula~\eqref{eqn:clubsuit}, we have $Z_{X,\vartriangle} \cap X_\pluscirc = \pi^{-1}(Z_X)$.
\end{prop}
\begin{cor}
	Assume that we are in Situation~\ref{sitn:projective-setup}. If $f_\vartriangle\colon \widebar\scX_\vartriangle \to \widebar\scB$ is algorithmically gls, then $f: X \to B$ is generically log smooth.
\end{cor}

The proof of Proposition~\ref{prop:projective-gls-criterion} can be found on Page~\pageref{pf:projective-gls-criterion}. Now, we may use Lemma~\ref{lem:gls-criterion-bar} to check these conditions in explicit examples.

\begin{figure}[t]\label{fig:projective-112}
	\begin{mdframed}
		\begin{center}
			\includegraphics[scale=.88]{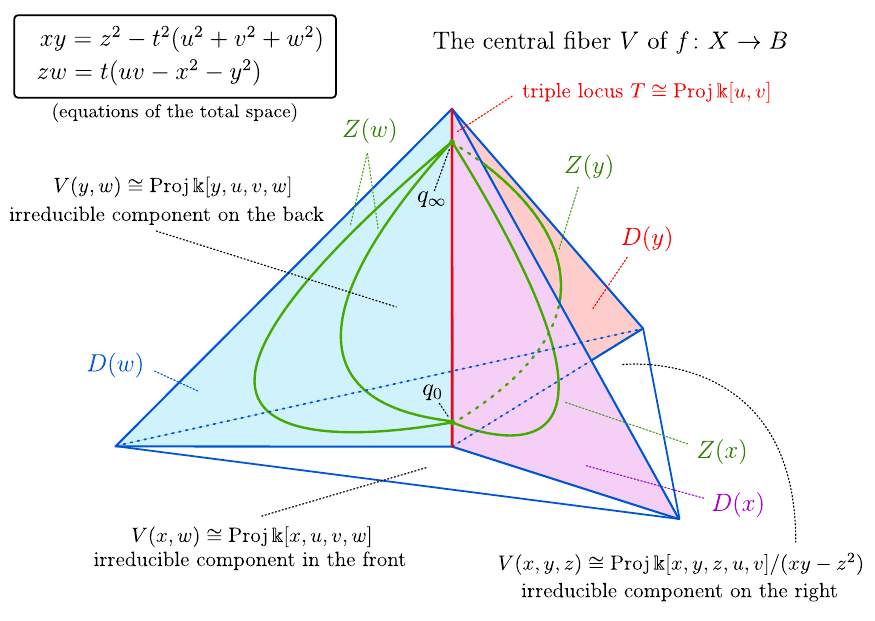}
		\end{center}
		\caption{The central fiber in Example~\ref{ex:projective-112}}
	\end{mdframed}
\end{figure}

\begin{ex}\label{ex:projective-112}\note{ex:projective-112}
We discuss the following example of Situation~\ref{sitn:projective-setup} and Proposition~\ref{prop:projective-gls-criterion}. The code can be found in the file \texttt{example-projective-gls-family-112.m2}. Let $\kq = \QQ$ and $\kk = \CC$. We set $R_{\widebar\scB} = \QQ[t]$ and consider the map 
\[R_{\widebar\scB} = \QQ[t] \to \QQ[t,x,y,z,u,v,w]/(f_1,f_2) = R^\bullet_{\widebar\scX}\]
with 
\[f_1 = xy - z^2 + t^2(u^2 + v^2 + w^2) \quad \text{and} \quad f_2 = zw - t(uv - x^2 - y^2).\]
The intended relative dimension is $d = 3$. Up to a minor modification by the author, these equations have been suggested by Alessio Corti as a projective variant of Example~\ref{ex:112}.\footnote{Note however that we do not prove here that Example~\ref{ex:112} and the present example are \'etale locally isomorphic.}

The computation yields that $\widebar\scX_\vartriangle$ is integral, of dimension $5$, and Cohen--Macaulay. The homomorphism $f^\sharp\colon R_{\widebar\scB} \to R^\bullet_{\widebar\scX}$ is injective, and the dimension of $\m{Sing}(\widebar\scX_\vartriangle/\widebar\scB; d + 1)$ is $3$. Thus, we are in Situation~\ref{sitn:projective-setup}.

When $\scA \subset \widebar\scX_\vartriangle$ is a closed subset defined by homogeneous equations, then we have $\scA \cap f_\vartriangle^{-1}(b) \not= \varnothing$ if and only if $\eps(b) \in \scA$. We apply this observation to the irreducible components of $\m{Sing}(\widebar\scX_\vartriangle/\widebar\scB;d + 1)$ and find that $\m{Sing}(\widebar\scX_\vartriangle/\widebar\scB;d + 1) \subset \eps(\widebar\scB) \cup f_\vartriangle^{-1}(\m{Disc}(f))$ for 
\[\m{Disc}(f) \coloneq \{t(2t + 1)(2t - 1)(4t^2 + 1)(16t^4 + 1) = 0\} \subset \widebar\scB.\]
Thus, the fibers of $f\colon \widebar\scX \to \widebar\scB$ are smooth over $\scB^\star = \widebar\scB \setminus \m{Disc}(f)$. We set $\scB = \scB^\star \cup \{0\} \subseteq \widebar\scB$ and write $f\colon \scX \to \scB$ for the restricted map.

The central fiber $\scV_{\vartriangle}$ is reduced and has three irreducible components
\begin{align*}
	\scV_{\vartriangle}(x,y,z) &= \{t = w = xy - z^2 = 0\} & &\cong \Spec \QQ[x,y,z,u,v]/(xy - z^2), \\
	\scV_{\vartriangle}(x,w) &= \{t = y = z = 0\} & &\cong \Spec \QQ[x,w,u,v], \\
	\scV_{\vartriangle}(y,w) &= \{t = x = z = 0\} &  &\cong \Spec \QQ[y,w,u,v],
\end{align*}
all of dimension $4$. Two of them are smooth while $\scV_{\vartriangle}(x,y,z)$ has a two-dimensional singular locus given by $\scT'_{\vartriangle} \coloneq \{x = y = z = w = 0\} \cong \IA^2_{u,v}$.

The reduced singular locus $\scD_{\vartriangle} = \cC_1^\times \scV_{\vartriangle}$ of $\scV_{\vartriangle}$ has three irreducible components
\begin{gather*}
	\scD_{\vartriangle}(x) = \{t = y = z = w = 0\}, \quad \scD_{\vartriangle}(y) = \{t = x = z = w = 0\}, \\ \scD_{\vartriangle}(w) = \{t = x = y = z = 0\}.
\end{gather*}
Each of them is isomorphic to $\IA^3$, and each of them is the intersection of two irreducible components of $\scV_{\vartriangle}$.

The locus $\scT_{\vartriangle} \subset \scD_{\vartriangle}$ cut out by the radical of the $(d - 1)$-th Fitting ideal of $\Omega^1_{\scD_{\vartriangle}/\kq}$, as considered in Notation~\ref{nota:times-stratification} and Lemma~\ref{lem:compute-C2}, is given by $\scT_{\vartriangle} = \{t = x = y = z =  w = 0\}$. We see that $\scT_{\vartriangle} = \scT'_{\vartriangle}$. This is also the intersection of all three irreducible components of $\scV_{\vartriangle}$.

With the same procedure as in Example~\ref{ex:112}, we use Lemma~\ref{lem:compute-C2} to check that we have indeed $\scT_{\vartriangle} = \cC_2^\times \scV_{\vartriangle}$. Now, we have $\m{dim}(\cC_2^\times\scV_{\vartriangle}) = \m{dim}(\scT_{\vartriangle}) = 2 \leq (d + 1) - 2$, as required in Lemma~\ref{lem:gls-criterion-bar}.

For $v \in \{x,y,w\}$, we set $\scD_{\vartriangle}^\circ(v) = \scD_{\vartriangle}(v) \setminus \scT_{\vartriangle}$. These are the three irreducible components of $\cS_1^\times\scV_{\vartriangle}$.

In the script, we compute that $\scZ_{\vartriangle}^{(1)}$ is the union of $\{t = x = y = z = 0\}$, $\{t = y = z = w = x^2 - uv = 0\}$, and $\{t = x = z = w = y^2 - uv = 0\}$. Thus, $\scD_{\vartriangle}^\circ(w)$ is contained in $\scZ_{\vartriangle}^{(1)}$ while 
\[\m{dim}(\scD_{\vartriangle}^\circ(x) \cap \scZ_{\vartriangle}^{(1)}) = \m{dim}(\scD_{\vartriangle}^\circ(y) \cap \scZ_{\vartriangle}^{(1)}) = 2 \leq (d + 1) - 2.\]
In particular, the kink of the latter two components of $\scD_{\vartriangle} \setminus \scT_{\vartriangle}$ is $\ell = 1$.

We compute that $\scK_{\vartriangle}^{(0)} = \{t = z = w = xy = 0\}$. Thus, $\scK_{\vartriangle}^{(0)} \cap \scD_{\vartriangle}^\circ(w) = \varnothing$. Finally, we compute that 
\[\scZ_{\vartriangle}^{(2)} = \scD_{\vartriangle}(w) \cap \{u^2v^2 - u^2w^2 - v^2w^2 - w^4 = 0\}.\]
In particular, the dimension drops, and the kink is $\ell = 2$. All conditions in Lemma~\ref{lem:gls-criterion-bar} are satisfied, and $f_\vartriangle\colon X_\vartriangle \to B$ becomes gls with
\[Z_{X,\vartriangle} =  T_{\vartriangle} \cup Z_{\vartriangle}(x) \cup Z_{\vartriangle}(y) \cup Z_{\vartriangle}(w),\]
where we set
\begin{align*}
	Z_{\vartriangle}(x) &= D_{\vartriangle}(x) \cap \{x^2 - uv = 0\}, \qquad \qquad Z_{\vartriangle}(y) = D_{\vartriangle}(y) \cap \{ y^2 - uv = 0\}, \\
	Z_{\vartriangle}(w) &= D_{\vartriangle}(w) \cap \{u^2v^2 - u^2w^2 - v^2w^2 - w^4 = 0\}.
\end{align*}
Note that $\m{Sing}(X_\vartriangle^\star/B^\star;d + 1) = \varnothing$.·

By Proposition~\ref{prop:projective-gls-criterion}, also $f\colon X \to B$ is gls, and we have $\pi^{-1}(Z_X) = Z_{X,\vartriangle} \cap X_\pluscirc$. On $V$, the branches $Z(x)$ and $Z(y)$ of $Z_X$ are copies of $\PP^1$ which intersect $T \cong \PP^1_{u:v}$ in two points $q_0 = [0:0:0:0:1:0]$ and $q_\infty = [0:0:0:1:0:0]$. The branch $Z(w)$ is a curve of degree $4$ in $D(w) \cong \PP^2_{u:v:w}$; it meets $T$ precisely in $q_0$ and $q_\infty$, and these are furthermore the singular points of $Z(w)$. While $Z(w)$ looks similar to two copies of $\PP^1$ in a picture over $\RR$,\footnote{The reader might want to plot this with a tool like \texttt{Surfer}.} it is an integral curve with nodal singularities in $q_0$ and $q_\infty$ (this is discussed in the script). \qedex
\end{ex}

\subsection{The infinitesimal criterion for generic log smoothness}

In Situation~\ref{sitn:explicit-gls-candidate-kk-intro}, suppose that we have obtained $f\colon X \to B$ by approximating an explicitly known formal deformation of $V$. Then, we do not have access to $f\colon X^\star \to B^\star$. In particular, we cannot directly check that its fibers are reduced or pure-dimensional (which is required to be pre--gls) or normal (as required in Theorem~\ref{thm:explicit-gls-criterion-kk}). In fact, in this situation, we can only ask for $f\colon (X,X \setminus Z_X|V) \to (B|0)$ to be gls in a neighborhood of $V$ in $X$ because what happens outside a neighborhood of $V$ is not controlled by the formal neighborhood of $V$ in $X$. 

Also $\Omega^1_X$ and the homomorphism $\varrho\colon \Omega^1_X \to \Omega^1_{X/B}$ are not accessible in this situation. However, for every $n \geq 0$, we can form $\Omega^1_{X_n}$ and $\varrho_n\colon \Omega^1_{X_n} \to \Omega^1_{X_n/B_n}$, and in Lemma~\ref{lem:eta-f-alternative}, we will see that $\Omega^1_X|_{X_k}$ can be identified with $\Omega^1_{X_n}|_{X_k}$, and that $\varrho|_{X_k}$ can be identified with $\varrho_n|_{X_k}$ for $k \leq n - 1$.

\begin{prop}\label{prop:infinitesimal-gls-criterion}\note{prop:infinitesimal-gls-criterion}
	Assume that we are in Situation~\ref{sitn:explicit-gls-candidate-kk-intro}, and that the following conditions hold:
	\begin{enum:alph}
		\item $V = f^{-1}(0)$ is Cohen--Macaulay;
		\item $\m{dim}(\cC_2^\times V) \leq d - 2$;
		\item Condition~\ref{item:ell-exists} in Theorem~\ref{thm:explicit-gls-criterion-kk} holds after replacing $\Omega^1_X|_{X_k}$ with $\Omega^1_{X_{k + 1}}|_{X_k}$ and $\varrho|_{X_k}$ with $\varrho_{k + 1}|_{X_k}$.
	\end{enum:alph}
	Then, there is an open subset $X' \subseteq X$ with $V \subset X'$ such that $f'\colon X' \to B$ is pre--gls and Cohen--Macaulay, and such that the conditions in Theorem~\ref{thm:explicit-gls-criterion-kk} are satisfied. In particular, when we choose $Z_{X'} \subset X'$ according to Formula~\eqref{eqn:clubsuit} in Theorem~\ref{thm:explicit-gls-criterion-kk}, then $f'\colon (X',X' \setminus Z_{X'}|V) \to (B|0)$ is generically log smooth.
\end{prop}

\section{Toroidal crossing schemes: main results}\label{sect:tc-main}\note{sect:tc-main}

We study the situation where the central fiber $V = X_0$ of $f\colon X \to B$ carries the structure of a toroidal crossing scheme $(V,\cP,\bar\rho)$. We give a criterion to show that $f\colon X \to B$ induces a section $\lambda$ of $\cL\cS_V$ and a means to compare the induced sections for two families $f\colon X \to B$ and $g\colon Y \to B$.

\subsection{Toroidal crossing schemes}

Let $\kk$ be an algebraically closed field of characteristic $0$. A \emph{toroidal crossing scheme} is a triple $(V,\cP,\bar\rho)$, where:
\begin{enum:alph}
	\item $V/\kk$ is a separated scheme of finite type, pure of some dimension $d$;
	\item $\cP$ is a sheaf of monoids on $V$, and $\bar\rho \in \Gamma(V,\cP)$ is a global section. 
\end{enum:alph}
These data are subject to the condition that $(\cP,\bar\rho)$ arises locally as the ghost sheaf of a log smooth, saturated, and vertical morphism of log schemes $f_0\colon (V,\cM_0) \to S_0$. Here, $S_0 = \Spec(\NN \to \kk)$ is the standard log point. That $(\cP,\bar\rho)$ is the ghost sheaf of the morphism means that $\cP$ is the ghost sheaf of $\cM_0$, and that $\bar\rho$ is the ``ghost'' of $f_0^\flat(\tau_0)$ for $\tau_0 = (1,1) \in \NN \oplus \kk^\times = \Gamma(S_0,\cM_{S_0})$. These conditions imply that $V$ is reduced and Cohen--Macaulay. For details about toroidal crossing schemes, we refer the reader to \cite{SchroerSiebert2006,GrossSiebertI,FFR2021} and in particular to \cite[\S9]{FeltenGLDT}, which contains the material in the form in which we use it.

A toroidal crossing scheme is intermediate between a scheme and a morphism of log schemes to the standard log point. For an \'etale open subset $W \to V$, we can consider quadruples $(\cM_0,\alpha,q,\rho)$, where $\alpha\colon \cM_0 \to \cO_W$ is a log structure, $q\colon \cM_0 \to \cP$ identifies $\cP$ with the ghost sheaf of $\cM_0$, and $\rho \in \Gamma(W,\cM_0)$ is a section with $q(\rho) = \bar\rho$ and $\alpha(\rho) = 0$. Such a datum gives rise to a morphism of log schemes $f_0\colon (W,\cM_0) \to S_0$ by means of $f_0^\flat(\tau_0) = \rho$. 

Among these data, we restrict our attention to those for which $f_0\colon (W,\cM_0) \to S_0$ is log smooth; it is then automatically saturated and vertical. Then, $f_0$ has no non-trivial automorphisms, and isomorphism classes of such quadruples form a sheaf $\cL\cS_V$ on the nose, without sheafifying.

A toroidal crossing scheme is stratified according to the rank of $\cP_{\bar v}$:
\[\cS_kV \coloneq \{v \in V \ | \ \m{rk}(\cP_{\bar v}) = k + 1\}.\]
Then, $\cU_kV = \bigsqcup_{i \leq k}\cS_iV \subseteq V$ is open, and $\cC_kV = \bigsqcup_{i \geq k}\cS_iV \subset V$ is closed. Consequently, $\cS_kV \subset V$ is locally closed. Furthermore, it is a smooth $\kk$-scheme and pure of dimension $d - k$. We write $[\cS_kV]$ for the set of irreducible components of $\cS_kV$.

One may show that $\cU_0V = \cS_0V$ is smooth. Every closed point in $\cS_1V$ is a double normal crossing point. The map $\NN \to \cP_{\bar v}$ given by $1 \mapsto \bar\rho$ is necessarily of the form $\theta_\ell\colon \NN \to P_\ell$ in the notation of Definition~\ref{defn:dtd-point}, and this $\ell \geq 1$ is constant on each irreducible component of $\cS_1V$. We say that $\ell$ is the \emph{kink} of that irreducible component. For every $\ell \geq 1$, we write $\cS_1^{(\ell)}V$ for the union of the irreducible components whose kink is $\ell$, and $[\cS_1^{(\ell)}V]$ for the set of those irreducible components.

\subsection{The classifying map $\eta_V^{(\ell)}\colon \cL\cS_V \to \cE xt^1(\Omega^1_V,\cO_V)$}\label{sect:classifying-map}\note{sect:classifying-map}

Let $(V,\cP,\bar\rho)$ be a toroidal crossing scheme. In \cite{FFR2021,FeltenGLDT}, we considered a map of sheaves of sets
\[\eta_V\colon \enspace \cL\cS_V \to \cE xt^1(\Omega^1_{V},\cO_{V})\]
given as follows: Let $W \to V$ be an affine \'etale open subset, and let $\lambda \in \Gamma(W,\cL\cS_V)$. Then, there is, up to non-unique isomorphism, a unique log smooth deformation $f_1\colon (W_1,\cM_1) \to S_1 = \Spec(\NN \to \kk[t]/(t^2))$ of the log smooth morphism $f_0\colon (W,\cM_0) \to S_0$ corresponding to $\lambda$. In particular, we have an underlying flat deformation, and therefore an extension 
\[0 \to \cO_W \to \Omega^1_{W_1}|_W \to \Omega^1_{W} \to 0.\]
Then, $\eta_V(\lambda)$ is the class in $\Gamma(W,\cE xt^1(\Omega^1_V,\cO_V))$ of this extension. This class is independent of the chosen log smooth deformation $f_1\colon (W_1,\cM_1) \to S_1$.

Let $\lambda \in \Gamma(W,\cL\cS_V)$ be represented by $(\cM_0,\alpha,q,\rho)$, and let $c \in \Gamma(W,\cO_V^\times)$. Then, we define $c \cdot \lambda$ to be represented by $(\cM_0,\alpha,q,c^{-1}\cdot \rho)$. This induces an $\cO_V^\times$-action on $\cL\cS_V$. In \cite[Prop.~5.2]{FFR2021}, we showed that $\eta_V$ is $\cO_V^\times$-equivariant this $\cO_V^\times$-action on $\cL\cS_V$ and the $\cO_V$-module structure on $\cE xt^1(\Omega^1_V,\cO_V)$. 

In \cite[Lemma~5.3]{FFR2021}, we showed that $\eta_V$ is injective if $(V,\cP,\bar\rho)$ is a normal crossing scheme (with its canonical $(\cP,\bar\rho)$, see \cite[Ex.~9.43]{FeltenGLDT}). However, when we take for example as $V$ the central fiber of the log smooth degeneration given by $xy = t^2$, then $\eta_V$ is the zero map. Thus, $\eta_V$ is not useful to compute sections of $\cL\cS_V$ in this situation.

In this article, we provide a remedy of this defect on the open subset
\[\cU_1^{(\ell)}V \coloneq \cU_0V \sqcup \cS_1^{(\ell)}V \subseteq V,\]
hence, by taking the union over all $\ell \geq 1$, a remedy on $\cU_1V$. In Construction~\ref{constr:eta-ell}, we provide a map of sheaves of sets 
\[\eta_V^{(\ell)}\colon \enspace \cL\cS_V \to \cE xt^1(\Omega^1_V,\cO_V).\]
It is a variant of $\eta_V$ but defined only on $\cU_1^{(\ell)}V \subseteq V$, not on the whole of $V$. We say that $\eta_V^{(\ell)}$ is the \emph{$\ell$-th classifying map}. Note that, on $\cU_1^{(\ell)}V$, the target $\cE xt^1(\Omega^1_V,\cO_V)$ is a line bundle on the reduced closed subset $\cS_1^{(\ell)}V \subset \cU_1^{(\ell)}V$ of codimension one because $\cU_1^{(\ell)}V$ is a normal crossing scheme (when considered as a scheme rather than a toroidal crossing scheme).

\begin{thm}\label{thm:ell-th-classifying-map}\note{thm:ell-th-classifying-map}
	Let $(V,\cP,\bar\rho)$ be a toroidal crossing scheme, and let $\eta_V^{(\ell)}$ be its $\ell$-th classifying map, defined on $\cU_1^{(\ell)}V \subseteq V$.
	\begin{enum:arabic}
		\item $\eta_V^{(\ell)}$ is injective.
		\item For $c \in \cO_V^\times$, we have $\eta_V^{(\ell)}(c \cdot \lambda) = c^\ell \cdot \eta_V^{(\ell)}(\lambda)$.
		\item The image of $\eta_V^{(\ell)}$ consists of the local generators of $\cE xt^1(\Omega^1_V,\cO_V)$.
	\end{enum:arabic}
\end{thm}

\subsection{The central fiber as a toroidal crossing scheme}

Suppose that we are in Situation~\ref{sitn:explicit-gls-candidate-kk-intro}, and that the conditions of Theorem~\ref{thm:explicit-gls-criterion-kk} hold so that $f\colon (X,X \setminus Z_X|V) \to (B|0)$ is a gls family. Here, we choose $Z_X$ as specified by Formula~\eqref{eqn:clubsuit} in Theorem~\ref{thm:explicit-gls-criterion-kk}. Furthermore, assume that $V$ carries a toroidal crossing structure $(\cP,\bar\rho)$. Then, we have $\cU_0V = \m{Reg}(V)$, $\cC_1^\times V = \m{Sing}(V;d) = \cC_1V$, $\cC_2^\times V = \cC_2V$, and $\cS_1^\times V = \cS_1V$ for the rank stratification of the toroidal crossing scheme $(V,\cP,\bar\rho)$.\footnote{This equality is our motivation for the notation $\cS_1^\times X_0$ etc.}

We would like to interpret the logarithmic central fiber $f_0\colon X_0 \to B_0$ as a section of $\cL\cS_{V}$, and indeed, this is often possible in a canonical way.

\begin{lem}\label{lem:explicit-gls-as-section-of-LS}\note{lem:explicit-gls-as-section-of-LS}
	Assume that we are in Situation~\ref{sitn:explicit-gls-candidate-kk-intro}, and that the conditions of Theorem~\ref{thm:explicit-gls-criterion-kk} are satisfied. Let $Z_X \subset X$ be the closed subset specified by Formula~\eqref{eqn:clubsuit}. Assume that $V$ is endowed with the structure of a toroidal crossing scheme $(\cP,\bar\rho)$. Suppose that:
	\begin{itemize}
		\item[] For every irreducible component $E \in [\cS_1^\times V]$, the kink $\ell(E)$ in the sense of Condition~\ref{item:ell-exists} in Theorem~\ref{thm:explicit-gls-criterion-kk} coincides with the kink as a stratum of the toroidal crossing scheme $(V,\cP,\bar\rho)$.
	\end{itemize}
	Then, on $V \setminus (\cC_2^\times V \cup Z_V^\circ)$, there is a unique isomorphism $(\cP,\bar\rho) \cong (\overline\cM_{X_0}, f^{\splus}(\bar\tau_0))$ of sheaves of monoids with a distinguished section. This isomorphism exhibits $f_0\colon X_0 \setminus (\cC_2^\times V \cup Z_V^\circ) \to B_0$ as an element of $\Gamma(V \setminus (\cC_2^\times V \cup Z_V^\circ),\cL\cS_{V})$.
\end{lem}

\subsection{Finite determinacy of the log structure on $X_0$}

Let $V/\kk$ be a scheme. Suppose that we have two instances of Situation~\ref{sitn:explicit-gls-candidate-kk-intro} $f\colon X \to B$ and $g\colon Y \to B$ (say over the same $B$ for simplicity), as well as chosen isomorphisms $V \cong X_0 \coloneq f^{-1}(0)$ and $V \cong Y_0 \coloneq g^{-1}(0)$ of $\kk$-schemes. Assume that the conditions of Theorem~\ref{thm:explicit-gls-criterion-kk} are satisfied so that both families become gls after choosing $Z_X \subset X$ and $Z_Y \subset Y$ according to Formula~\eqref{eqn:clubsuit}. Let $Z_V = Z_{X,0} \cup Z_{Y,0}$, and let $f_0\colon (V,V \setminus Z_V,\cM_0) \to S_0$ and $g_0\colon (V,V \setminus Z_V,\cN_0) \to S_0$ be the two induced structures of a morphism of partial log schemes on the central fiber. Then, we would like to determine if we have $f_0 \cong g_0$ as morphisms of partial log schemes.

The log smooth, saturated, and vertical morphism $f_0\colon X_0 \setminus Z_V \to S_0$ induces a structure of a toroidal crossing scheme $(\cP_f,\bar\rho_f)$ on $V \setminus Z_V$. Then, for every irreducible component $E \in [D \setminus Z_V]$, the kink $\ell_f(E)$ in the sense of Condition~\ref{item:ell-exists} in Theorem~\ref{thm:explicit-gls-criterion-kk} coincides with the kink of $E$ as a stratum of the toroidal crossing scheme $(V \setminus Z_V,\cP_f,\bar\rho_f)$. 

Similarly, we have the structure of a toroidal crossing scheme $(V \setminus Z_V,\cP_g,\bar\rho_g)$ induced by $g_0\colon Y_0 \setminus Z_V \to S_0$. Thus, if $f_0 \cong g_0$, then we have $(\cP_f,\bar\rho_f) \cong (\cP_g,\bar\rho_g)$, and for every $E \in [D \setminus Z_V]$, we have $\ell_f(E) = \ell_g(E)$ for the kinks in the sense of Condition~\ref{item:ell-exists} in Theorem~\ref{thm:explicit-gls-criterion-kk}.

Conversely, if we have $\ell_f(E) = \ell_g(E)$ for all $E \in [D \setminus Z_V]$, then we can apply Lemma~\ref{lem:explicit-gls-as-section-of-LS} to interpret $f_0\colon X_0 \setminus Z_V \to S_0$ and $g_0\colon Y_0 \setminus Z_V \to S_0$ as two sections of $\cL\cS_{V \setminus Z_V}$ for the same structure of a toroidal crossing scheme $(V \setminus Z_V,\cP_f,\bar\rho_f)$. Then, we can apply Theorem~\ref{thm:ell-th-classifying-map} to determine if the two sections are the same.

\begin{prop}\label{prop:comparing-sections-of-LS}\note{prop:comparing-sections-of-LS}
	Let $f\colon X \to B$ and $g\colon Y \to B$ be two instances of Situation~\ref{sitn:explicit-gls-candidate-kk-intro} over the same $B$. Assume that we have a fixed isomorphism $X_0 \cong Y_0 \eqcolon V$ of underlying $\kk$-schemes. Assume that the conditions of Theorem~\ref{thm:explicit-gls-criterion-kk} are satisfied for both $f\colon X \to B$ and $g\colon Y \to B$. Let $Z_X \subset X$ and $Z_Y \subset Y$ be the closed subsets defined by Formula~\eqref{eqn:clubsuit} in Theorem~\ref{thm:explicit-gls-criterion-kk}. Let $n \geq 0$ be such that $\ell_f(E) \leq n$ and $\ell_g(E) \leq n$ for all $E \in [\cS_1^\times V]$, and assume that $f_n \cong g_n$ as flat deformations of $V$. Then:
	\begin{enum:arabic}
		\item We have $Z_{X,0} = Z_{Y,0}$ under the given identification $X_0 \cong Y_0 = V$. 
		\item Let $Z_V \coloneq Z_{Y,0} \cong Z_{X,0}$. Then, there is a canonical isomorphism of log structures $\cM_{X_0 \setminus Z_V} \cong \cM_{Y_0 \setminus Z_V}$ on $V \setminus Z_V$ such that the two morphisms of partial log schemes $f_0\colon (V,V \setminus Z_V,\cM_{X_0 \setminus Z_V}) \to S_0$ and $g_0\colon (V, V \setminus Z_V,\cM_{Y_0 \setminus Z_V}) \to S_0$ are identified with each other.
	\end{enum:arabic}
\end{prop}

\begin{ex}
	We review \cite[Ex.~1.83]{FeltenGLDT}. Let $V = \Spec \kk[x,y,z,w]/(xy)$, and let $(\cP,\bar\rho)$ be its canonical structure of a toroidal crossing scheme, see \cite[Ex.~9.43]{FeltenGLDT}. Let
	\begin{align*}
		f\colon \enspace X &= \Spec \kk[t,x,y,z,w]/(xy - tzw) \to \IA^1_t, \\
		g\colon \enspace Y &= \Spec \kk[t,x,y,z,w]/(xy - tzw - t^2) \to \IA^1_t.
	\end{align*}
	Let $\eta_0(f), \eta_0(g) \in \Gamma(V,\cE xt^1(\Omega^1_V,\cO_V))$ be the extension classes defined by the first order deformations $f_1$ and $g_1$. Since evidently $f_1 \cong g_1$, we have $\eta_0(f) = \eta_0(g)$ in the notation of Section~\ref{sect:ext-criterion-dtd-points}. The zero locus of these sections is $Z_V = \{x = y = zw = 0\}$ so that $f$ and $g$ are semistable in a neighborhood of $V \setminus Z_V$ by the discussion in \cite[p.~29]{FeltenGLDT}. This implies that $f$ and $g$ induce sections $\lambda_f,\lambda_g \in \Gamma(V \setminus Z_V,\cL\cS_V)$. Furthermore, we have $\eta_V(\lambda_f) = \eta_0(f) = \eta_0(g) = \eta_V(\lambda_g)$. Since $(V,\cP,\bar\rho)$ is normal crossing, $\eta_V$ is injective, and we have $\lambda_f = \lambda_g$. Therefore, $f$ and $g$ can be interpreted as two gls deformations of the same gls family over $S_0$.
	
	Now, let us endow $V$ instead with the structure of a toroidal crossing scheme which comes from the log smooth deformation $\Spec \kk[t,x,y,z,w]/(xy - t^2) \to \IA^1_t$, i.e., the kink of $D = \{x = y = 0\}$ is $\ell = 2$. Then, we can consider the two families 
	\begin{align*}
		f\colon \enspace X &= \Spec \kk[t,x,y,z,w]/(xy - t^2zw) \to \IA^1_t, \\
		g\colon \enspace Y &= \Spec \kk[t,x,y,z,w]/(xy - t^2zw - t^3) \to \IA^1_t.
	\end{align*}
	Now, $\eta_V\colon \cL\cS_V \to \cE xt^1(\Omega^1_V,\cO_V)$ is constant, and we cannot use the preceding argument to show that the induced log structures on the central fiber can be identified with each other. However, $\eta_V^{(2)}$ can be used for this purpose since it is injective on its locus of definition by Theorem~\ref{thm:ell-th-classifying-map}. Proposition~\ref{prop:comparing-sections-of-LS} summarizes how to do this.
	\qedloz
\end{ex}

\section{Differential log smoothness}\label{sect:differential-log-smoothness}\note{sect:differential-log-smoothness}

Assume that we are in Situation~\ref{sitn:explicit-gls-candidate-kk-intro}. Theorem~\ref{thm:explicit-gls-criterion-kk} gives us a means to determine if $f\colon (X|V) \to (B|0)$ can be made gls, and if yes, gives us a suitable closed subset $Z_X \subset X$ by Formula~\eqref{eqn:clubsuit}. Furthermore, on $Z_V \cap \cS_1^\times V$, the morphism of log schemes $f\colon (X|V) \to (B|0)$ is actually not both log smooth and saturated. However, in a point $p \in T$, the map $f\colon (X|V) \to (B|0)$ could still be log smooth and saturated. We have seen an instance of this in Example~\ref{ex:112} and Remark~\ref{rem:precise-log-smooth-locus-II}.

As pointed out in Remark~\ref{rem:precise-log-smooth-locus-I}, we currently do not have a simple exact criterion for deciding whether $f\colon (X|V) \to (B|0)$ is log smooth and saturated in a point $p \in T$. However, if it is log smooth and saturated, then the relative log differential forms $\Omega^{1,\m{log}}_{X/B}$ are locally free of rank $d$. Furthermore, the sheaf $\Omega^{1,\m{log}}_{X/B}$ can be computed as $j_*\Omega^{1,\m{log}}_{U/B}$ in this case, where $U = X \setminus Z_X$, and where $j\colon U \to X$ is the inclusion. This gives a necessary criterion.

For a gls family $f\colon (X,U) \to S$, we write 
\[\cW^1_{X/S} \coloneq j_*\Omega^{1,\m{log}}_{U/S},\]
where $j\colon U \to X$ is the inclusion.

\begin{defn}
	A generically log smooth family $f\colon (X,U) \to S$ of relative dimension $d$ is  \emph{differentially log smooth} at a point $x \in X$ if $\cW^1_{X/S,x}$ is a free $\cO_{X,x}$-module of rank $d$. We write $\Delta\m{LogReg}(X/S)$ for the set of points at which $f\colon (X,U) \to S$ is differentially log smooth. \qeddef
\end{defn}

\begin{lem}
	Let $f\colon (X,U) \to S$ be a gls family of relative dimension $d \geq 1$.
	\begin{enum:arabic}
		\item $\Delta\m{LogReg}(X/S) \subseteq X$ is an open subset which contains $U$.
		\item The complement of $\Delta\m{LogReg}(X/S)$ is cut out by the $d$-th Fitting ideal of $\cW^1_{X/S}$.
		\item\label{item:diff-log-reg-base-change} Let $b\colon T \to S$ be a morphism between locally Noetherian fs log schemes, and let 
		\[
		\xymatrix{
			(Y,U_Y) \ar[r]^c \ar[d]^g & (X,U) \ar[d]^f \\
			T \ar[r]^b & S \\
		}
		\]
		be the fiber product in the category of partial log schemes. Then, we have 
		\[c^{-1}(\Delta\m{LogReg}(X/S)) \subseteq \Delta\m{LogReg}(Y/T).\]
	\end{enum:arabic}
\end{lem}
\begin{proof}
	Note that $\cW^1_{X/S,x}$ is coherent. Let $x \in X$ be such that $\cW^1_{X/S,x}$ is a free $\cO_{X,x}$-module of rank $d$. By Nakayama's lemma, we can find a surjection $\cO_X^{\oplus d} \to \cW^1_{X/S}$ in a neighborhood of $x$. When $\cK$ is the kernel, we have $\cK_x = 0$, and therefore $\cK = 0$ after shrinking the neighborhood. Therefore, $\Delta\m{LogReg}(X/S) \subseteq X$ is open. If $x \in U$, then $x \in \Delta\m{LogReg}(X/S)$ by \cite[Cor.~8.36]{FeltenGLDT}.
	
	Let $Z \subset X$ be the locus cut out by the $d$-th Fitting ideal $\m{Fitt}_d(\cW^1_{X/S})$. Then, we have $x \notin Z$ if and only if $\cW^1_{X/S,x}$ can be generated by $d$ elements. In particular, we have $\Delta\m{LogReg}(X/S) \subseteq X \setminus Z$. Conversely, if $x \notin Z$, then we have a surjection $\cO_X^{\oplus d} \to \cW^1_{X/S,x}$ in a neighborhood $x \in V \subseteq X$. Since $\cW^1_{X/S}$ is locally free of rank $d$ on $U \cap V$, we have $\cK|_{U \cap V} = 0$ for the kernel $\cK$ of the surjection. Since $\cK$ is a subsheaf of $\cO_X^{\oplus d}$, we find $\cK = 0$. But then, $x \in \Delta\m{LogReg}(X/S)$.
	
	Let $y \in Y$ be such that $c(y) \in \Delta\m{LogReg}(X/S)$. Then, $\cW^1_{X/S}$ is locally free of rank $d$ in a neighborhood $c(y) \in W \subseteq X$. Thus, $c^*\cW^1_{X/S}$ is locally free of rank $d$ on $c^{-1}(W)$, and therefore, $c^*\cW^1_{X/S} \to \cW^1_{Y/T}$ is an isomorphism on $c^{-1}(W)$. But then, we have $y \in \Delta\m{LogReg}(Y/T)$.
\end{proof}
\begin{rem}
	When $\gamma\colon c^*\cW^1_{X/S} \to \cW^1_{Y/T}$ is an isomorphism, then the inclusion in \ref{item:diff-log-reg-base-change} is an equality. This is, for example, always the case when $b\colon T \to S$ is flat. 
	
	We know many examples where $\gamma$ is not an isomorphism, for instance Example~\ref{ex:xyz-t-xyz}, but this circumstance alone does not prevent the inclusion from being an equality, as is the case, for instance, in Example~\ref{ex:xyz-t-xyz}. Nonetheless, it appears plausible that there are examples where the inclusion is not an equality. \qedloz
\end{rem}
\begin{rem}
	We currently do not know if, or under which conditions, in Situation~\ref{sitn:explicit-gls-candidate-kk-intro}, differential log smoothness implies log smoothness and saturatedness. \qedloz
\end{rem}

When we want to use differential log smoothness as a necessary criterion for log smoothness in Situation~\ref{sitn:explicit-gls-candidate-kk-intro}, then we have to compute $\cW^1_{X/S}$. This can be done effectively as follows.

\begin{lem}
	Assume that we are in the (explicit affine) situation of Lemma~\ref{lem:gls-criterion-bar}. Let $Z_X$ be given by Formula~\eqref{eqn:clubsuit} in Theorem~\ref{thm:explicit-gls-criterion-kk}, and let $U = X \setminus Z_X$. Then:
	\begin{enum:arabic}
		\item The total space $X$ is normal.
		\item The canonical map $\Omega^1_{U/B} \to \Omega^{1,\m{log}}_{U/B}$ becomes an isomorphism after passing to the bidual. 
		\item\label{item:compute-W} We have isomorphisms
		\[(\Omega^1_{X/B})^{\vee\vee} \xrightarrow{\cong} j_*((\Omega^1_{U/B})^{\vee\vee}) \xrightarrow{\cong} j_*((\Omega^{1,\m{log}}_{U/B})^{\vee\vee}) \xleftarrow{\cong} j_*\Omega^{1,\m{log}}_{U/B} = \cW^1_{X/B}.\]
		\item\label{item:compute-Delta-log-reg} $(\Omega^1_{X/B})^{\vee\vee}$ can be computed as $(\Omega^1_{\widebar\scX/\widebar\scB})^{\vee\vee}|_\scX \otimes_\kq \kk$, and we have 
		\[\Delta\m{LogReg}(X/B) = V(\m{Fitt}_d((\Omega^1_{\widebar\scX/\widebar\scB})^{\vee\vee})|_\scX \otimes_\kq \kk.\]
	\end{enum:arabic}
\end{lem}
\begin{proof}
	We leave the elementary proof to the reader.
\end{proof}

\begin{ex}
	We continue Example~\ref{ex:112}. Using the Fitting-ideal description, we compute in the script that
	\begin{align*}
		\Delta\m{LogSing}(X/B) &= \{t = y = z = w = u = 0\} \cup \{t = x = z = w = u = 0\} \\
		&\quad \cup \{t = x = y = z = u - w = 0\} \cup \{t = x = y = z = u + w = 0\}.
	\end{align*}
	Thus, the differentially log singular locus is the closure $\overline{Z_V^\circ}$ of $Z_V^\circ = \bigsqcup_{E \in [\cS_1^\times V]} Z_E$. This coincides with the precise log singular locus which we identified in Remark~\ref{rem:precise-log-smooth-locus-II}. \qedloz
\end{ex}

\begin{ex}\label{ex:xyz-t-xyz}\note{ex:xyz-t-xyz}
	We illustrate how to use differential log smoothness to verify a log singular point inside $T$ which is not in $\overline{Z_V^\circ}$. To this end, let $\kq = \QQ$ and $\kk = \CC$, and consider 
	\[f^\sharp\colon \enspace R_{\widebar\scB} = \QQ[t] \to \QQ[t,x,y,z]/(xyz - t(x + y + z)) = R_{\widebar\scX},\]
	which should be compared with \cite[Ex.~1.152]{FeltenGLDT}. Let $ f\colon \widebar\scX \to \widebar\scB$ be the associated morphism of schemes. The code for the following computations is contained in the script \texttt{example-log-singular-point-in-T.m2}.
	
	The ring $\widebar\scX$ is integral, of dimension $3$, and Cohen--Macaulay. The ring homomorphism $f^\sharp\colon R_{\widebar\scB} \to R_{\widebar\scX}$ is injective, and we have $\m{dim}(\m{Sing}(\widebar\scX/\widebar\scB;d)) = 1$. Thus, we are in Situation~\ref{sitn:affine-gls-candidate}, and the conditions of Lemma~\ref{lem:pre-gls-criterion} are satisfied. Moreover, $\m{Sing}(\widebar\scX/\widebar\scB;d) \subset \scV$ so that $f\colon \widebar\scX \to \widebar\scB$ is smooth over $\widebar\scB \setminus \{0\}$. We set $\scB = \widebar\scB$ and write $f\colon \scX \to \scB$ from now on.
	
	The central fiber is $\scV = \{xyz = 0\} \subset \IA^3$, a normal crossing scheme. The locus $\scD = \cC_1^\times \scV$ consists of three lines, and $\scT = \cC_2^\times\scV = \{0\}$ consists of a single point. In particular, $\m{dim}(\scT) = 0 \leq d - 2$.
	
	In the script, we compute that $\scZ^{(1)} = \{t = x = y = z = 0\} \subset \scV$ so that we have $\ell(E) = 1$ for every $E \in [\cS_1^\times\scV]$. Thus, the conditions in Lemma~\ref{lem:gls-criterion-bar} are satisfied. After base change to $\kk = \CC$, we have $Z_V^\circ = \varnothing$ and $Z_X = T = \{0\}$ (defined by Formula~\eqref{eqn:clubsuit}).
	
	By our computation, the radical of the second Fitting ideal of $\cW^1_{\scX/\scB} \cong (\Omega^1_{\scX/\scB})^{\vee\vee}$ is $(t,x,y,z)$. Thus, $0 \notin \Delta\m{LogReg}(X/B)$, and $f\colon (X|V) \to (B|0)$ cannot be log smooth and saturated in $0 \in X$. \qedloz
\end{ex}

\begin{rem}
	As explained in \cite[Ex.~1.152]{FeltenGLDT}, in the preceding example, there is a second method to show that $0 \in V$ is log singular. Namely, $\cW^1_{X/B} \otimes_{\cO_X} \cO_{V}$ is not reflexive in $0 \in V$, thus a fortiori not locally free of rank $d$. But then, also $\cW^1_{X/B}$ is not locally free of rank $d$ in $0 \in X$. This method is slightly stronger in that it even shows that $f\colon (X|V) \to (B|0)$ is not \emph{log toroidal} in $0 \in X$. For details, see \cite{FeltenGLDT}. \qedloz
\end{rem}
\begin{rem}
	In the file \texttt{methods-gls-families.m2}, the method to compute the differentially log smooth locus uses $\Theta^{1}_{X/B} = (\Omega^{1,\m{log}}_{X/B})^\vee$ rather than $\cW^1_{X/B}$ because this is far faster to compute in more complicated explicit examples such as Example~\ref{ex:projective-112}. The reader may test this for themselves. Since also $\Theta^1_{X/B} = (\cW^1_{X/B})^\vee$, local freeness of $\Theta^1_{X/B}$ is equivalent to local freeness of $\cW^1_{X/B}$. \qedex
\end{rem}

\newpage

\part{Proofs}

\section{Affine pre--gls families}\label{sect:affine-pregls-families}\note{sect:affine-pregls-families}

In this section, we prove Lemma~\ref{lem:pre-gls-criterion}. Thus, let $f^\sharp\colon \kq[t] \to R_{\widebar\scX}$ be a ring homomorphism as in Situation~\ref{sitn:affine-gls-candidate}. Let $f\colon \widebar\scX \to \widebar\scB$ be the associated morphism of affine schemes.

\begin{clm}\label{clm:flatness-criterion}\note{clm:flatness-criterion}
	Assume that $R_{\widebar\scX}$ is integral. Then, $f\colon \widebar\scX \to \widebar\scB$ is flat if and only if $f^\sharp\colon \kq[t] \to R_{\widebar\scX}$ is injective.
\end{clm}
\begin{proof}
	This follows from \cite[Prop.~III.9.7]{Hartshorne1977}.
\end{proof}

While the fibers of a pre--gls family are only required to satisfy Serre's condition $(S_2)$, it is in practice usually easier to check that they are even Cohen--Macaulay.\footnote{See for example the method \texttt{isCohenMacaulay} in the package \texttt{TestIdeals} for Macaulay2.}

\begin{clm}
	Assume that $f\colon \widebar\scX \to \widebar\scB$ is flat.\footnote{Here, we do \emph{not} assume that $\widebar\scX$ is integral.} Then, the fibers $\widebar\scX_b$ of $f$ are Cohen--Macaulay if and only if $R_{\widebar\scX}$ is Cohen--Macaulay.
\end{clm}
\begin{proof}
	This is \cite[Tag~0C0W]{stacks}.
\end{proof}

\begin{clm}
	Assume that $f\colon \widebar\scX \to \widebar\scB$ is flat. Then, every fiber $\widebar\scX_b$ is pure of dimension $d$ if and only if $\widebar \scX$ is pure of dimension $d + 1$.
\end{clm}
\begin{proof}
	First, suppose that every fiber $\widebar\scX_b$ is pure of dimension $d$. Then, for every closed point $x \in \widebar\scX$, we have $\m{dim}(\cO_{\widebar\scX,x}) = \m{dim}(\cO_{\widebar\scX_{f(x)},x}) + 1 = d + 1$. But then, $\widebar\scX$ is pure of dimension $d + 1$. 
	
	Conversely, suppose that $\widebar\scX$ is pure of dimension $d + 1$. Then, we have $\m{dim}(\cO_{\widebar\scX_{f(x)},x}) = \m{dim}(\cO_{\widebar\scX,x}) - 1 = d$ for every closed point $x \in \widebar\scX$, and every fiber $\widebar\scX_b$ over a closed point $b \in \widebar\scB$ is pure of dimension $d$. Now, \cite[Tag~02FZ]{stacks} can be used to show that also the generic fiber $\widebar\scX_\eta$ is pure of dimension $d$.
\end{proof}

Assume that we have already established that $f$ is flat and that its fibers are pure of dimension $d$. Let $\m{Sing}(\widebar\scX/\widebar\scB;d) \subset\widebar\scX$ be the closed subset cut out by the $d$-th Fitting ideal $\m{Fitt}_d(\Omega^1_{\widebar\scX/\widebar\scB})$. By \cite[Tag~0C3K]{stacks}, its complement $\m{Reg}(\widebar\scX/\widebar\scB;d)$ is the smooth locus of $f\colon \widebar\scX \to \widebar\scB$. 

Since the formation of $\m{Fitt}_d(\Omega^1_{\widebar\scX/\widebar\scB})$ commutes with base change along $\Spec \kappa(b) \to \widebar\scB$, we have $\m{Sing}(\widebar\scX/\widebar\scB;d) \times_{\widebar\scB} \Spec \kappa(b) = \m{Sing}(\widebar\scX_b;d)$ for any point $b \in \widebar\scB$. Furthermore, its complement coincides with the regular locus $\m{Reg}(\widebar\scX_b)$ of $\widebar\scX_b$ by \cite[Tag~0B8X]{stacks} since $\kq$ is perfect.

\begin{clm}\label{clm:reduced-fibers}\note{clm:reduced-fibers}
	Assume that $f\colon \widebar\scX \to \widebar\scB$ is flat, and that each fiber $\widebar\scX_b$ is pure of dimension $d$ and satisfies Serre's condition $(S_1)$. Then, all fibers $\widebar\scX_b$ are geometrically reduced if and only if 
	\[\m{dim}(\m{Sing}(\widebar\scX/\widebar\scB;d) \cap \widebar\scX_b) \leq d - 1\]
	for every closed point $b \in \widebar\scB$. 
\end{clm}
\begin{proof}
	If each fiber is geometrically reduced, then each fiber $\widebar\scX_b$ satisfies Serre's condition $(R_0)$, so $\m{Reg}(\widebar\scX_b) \subseteq \widebar\scX_b$ is dense, and $\m{dim}(\m{Sing}(\widebar\scX/\widebar\scB;d) \cap \widebar\scX_b) \leq d - 1$.
	
	Conversely, let us consider $\m{Sing}(\widebar\scX/\widebar\scB;d) \to \widebar\scB$. Over every closed point $b \in \widebar\scB$, the dimension of the fiber is $\leq d - 1$. By \cite[Tag~02FZ]{stacks}, the locus where the fiber dimension is $\leq d - 1$ is open in $\m{Sing}(\widebar\scX/\widebar\scB;d)$, so the density of the closed points yields that $\m{dim}(\m{Sing}(\widebar\scX/\widebar\scB;d) \times_{\widebar\scB} \Spec \kappa(\eta)) \leq d - 1$ also for the generic point $\eta \in \widebar\scB$.
	
	Now, $\m{Reg}(\widebar\scX_b) \subseteq \widebar\scX_b$ is dense, and $(R_0)$ is satisfied. Since we assume that $\widebar\scX_b$ satisfies $(S_1)$, we find that $\widebar\scX_b$ is reduced. Since $\m{char}(\kq) = 0$, the fibers are even geometrically reduced as soon as they are reduced (see \cite[Tag~020I]{stacks}).
\end{proof}

The condition from the claim is clearly satisfied in the setup of Lemma~\ref{lem:pre-gls-criterion}.

\begin{clm}
	In the situation of Lemma~\ref{lem:pre-gls-criterion}, there is an open subset $0 \in \scB \subseteq \widebar\scB$ such that $\widebar\scX_b$ is normal for every $0 \not= b \in \scB$.
\end{clm}
\begin{proof}
	Let $\m{Sing}(\widebar\scX/\widebar\scB;d) = \scS^\parallel \cup \scS^\perp$ be the decomposition in irreducible components which are contained in a single closed fiber, and irreducible components which are not contained in any single closed fiber. Since $\m{Sing}(\widebar\scX/\widebar\scB;d)$ has only finitely many irreducible components, there is an open subset $\scB^\star \subseteq \widebar\scB$ such that $\scS^\parallel \cap \scX^\star = \varnothing$ for $\scX^\star = f^{-1}(\scB^\star)$.
	
	Let $\scA \subset \scS^\perp$ be an irreducible component with its reduced induced scheme structure. Since $\scA$ is not contained in any single closed fiber, $\scA \to \widebar\scB$ is dominant and hence flat by \cite[Prop.~III.9.7]{Hartshorne1977}. Thus, for a closed point $b \in \widebar\scB$, we have $\m{dim}(\scA_b) \leq d - 2$ since $\m{dim}(\scA) \leq d - 1$. Now, if $b \in \scB^\star$ is closed, we find that $\m{dim}(\m{Sing}(\widebar\scX_b;d)) \leq d - 2$ so that $\widebar\scX_b$ satisfies Serre's condition $(R_1)$ because $\widebar\scX_b$ is pure of dimension $d$. Since $\widebar\scX_b$ is Cohen--Macaulay, it is normal.
	
	Since $\widebar\scB/\kq$ is smooth, $\m{Reg}(\widebar\scX/\widebar\scB;d)$ is regular. Because $\m{dim}(\m{Sing}(\widebar\scX/\widebar\scB;d)) \leq d - 1$, $\widebar\scX$ satisfies Serre's condition $(R_1)$, and then, it is normal. Since the generic fiber $\widebar\scX_\eta$ is a localization of $\widebar\scX$, it is normal as well.
\end{proof}

We conclude this section by recalling the following result.

\begin{lem}\label{lem:reduced-fibers-purity}\note{lem:reduced-fibers-purity}
	Let $f\colon X \to S$ be pre--gls. Let $b\colon T \to S$ be a morphism of locally Noetherian schemes, and let $g\colon Y \to T$ be the base change of $f$ along $b$. Then, $\cO_Y \to j_*\cO_{\m{Reg}(Y/T)}$ is injective, where $j\colon \m{Reg}(Y/T) \to Y$ is the inclusion.
\end{lem}
\begin{proof}
	See \cite[Rem.~8.37]{FeltenGLDT}.
\end{proof}

\section{Establishing generic log smoothness}

In this section, we prove Theorem~\ref{thm:explicit-gls-criterion-kk}. To this end, we review the \'etale local form of log smooth, saturated, and vertical morphisms to $(B|0)$ in codimension $1$, and then we give effective criteria to determine if $f\colon X \to B$ has this form \'etale locally.

\subsection{Dtc degeneration points}

The hardest part in showing that $f\colon (X|V) \to (B|0)$ can be made gls is to establish that $f$ is log smooth and saturated in codimension $1$ in $V$. The key to doing so is the toroidal characterization of log smoothness. In our setting, it takes the following form.

\begin{thm}\label{thm:toroidal-characterization}\note{thm:toroidal-characterization}
	In Situation~\ref{sitn:explicit-gls-candidate-kk-intro}, let $U \subseteq X$ be an open subset such that $f\colon (X|V) \to (B|0)$ is log smooth and saturated on $U$, and let $\itUpsilon = U \cap V$. Then, we have a stratification 
	\[\itUpsilon = \bigsqcup_{k \geq 0} \cS_k\itUpsilon\]
	by locally closed smooth subschemes $\cS_k\itUpsilon \subset \itUpsilon$ of codimension $k$ such that $p \in \cS_k\itUpsilon$ if and only if $\m{rk}(\overline\cM_{(X|V),\bar p}) = k + 1$. Here, $\bar p$ is a geometric point lying over $p \in \itUpsilon$. When $p \in \cS_k\itUpsilon$ is a closed point, we can find a commutative diagram of log schemes
	\[
	\xymatrix{
		(X|V) \ar[d]_f & (W,\cM_W) \ar[l]_-h \ar[r]_-g \ar@/^1.5em/[rr]^{g_\minuso} & (A_P \times \IA^r) \times_{A_\NN} (B|0) \ar[r] \ar[d] & A_P \times \IA^r \ar[d]^{\mu_{\theta;r}} \\
		(B|0) \ar@{=}[rr] & & (B|0) \ar[r]^{\ms{b}_\minuso} & A_\NN = (\IA^1_t|0) \\
	}
	\]
	and a closed point $w \in W$ where:
	\begin{enum:alph}
		\item $r = d - k$, and $P$ is a sharp toric monoid of rank $k + 1$;
		\item $\theta\colon \NN \to P$ is an injective and saturated homomorphism; it induces a log smooth and saturated morphism of log schemes
		\[\mu_{\theta;r}\colon \enspace A_P \times \IA^r = \Spec(P \to \kk[P \oplus \NN^r]) \to \Spec(\NN \to \kk[t]) = A_\NN;\]
		the log structure on $A_\NN$ coincides with the divisorial log structure defined by $\{0\} \subset \IA^1_t$, and the log structure on $A_P \times \IA^r$ coincides with the divisorial log structure defined by $\mu_{\theta;r}^{-1}(0) \subset A_P \times \IA^r$; when $\m{Int}(P) \subseteq P$ is the set of those elements of $P$ which are not contained in any proper face $F \subsetneq P$, then we have $\m{Int}(P) = \theta(1) + P$; the homomorphism $\theta\colon \NN \to P$ is vertical (see \cite[Defn.~I.4.3.1]{LoAG2018});
		\item $h\colon (W,\cM_W) \to (X|V)$ is strict and \'etale, and $h(w) = p$; 
		\item $g_\minuso\colon (W,\cM_W) \to A_P \times \IA^r$ is strict and \'etale, and $g_\minuso(w) = 0$;
		\item we have charts $\NN \to \cM_{A_\NN} \xrightarrow{\ms{b}_\minuso^\flat} \cM_{(B|0)}$ and $P \to \cM_{A_P \times \IA^r} \xrightarrow{g_\minuso^\flat} \cM_W$ which induce an identification of $f^\splus\colon \overline\cM_{(B|0),\bar 0} \to \overline\cM_{W,\bar w}$ with $\theta\colon \NN \to P$.
	\end{enum:alph}
\end{thm}
\begin{proof}
	First, one establishes the commutative diagram by \cite[Thm.~IV.3.3.3]{LoAG2018}. To see that $\theta(1) + P \subseteq \m{Int}(P)$, analyze the log trivial loci. Since $V$ is reduced, $\theta(1) + P$ is a radical monoid ideal. By \cite[Cor.~I.1.4.3]{LoAG2018}, $\theta(1) + P$ is then the intersection of the prime monoid ideals containing it. Now, every non-empty prime monoid ideal is the complement of a proper face, so this intersection is $\m{Int}(P)$. By \cite[Rem.~I.4.3.2]{LoAG2018}, $\theta\colon \NN \to P$ is vertical because $\theta(1)$ is not contained in any proper face. For the stratification, study the local model $A_P \times \IA^r \to A_\NN$ in some detail. See also \cite[\S\S8 - 9]{FeltenGLDT}.
\end{proof}

Conversely, if a closed point $p \in V$ admits a commutative diagram as in the theorem but in the category of schemes, then $f\colon (X|V) \to (B|0)$ is log smooth and saturated in a neighborhood of $p$ precisely because it is \'etale locally of the form $A_P \times \IA^r \to A_\NN$.

Due to the stratification, for understanding $f\colon (X|V) \to (B|0)$ in codimension $1$, we may assume that $1 \leq \m{rk}(P) \leq 2$. For such points, a complete classification of possible local models $\theta\colon \NN \to P$ is available.

\begin{prop}\label{prop:P-ell-classification}\note{prop:P-ell-classification}
	In Situation~\ref{sitn:explicit-gls-candidate-kk-intro}, let $U \subseteq X$ be an open subset such that $f\colon (X|V) \to (B|0)$ is log smooth and saturated on $U$.
	\begin{enum:arabic}
		\item Let $p \in \cS_0\itUpsilon$. Then, the homomorphism $\theta\colon \NN \to P$ in Theorem~\ref{thm:toroidal-characterization} is an isomorphism.
		\item Let $p \in \cS_1\itUpsilon$. Then, the homomorphism $\theta\colon \NN \to P$ in Theorem~\ref{thm:toroidal-characterization} is of the form
		\begin{equation*}
			\theta_\ell\colon \enspace \NN \to P_\ell \coloneq \langle (-1,\ell),(0,1),(1,0)\rangle \subset \ZZ^2, \quad 1 \mapsto (0,1);
		\end{equation*}
		$\ell$ is called the \emph{kink}; here, $P_\ell \subset \ZZ^2$ is the submonoid generated by the three indicated vectors.
	\end{enum:arabic}
\end{prop}
\begin{proof}
	For the first statement, use \cite[Cor.~I.4.8.12]{LoAG2018}. For the second statement, use \cite[p.~126]{Oda1988} and \cite[Prop.~10.1.6]{Cox2011}.
\end{proof}

Thus, to establish generic log smoothness, we have to show that enough among the non-smooth points in $X_0$ are \'etale locally of the form $A_{P_\ell} \times \IA^r \to A_\NN$. We say that $p$ is a \emph{double toroidal crossing degeneration point}. 

\begin{defn}\label{defn:dtd-point}\note{defn:dtd-point}
	Assume that we are in Situation~\ref{sitn:explicit-gls-candidate-kk-intro}. In particular, we have $d \geq 1$ fixed.
	\begin{enum:arabic}
		\item For $\ell \geq 1$, let $P_\ell = \langle(-1,\ell),(0,1),(1,0)\rangle \subset \ZZ^2$ be the submonoid generated by the three vectors. Let $\bar x = (1,0)$, $\bar y = (-1,\ell)$, and $\bar t = (0,1)$. Let $\theta_\ell\colon \NN \to P_\ell$ be given by $\theta_\ell(1) = (0,1) = \bar t$. For the monoid ring $\kk[P_\ell] = \bigoplus_{p \in P} \kk \cdot \m{z}^p$, we have an isomorphism $\kk[P_\ell] \cong \kk[t,x,y]/(xy - t^\ell)$ given by $\m{z}^{\bar x} \mapsto x$, $\m{z}^{\bar y} \mapsto y$, and $\m{z}^{\bar t} \mapsto t$.
		
		\item For $\ell \geq 1$, let
		\[\mu_{\ell;d}^\sharp\colon \enspace \kk[t] \to \kk[t,x,y,z_1,\ldots,z_{d - 1}]/(xy - t^\ell) \eqcolon R_{\ell;d},\]
		and let $\mu_{\ell;d}\colon M_{\ell;d} \coloneq \Spec R_{\ell;d} \to \IA^1_t$ be the associated morphism of schemes. Let $R_{\ell;d}^\m{h}$ be the Henselization of the localization of $R_{\ell;d}$ in $(t,x,y,z_1,\ldots,z_{d - 1})$.
		
		\item Let $\ell \geq 1$. A closed point $p \in D$ is a \emph{double toroidal crossing (dtc) degeneration point of kink $\ell$} if there is a commutative diagram
		\[
		\xymatrix{
			\cO_{X,p}^\m{h} & R_{\ell;d}^\m{h} \ar[l]^-\cong \\
			\cO_{B,0}^\m{h} \ar[u]^{f^\sharp} & \kk[t]_{(t)}^\m{h} \ar[l]_-{t \mapsto s}^-\cong \ar[u]_{\mu_{\ell;d}^\sharp}\\
		}
		\]
		of Henselian local rings and local homomorphisms where the horizontal maps are isomorphisms. \qeddef
	\end{enum:arabic}
\end{defn}

\subsection{Double normal crossing points}\label{sect:dnc-points}\note{sect:dnc-points}

When $p \in V$ is a dtc degeneration point, then $V$ is a \emph{double normal crossing scheme} at $p$. Conversely, when we determine below if $p$ is a dtc degeneration point, we will first determine if $V$ is a double normal crossing scheme at $p$.

\begin{sitn}\label{sitn:dnc-points}\note{sitn:dnc-points}
	Let $\kk$ be an algebraically closed field of characteristic $0$, and let $V/\kk$ be separated, of finite type, reduced, and pure of some dimension $d \geq 0$. \qeddef
\end{sitn}

\begin{defn}\label{defn:dnc-point}\note{defn:dnc-point}
	For $d \geq 1$, we write $R_{\times,d} = \kk[x,y,z_1,\ldots,z_{d - 1}]/(xy)$ and
	\[R_{\times,d}^\m{h} \coloneq \big(\kk[x,y,z_1,\ldots,z_{d - 1}]_{(x,y,z_1,\ldots,z_{d - 1})}/(xy)\big)^\m{h}.\]
	Furthermore, we write $R_{\times,0} = R_{\times,0}^\m{h} = \kk$. 
	
	In Situation~\ref{sitn:dnc-points}, a closed point $p \in V$ is a \emph{double normal crossing (dnc) point} if the (strictly) Henselian local ring $\cO_{V,p}^\m{h}$ is isomorphic to $R_{\times,d}^\m{h}$ as a $\kk$-algebra. We say that $V$ is a \emph{double normal crossing (dnc) scheme} if every closed point $p \in V$ is either regular or a dnc point. \qeddef
\end{defn}

We prove an intrinsic characterization of dnc points which is inspired by \cite[Tag~0CBR]{stacks}.\footnote{The criterion is relatively straightforward, and it appears likely that it is known to some experts. For instance, double normal crossing schemes are explored in the context of demi-normal schemes \cite{KSB1988,Berquist2014,Kollar2023,Posva2023,FantechiFranciosiPardini2023}. However, we did not find our desired characterization (or a similar one) in the literature.} Our aim is to introduce the \emph{double normal crossing locus} $\cU_1^\times V \subseteq V$ and to show that it is compatible with smooth morphisms $\pi\colon Y \to X$, i.e., $\cU_1^\times Y = \pi^{-1}(\cU_1^\times V)$. This is used in the proof of Proposition~\ref{prop:projective-gls-criterion}. We define the locus $\cU_1^\times V$ by means of the intrinsic characterization and then show that it is indeed the locus where $V$ is a dnc scheme.

\begin{defn}\label{defn:dnc-locus}\note{defn:dnc-locus}
	In Situation~\ref{sitn:dnc-points}, let $\nu\colon \widetilde V \to V$ be the normalization. Recall that the \emph{conductor} is the closed subscheme $i\colon C \to V$ defined by the annihilator of the $\cO_V$-module $(\nu_*\cO_{\widetilde V})/\cO_{V}$, and consider the Cartesian square 
	\[
	\xymatrix{
		\widetilde C \ar[r]^\mu \ar[d]^k & C \ar[d]^i \\
		\widetilde V \ar[r]^\nu & V, \\
	}
	\]
	which is called the \emph{conductor square}. We say that $V$ is \emph{intrinsically dnc} if the following conditions are satisfied:
	\begin{enum:alph}
		\item $\widetilde V$ is smooth;
		\item $C$ is smooth of dimension $d - 1$;
		\item $\nu\colon \widetilde V \to V$ is unramified;
		\item $\mu_*\cO_{\widetilde C}$ is a locally free $\cO_{C}$-module of rank $2$.
	\end{enum:alph}
	The \emph{dnc locus} $\cU_1^\times V \subseteq V$ is the union of all open subsets $W \subseteq V$ which are intrinsically dnc. Furthermore, we write $\cU_0^\times V$ as well as $\cS_0^\times V$ for the smooth locus of $V$; it is contained in $\cU_1^\times V$. We write 
	\[\cS_1^\times V = \cU_1^\times V \setminus \cU_0^\times V\]
	for its complement and endow $\cS_1^\times V$ with its reduced scheme structure as a closed subset of $\cU_1^\times V$. \qeddef
\end{defn}

Our notation and in particular the choice of the numbers in the subscript are motivated by the notation for the rank stratification of a toroidal crossing scheme. We will see a certain comparison in Proposition~\ref{prop:strata-dnc-rank}.

The compatibility of $\cU_1^\times V$ with smooth morphisms has the following form.

\begin{clm}\label{clm:intrinsic-dnc-smooth-local}\note{clm:intrinsic-dnc-smooth-local}
	Assume that we are in Situation~\ref{sitn:dnc-points}, and let $\pi\colon Y \to V$ be smooth of relative dimension $e \geq 0$. Then, we have $\cU_1^\times Y = \pi^{-1}(\cU_1^\times V)$. In particular, if $\cU_1^\times Y = Y$, and $\pi$ is surjective, then $\cU_1^\times V = V$.
\end{clm}
\begin{proof}
	By \cite[Tag~07TD]{stacks}, the normalization of $Y$ is given by $\nu'\colon \widetilde Y \coloneq Y \times_{V} \widetilde V \to Y$. Then, we have $(\nu'_*\cO_{\widetilde Y})/\cO_{Y} \cong \pi^*((\nu_*\cO_{\widetilde V})/\cO_{V})$ so that the conductor of $\nu'\colon \widetilde Y \to Y$ is given by $C' = Y \times_{V} C \to Y$. The claim now follows from \cite[Tags~036D, 02VM, 02VO]{stacks}.
\end{proof}

With this preparation, we show that $\cU_1^\times V = V$ for every dnc scheme $V$.

\begin{clm}
	Assume that we are in Situation~\ref{sitn:dnc-points}, and that $V$ is a dnc scheme in the sense of Definition~\ref{defn:dnc-point}. Then, $\cU_1^\times V = V$.
\end{clm}
\begin{proof}
	First, assume that $V$ is smooth. Then, $\nu\colon \widetilde V \to V$ is an isomorphism, $\widetilde V$ is smooth, and $C = \varnothing$. Thus, $\cU_1^\times V = V$.
	
	Next, assume that $V = \Spec \kk[x,y]/(xy)$. Then, $\widetilde V = \Spec \kk[x] \sqcup \Spec \kk[y]$. In particular, $\widetilde V$ is smooth, and $\nu$ is unramified since it is locally on $\widetilde V$ a closed immersion. The conductor ideal is $(x,y)$ so that $C \cong \Spec \kk$ is smooth of dimension $0$. Furthermore, $\mu\colon \widetilde C \to C$ is the map from two $\kk$-valued points to one $\kk$-valued point. Thus, $\mu_*\cO_{\widetilde C}$ is a free $\cO_{C}$-module of rank $2$.
	
	For $V = \Spec \kk[x,y,z_1,\ldots,z_{d - 1}]/(xy)$, we have $\cU_1^\times V = V$ by Claim~\ref{clm:intrinsic-dnc-smooth-local} since it maps smoothly and surjectively onto $\Spec \kk[x,y]/(xy)$.
	
	When $V$ is a dnc scheme, and $p \in V$ a non-smooth closed point, then $V$ is \'etale locally around $p$ isomorphic to $\Spec \kk[x,y,z_1,\ldots,z_{d - 1}]/(xy)$. Thus, Claim~\ref{clm:intrinsic-dnc-smooth-local} implies that $\cU_1^\times W = W$ for a Zariski neighborhood $p \in W \subseteq V$.
\end{proof}

The converse direction is more intricate. Let us first show that the conductor square is co-Cartesian. This appears to be a well-known fact.

\begin{clm}\label{clm:conductor-square-co-Cartesian}\note{clm:conductor-square-co-Cartesian}
	The conductor square is co-Cartesian.
\end{clm}
\begin{proof}
	It is sufficient to show this when $V$ is affine, so let $V = \Spec A$ and $\widetilde V = \Spec \widetilde A$. Let $I \subset A$ be the conductor ideal. We have to show that $\phi\colon A \to \widetilde A \times_{\widetilde A/I \cdot \widetilde A} A/I$ is an isomorphism.
	
	Since $A$ is reduced and $\widetilde A$ is normal, $A \to \widetilde A$ is injective so that also $\phi$ is injective. Furthermore, we have $I \cdot \widetilde A = I$ under the composite injection $I \to A \to \widetilde A$. Thus, if $(\widetilde a,[a]) \in \widetilde A \times_{\widetilde A/I} A/I$, there is some $i \in I$ with $\widetilde a = a + i$. But then, $\phi(a + i) = (\widetilde a,[a])$ so that $\phi$ is surjective.
\end{proof}

For lack of reference, we provide a proof of the following comparison result between the complete stalks and the Henselian stalks. We use this lemma also in Section~\ref{sect:dtd-point-criteria}.

\begin{lem}\label{lem:complete-Henselian}\note{lem:complete-Henselian}
	Let $\kk$ be an algebraically closed field, let $X$, $Y$ be locally of finite type over $\Spec \kk$, and let $f\colon X \to Y$ be a morphism. Let $x \in X$ be a closed point. Let $\hat f^\sharp\colon \widehat\cO_{Y,f(x)} \to \widehat\cO_{X,x}$ be the homomorphism between complete local rings, and let $f^\sharp\colon \cO_{Y,f(x)}^\m{h} \to \cO_{X,x}^\m{h}$ be the homomorphism between (strictly) Henselian local rings.
	\begin{enum:arabic}
		\item If $\hat f^\sharp\colon \widehat\cO_{Y,f(x)} \to \widehat\cO_{X,x}$ is surjective, then $f^\sharp\colon \cO_{Y,f(x)}^\m{h} \to \cO_{X,x}^\m{h}$ is surjective.
		\item If $\hat f^\sharp\colon \widehat\cO_{Y,f(x)} \to \widehat\cO_{X,x}$ is an isomorphism, then $f^\sharp\colon \cO_{Y,f(x)}^\m{h} \to \cO_{X,x}^\m{h}$ is an isomorphism.
	\end{enum:arabic}
\end{lem}
\begin{proof}
	Suppose that $\hat f^\sharp\colon \widehat\cO_{Y,f(x)} \to \widehat\cO_{X,x}$ is surjective. We have an exact sequence 
	\[(f^*\Omega^1_Y) \otimes_{\cO_{X,x}} \kappa(x) \to \Omega^1_X \otimes_{\cO_{X,x}} \kappa(x) \to \Omega^1_{X/Y} \otimes_{\cO_{X,x}} \kappa(x) \to 0\]
	of $\kk$-vector spaces. The map on the left can be identified with $\fm_{Y,f(x)}/\fm_{Y,f(x)}^2 \to \fm_{X,x}/\fm_{X,x}^2$, and hence it is surjective. Thus, the right-hand term is zero. Since $\Omega^1_{X/Y}$ is coherent, this implies $\Omega^1_{X/Y,x} = 0$, and $f$ is unramified at $x$. Therefore, $f$ is also quasi-finite at $x$, and \cite[Tag~05WP]{stacks} implies that $f^\sharp\colon \cO_{Y,f(x)}^\m{h} \to \cO_{X,x}^\m{h}$ is finite.
	
	Let $\fa^\m{h} \subset \cO_{Y,f(x)}^\m{h}$ be the kernel of $f^\sharp$. The quotient ring $\cO_{Y,f(x)}^\m{h}$ is Henselian, and we have an injective and finite map $\cO_{Y,f(x)}^\m{h}/\fa^\m{h} \to \cO_{X,x}^\m{h}$.
	
	Set $\widetilde\fa = \fa^\m{h} \cdot \widehat\cO_{Y,f(x)}$. Then,  $\widehat\cO_{Y,f(x)}/\widetilde\fa$ is the completion of $\cO_{Y,f(x)}^\m{h}/\fa^\m{h}$, and in particular, the map between them is flat. Then, also 
	\[\widehat\cO_{Y,f(x)}/\widetilde\fa \to \widehat\cO_{Y,f(x)}/\widetilde\fa \otimes_{\cO_{Y,f(x)}^\m{h}/\fa^\m{h}} \cO_{X,x}^\m{h}\]
	is injective. The right-hand side can be identified with the completion of $\cO_{X,x}^\m{h}$ as an $\cO_{Y,f(x)}^\m{h}/\fa^\m{h}$-module, and by \cite[Tag~0394]{stacks}, we can further identify it with $\widehat\cO_{X,x}$, the completion of $\cO_{X,x}^\m{h}$ in its maximal ideal. Thus, the map $\widehat\cO_{Y,f(x)}/\widetilde\fa \to \widehat\cO_{X,x}$ is injective. Furthermore, it is surjective by assumption, so it is an isomorphism. But then, $\cO_{Y,f(x)}^\m{h}/\fa^\m{h} \to \cO_{X,x}^\m{h}$ induces an isomorphism on completions, and hence is flat by \cite[Tag~0C4G]{stacks}.
	
	We can find an \'etale morphism $\pi\colon V \to Y$, a point $v \in V$ with $\pi(v) = f(x)$, and a closed subset $Z \subset V$ containing $v$ such that $\cO_{Z,v}^\m{h} \cong \cO_{Y,f(x)}^\m{h}/\fa^\m{h}$. Then, the morphism $\cO_{Y,f(x)}^\m{h}/\fa^\m{h} \to \cO_{X,x}^\m{h}$ is induced by the morphism $X \times_Y V \to Z$ of finite type. This morphism is flat and unramified in a neighborhood of $x$, so it is \'etale. But then, $\cO_{Y,f(x)}^\m{h}/\fa^\m{h} \to \cO_{X,x}^\m{h}$ is an isomorphism because the residue field is algebraically closed.
	
	Now, suppose that $\hat f^\sharp\colon \widehat\cO_{Y,f(x)} \to \widehat\cO_{X,x}$ is actually an isomorphism. By the Krull intersection theorem, $\cO_{Y,f(x)}^\m{h} \to \widehat\cO_{Y,f(x)}$ is injective, but then also $f^\sharp\colon \cO_{Y,f(x)}^\m{h} \to \cO_{X,x}^\m{h}$ is injective.
\end{proof}

Another elementary result for which we could not find a reference (in our precise way of stating it) is the following.

\begin{lem}\label{lem:Henselian-finite-splitting}\note{lem:Henselian-finite-splitting}
	Let $\kk$ be an algebraically closed field, let $R$ be a $\kk$-algebra of finite type, and let $\phi\colon R \to S$ be a finite ring homomorphism. Let $\fp \subset R$ be a maximal ideal, and let $\fq_1,\ldots,\fq_r \subset S$ be the maximal ideals of $S$ which lie over $\fp$. Then, the canonical ring homomorphism 
	\[S \otimes_R R_\fp^\m{h} \to S_{\fq_1}^\m{h} \times \ldots \times S_{\fq_r}^\m{h}\]
	is an isomorphism.
\end{lem}
\begin{proof}
	Note that $\fq_1,\ldots,\fq_r$ are precisely the $\kk$-valued points lying over $\fp$. Thus, there are precisely $r$ $\kk$-valued points in $S \otimes_R R_\fp^\m{h}$ lying over the maximal ideal $\fm_\fp^\m{h} \subset R_\fp^\m{h}$. Let $\fq_1',\ldots, \fq_r' \subset S \otimes_R R_\fp^\m{h}$ be their maximal ideals.
	
	Let $\fr \subset S \otimes_R R_\fp^\m{h}$ be any maximal ideal, and assume that $\psi^{-1}(\fr) \not= \fm_\fp^\m{h}$ under $\psi\colon R_\fp^\m{h} \to S \otimes_R R_\fp^\m{h}$. Since $\psi$ is finite, it satisfies going up, and $\fr$ was not maximal. Thus, every maximal ideal of $S \otimes_R R_\fp^\m{h}$ lies over $\fm_\fp^\m{h}$, and in particular, has residue field $\kk$.
	
	Because $\psi$ is finite, \cite[Tag~04GH]{stacks} shows the existence of a factorization $S \otimes_R R_\fp^\m{h} \cong S_1 \times \ldots \times S_m$ over $R_\fp^\m{h}$, where each $S_i$ is a Henselian local ring. Now, if $\fn_1,\ldots,\fn_m$ are the maximal ideals of $S_1$, \ldots, $S_m$, then the maximal ideals of $S_1 \times \ldots \times S_m$ are $\fn_i' \coloneq (1) \times \ldots \times \fn_i \times \ldots \times (1)$. In particular, we have $S_i \cong (S_1 \times \ldots \times S_m)_{\fn_i'}$, compatibly with the projection respectively the localization map from $S_1 \times \ldots \times S_m$. But then, we have 
	\[S \otimes_R R_\fp^\m{h} \cong (S \otimes_R R_\fp^\m{h})_{\fq_1'} \times \ldots \times (S \otimes_R R_\fp^\m{h})_{\fq_r'}.\]
	By \cite[Tag~08HU]{stacks}, $S_{\fq_i}^\m{h}$ is the Henselization of $(S \otimes_R R_\fp^\m{h})_{\fq_i'}$. Since we already know that the latter ring is Henselian, the claim follows.
\end{proof}

We are finally ready to show that closed points in $\cS_1^\times V$ are dnc points.

\begin{clm}
	In Situation~\ref{sitn:dnc-points}, let $p \in \cS_1^\times V$ be a closed point. Then, $p$ is a dnc point. In particular, $\cU_1^\times V$ is a dnc scheme in the sense of Definition~\ref{defn:dnc-point}.
\end{clm}
\begin{proof}
	After shrinking $V$, we can assume that $\cU_1^\times V = V$, and that $V = \Spec A$ is affine. Let $\nu^\sharp\colon A \to \widetilde A$ be the normalization, and let $I \subset A$ be the conductor ideal. Let $\fp \subset A$ be the maximal ideal defining $p \in V$.
	
	First, we show that $p \in C$. Assume the contrary. Then, we have $(\nu_*\cO_{\widetilde V}/\cO_{V})_p = 0$, and hence $\nu\colon \widetilde V \to V$ is an isomorphism in a neighborhood of $p$ because it is affine. Since $\widetilde V$ is smooth in a neighborhood of $\nu^{-1}(p)$, we find that $\cO_{V,p}$ is regular. This contradicts $p \in \cS_1^\times V$.
	
	Let $\ell = i \circ \mu = \nu \circ k\colon \widetilde C \to V$. Then, we have $\cO_{V} \cong (\nu_*\cO_{\widetilde V}) \times_{\ell_*\cO_{\widetilde C}} (i_*\cO_{C})$ by Claim~\ref{clm:conductor-square-co-Cartesian}. Since filtered colimits commute with finite limits (in the category of rings), we find that 
	\[\cO_{V,p}^\m{h} \cong \cO_{V,\bar p} \cong (\nu_*\cO_{\widetilde V})_{\bar p} \times_{(\ell_*\cO_{\widetilde C})_{\bar p}} (i_*\cO_{C})_{\bar p}\]
	for the stalks in the \'etale topology. This can be rewritten as 
	\[\cO_{V,p}^\m{h} \cong (\widetilde A \otimes_A A_\fp^\m{h}) \times_{(\widetilde A/I \otimes_A A_\fp^\m{h})} (A/I \otimes_A A_\fp^\m{h}).\]
	
	As a finite locally free morphism, $\mu\colon \widetilde C \to C$ is flat, and it is unramified as the base change of $\nu\colon \widetilde V \to V$. Thus, $\mu$ is \'etale. In particular, $\mu^{-1}(p) \to \{p\}$ is \'etale and finite locally free, so $\mu^{-1}(p) = \{q_1,q_2\}$ consists of two isolated $\kk$-valued points. Let $\fq_1,\fq_2 \subset \widetilde A$ be the two maximal ideals defining $q_1,q_2$. These are precisely the maximal ideals of $\widetilde A$ which lie over $\fp \subset A$ under $\nu^\sharp\colon A \to \widetilde A$. Furthermore, $\fq_1/I$ and $\fq_2/I$ are the maximal ideals of $\widetilde A/I$ which lie over $\fp$, and $\fp/I$ is the only maximal ideal of $A/I$ which lies over $\fp$.
	
	When we apply Lemma~\ref{lem:Henselian-finite-splitting} to our description of $\cO_{V,p}^\m{h}$, we obtain
	\[\cO_{V,p}^\m{h} \cong (\widetilde A_{\fq_1}^\m{h} \times \widetilde A_{\fq_2}^\m{h}) \times_{(\widetilde A/I)_{\fq_1/I}^\m{h} \times (\widetilde A/I)_{\fq_2/I}^\m{h}} (A/I)_{\fp/I}^\m{h}.\]
	
	By assumption, $(A/I)_{\fp/I}^\m{h}$ is a regular local ring of dimension $d - 1$. Let $\tilde z_1,\ldots,\tilde z_{d - 1}$ be a system of parameters. This defines a ring homomorphism 
	\[\kk[z_1,\ldots,z_{d - 1}]_{(z_1,\ldots,z_{d - 1})}^\m{h} \to (A/I)_{\fp/I}^\m{h},\]
	which is surjective on Zariski cotangent spaces, and hence surjective after completion. By Lemma~\ref{lem:complete-Henselian}, the map is already surjective before completion. Since both source and target are integral local rings of the same dimension, it is in fact an isomorphism.
	
	Since $\mu\colon \widetilde C \to C$ is \'etale, the two ring homomorphisms $(A/I)_{\fp/I}^\m{h} \to (\widetilde A/I)_{\fq_1/I}^\m{h}$ and $(A/I)_{\fp/I}^\m{h} \to (\widetilde A/I)_{\fq_2/I}^\m{h}$ are isomorphisms. Thus, we obtain systems of parameters $\bar z_1,\ldots,\bar z_{d - 1}$ and $\hat z_1,\ldots,\hat z_{d - 1}$, and isomorphisms with $\kk[z_1,\ldots,z_{d - 1}]_{(z_1,\ldots,z_{d - 1})}^\m{h}$.
	
	The local ring $\widetilde A_{\fq_1}^\m{h}$ is regular of dimension $d$. Thus, we can find a system of parameters $\bar x, \bar z_1,\ldots,\bar z_{d - 1}$ such that $\bar x$ cuts out $(\widetilde A/I)_{\fq_1}^\m{h}$, and such that $\bar z_1,\ldots,\bar z_{d - 1}$ induces the already given system of parameters on $(\widetilde A/I)_{\fq_1}^\m{h}$. We can find an analogous system of parameters $\hat y,\hat z_1,\ldots, \hat z_{d - 1}$ on $\widetilde A_{\fq_2}^\m{h}$.
	
	Now, we have 
	\[\cO_{V,p}^\m{h} \cong \left(\kk[x,z_1,\ldots]_{(\ldots)}^\m{h} \times \kk[y,z_1,\ldots]_{(\ldots)}^\m{h}\right) \times_{(\kk[z_1,\ldots]_{(\ldots)}^\m{h} \times \kk[z_1,\ldots]_{(\ldots)}^\m{h})} \kk[z_1,\ldots]_{(\ldots)}^\m{h},\]
	and hence $\cO_{V,p}^\m{h} \cong R_{\times,d}^\m{h}$, i.e., $p$ is a dnc point.
\end{proof}

We compare $\cS_0^\times V$ and $\cS_1^\times V$ with the stratification in Theorem~\ref{thm:toroidal-characterization}.

\begin{prop}\label{prop:strata-dnc-rank}\note{prop:strata-dnc-rank}
	In Situation~\ref{sitn:explicit-gls-candidate-kk-intro}, let $U \subseteq X$ be an open subset where $f\colon (X|V) \to (B|0)$ is log smooth and saturated. Then, we have $\itUpsilon \cap \cS_0^\times V = \cS_0\itUpsilon$ and $\itUpsilon \cap \cS_1^\times V = \cS_1\itUpsilon$ for the stratification $\cS_\bullet \itUpsilon$ from Theorem~\ref{thm:toroidal-characterization}.
\end{prop}
\begin{proof}
	For two locally closed subsets of $V$ to be equal, it is sufficient to show that they have the same closed points. For every closed point $p \in \itUpsilon$, we apply Theorem~\ref{thm:toroidal-characterization} and obtain an \'etale roof which compares a neighborhood of $p \in \itUpsilon$ with a neighborhood of $0 \in \mu_{\theta;r}^{-1}(0)$. Note that the number of irreducible components of $\mu_{\theta;r}^{-1}(0)$ equals the number of facets of $P$, and that each irreducible component of $\mu_{\theta;r}^{-1}(0)$ contains $0$.
	
	First, let $p \in \cS_0\itUpsilon$ be a closed point. By Proposition~\ref{prop:P-ell-classification}, $\theta\colon \NN \to P$ is an isomorphism, so $\mu_{\theta;r}$ is smooth, and $p \in \cS_0^\times \itUpsilon$.
	
	Conversely, let $p \in \itUpsilon \cap \cS_0^\times V$. Then, $\mu_{\theta;r}^{-1}(0)$ is smooth around $0$. In particular, it has only a single irreducible component, and hence $P$ has only one facet. The only sharp toric monoid with only one facet is $\NN$, so we have $\m{rk}(P) = 1$, and therefore $p \in \cS_0\itUpsilon$.
	
	Let $p \in \cS_1\itUpsilon$ be a closed point. By Proposition~\ref{prop:P-ell-classification}, $\theta\colon \NN \to P$ is of the form $\theta_\ell\colon \NN \to P_\ell$. One easily shows that $\mu_{\theta_\ell;r}^{-1}(0)$ is a dnc scheme and that $0 \in \cS_1^\times(\mu_{\theta_\ell;r}^{-1}(0))$, so we have $p \in \cS_1^\times V$.
	
	Finally, let $p \in \itUpsilon \cap \cS_1^\times V$ be a closed point. Then, we have $0 \in \cS_1^\times(\mu_{\theta;r}^{-1}(0))$. Since $\mu_{\theta;r}^{-1}(0)$ is the union of normal irreducible components, it has precisely two irreducible components which meet $0$. Thus, $P$ has precisely two facets $F_1$ and $F_2$. Now, every proper face is the intersection of facets, so the only proper face which is not a facet is $F_1 \cap F_2$. Thus, $F_1 \cap F_2 \subsetneq F_1 \subsetneq P$ is a maximal chain of faces. Since $P$ is sharp and toric, this implies $\m{rk}(P) = 2$ by \cite[Prop.~I.1.4.7]{LoAG2018}, and we have $p \in \cS_1\itUpsilon$.
\end{proof}

To enable computations, we also need a version of $\cU_1^\times V$ when our ground field is not algebraically closed.

\begin{sitn}\label{sitn:dnc-points-not-closed}\note{sitn:dnc-points-not-closed}
	Let $\kq$ be a field of characteristic $0$, and let $\kq \to \kk$ be a field extension to an algebraically closed field. Let $\scV/\kq$ be separated, of finite type, reduced, and pure of some dimension $d \geq 0$. Let $V = \scV \otimes_\kq \kk$ be the base change. \qeddef
\end{sitn}

In Situation~\ref{sitn:dnc-points-not-closed}, we say that $\scV$ is \emph{intrinsically dnc} if the analogs of the conditions listed in Definition~\ref{defn:dnc-locus} hold. In general, the \emph{dnc locus} $\cU_1^\times \scV \subseteq \scV$ is the union of all open subsets which are intrinsically dnc.

\begin{lem}\label{lem:U1x-base-change}\note{lem:U1x-base-change}
	In Situation~\ref{sitn:dnc-points-not-closed}, we have $\cU_1^\times V = (\cU_1^\times\scV) \otimes_\kq \kk$.
\end{lem}
\begin{proof}
	For the normalizations, we have $\widetilde V = \widetilde\scV \otimes_\kq \kk$ by \cite[Tag~0C3N]{stacks} because $\kq \to \kk$ is separable. Since $\kq \to \kk$ is flat, we then have $C = \scC \otimes_\kq \kk$ for the conductor loci.
	
	Let $\scV^a, \scV^b, \scV^c, \scV^d \subseteq \scV$ be the loci where the respective conditions of Definition~\ref{defn:dnc-locus} hold. Since the formation of the smooth locus and of the unramified locus commutes with base change, we have $V^a = \scV^a \otimes_\kq \kk$, $V^b = \scV^b \otimes_\kq \kk$, and $V^c = \scV^c \otimes_\kq \kk$.
	
	Let $\scW \subseteq \scV^d$ so that $\mu\colon \widetilde \scC \to \scC$ is finite locally free of degree $2$ on $\scC \cap \scW$. Let $W = \scW \otimes_\kq \kk$. Then, $C \cap W \to \scC \cap \scW$ is an fpqc cover, and $\mu\colon \widetilde C \to C$ is finite locally free of degree $2$ because this property is fpqc local on the target.
	
	For the converse, note that $\Spec \kk \to \Spec \kq$ is universally open by \cite[Tag~0383]{stacks}. Now, if $W \subseteq V$ is an open subset where $V$ is intrinsically dnc, then its image $\scW'$ under $V \to \scV$ is open. Now, $W \cap C \to \scW' \cap \scC$ is an fpqc cover, and $\mu\colon \widetilde\scC \to \scC$ is finite locally free of degree $2$. 
\end{proof}

\begin{proof}[Proof of Lemma~\ref{lem:compute-C2}]
	Since $T = \scT \otimes_\kq \kk$ and $\cU_1^\times V = (\cU_1^\times\scV) \otimes_\kq \kk$, it suffices to show the inclusion and the equality over $\Spec \kk$. First, let $x \in \cU_1^\times V$. Then, there are \'etale morphisms $g\colon W \to \cU_1^\times V$ and $h\colon W \to \{xy = 0\} \subset \IA^{d + 1}$ such that $x$ is in the image of $g$. Since the formation of $T$ commutes with \'etale morphisms, and since the analog of $T$ on $\{xy = 0\}$ is empty, we have $g^{-1}(T) = \varnothing$ and thus $x \notin T$. In particular, $\cU_1^\times V \subseteq V \setminus T$. Conversely, under the conditions listed in the statement of the lemma, we have $\scV \setminus \scT \subseteq \cU_1^\times\scV$ by definition of $\cU_1^\times\scV$. 
\end{proof}

\subsection{A preliminary characterization of dtc degeneration points}\label{sect:dtd-point-criteria}\note{sect:dtd-point-criteria}

We reformulate the condition of being a dtc degeneration point in a way that is more suitable for checking the condition. 

\begin{lem}\label{lem:dtd-intermediate}\note{lem:dtd-intermediate}
	Suppose that we are in Situation~\ref{sitn:explicit-gls-candidate-kk-intro}. For a closed point $p \in D$ and $\ell \geq 1$, let us consider the following condition:
	\begin{itemize}[label=$(\spadesuit)$]
		\item There are elements $\tilde x,\tilde y, \tilde z_1,\ldots, \tilde z_{d - 1} \in \fm_{X,p}^\m{h}$  which induce an isomorphism 
		\[\big(\kk[t,x,y,z_1,\ldots,z_{d - 1}]_{(t,x,y,z_1,\ldots,z_{d - 1})}/(t^\ell,xy)\big)^\m{h} \to \cO_{X,p}^\m{h}/(f^\sharp(s)^\ell)\]
		of (strictly) Henselian local rings which is given by $t \mapsto f^\sharp(s)$, $x \mapsto \tilde x$, $y \mapsto \tilde y$, $z_i \mapsto \tilde z_i$.
	\end{itemize}
	Then, the following statements hold:
	\begin{enum:arabic}
		\item If Condition~$(\spadesuit)$ holds, there is a unique element $\tilde u \in \cO_{X,p}^\m{h}$ with $\tilde x\tilde y = f^\sharp(s)^\ell \tilde u$.
		\item If $p$ is a dtc degeneration point of kink $\ell$, then Condition~$(\spadesuit)$ holds, and $\tilde x$, $\tilde y$, $\tilde z_1$, \ldots, $\tilde z_{d - 1}$ can be chosen such that $\tilde u \in \cO_{X,p}^\m{h}$ is invertible.
		\item If Condition~$(\spadesuit)$ holds, we can find an element 
		\[u \in P \coloneq \kk[t,x,y,z_1,\ldots,z_{d - 1}]_{(t,x,y,z_1,\ldots,z_{d - 1})}^\m{h}\]
		and an isomorphism $\phi\colon P/(xy - t^\ell u) \to \cO_{X,p}^\m{h}$ of $\kk$-algebras with $\phi(t) = f^\sharp(s)$, $\phi(x) = \tilde x$, $\phi(y) = \tilde y$, $\phi(z_i) = \tilde z_i$, and $\phi(u) = \tilde u$.
		\item If Condition~$(\spadesuit)$ holds and $\tilde u \in \cO_{X,p}^\m{h}$ is invertible, then $p$ is a dtc degeneration point of kink $\ell$.
	\end{enum:arabic}
\end{lem}
\begin{proof}
	Assume that Condition~$(\spadesuit)$ holds. Then, we have $\tilde x\tilde y \in (f^\sharp(s)^\ell)$ so that there is some $\tilde u \in \cO_{X,p}^\m{h}$ with $\tilde x\tilde y = f^\sharp(s)^\ell \tilde u$. Since $\cO_{X,p}^\m{h}$ is flat over the discrete valuation ring $\cO_{B,0}^\m{h}$, it is torsion-free. But then, $f^\sharp(s)$ is not a zero divisor, and $\tilde u$ is unique.
	
	Assume that $p$ is a dtc degeneration point. Let $\cp\colon R_{\ell;d}^\m{h} \to \cO_{X,p}^\m{h}$ be the isomorphism in Definition~\ref{defn:dtd-point}, and set $\tilde x = \phi(x)$, $\tilde y = \phi(y)$, $\tilde z_i = \phi(z_i)$. Then, Condition~$(\spadesuit)$ holds, and $u = 1$ is invertible.
	
	Condition~$(\spadesuit)$ implies that $[\tilde x]$, $[\tilde y]$, $[\tilde z_1]$, \ldots, $[\tilde z_{d - 1}]$ is a $\kk$-basis of $\fm_{X_0,p}^\m{h}/(\fm_{X_0,p}^\m{h})^2$. Now, the kernel of $\fm_{X,p}^\m{h}/(\fm_{X,p}^\m{h})^2 \to \fm_{X_0,p}^\m{h}/(\fm_{X_0,p}^\m{h})^2$ is generated by $[f^\sharp(s)]$ as a $\kk$-vector space so that $\fm_{X,p}^\m{h}/(\fm_{X,p}^\m{h})^2$ is generated by $[f^\sharp(s)]$, $[\tilde x]$, $[\tilde y]$, $[\tilde z_1]$, \ldots, $[\tilde z_{d - 1}]$.
	
	Define $\psi\colon P \to \cO_{X,p}^\m{h}$ by $t \mapsto f^\sharp(s)$, $x \mapsto \tilde x$, $y \mapsto \tilde y$, $z_i \mapsto \tilde z_i$. One may easily check that $\psi$ is local. Let $\fn \subset P$ be the maximal ideal. Then, the induced map $\fn/\fn^2 \to \fm_{X,p}^\m{h}/(\fm_{X,p}^\m{h})^2$ is surjective so that also $\widehat\psi\colon \widehat P \to \widehat\cO_{X,p}$ is surjective. Then, Lemma~\ref{lem:complete-Henselian} guarantees that $\psi$ is surjective.
	
	Now, let $u \in P$ be such that $\psi(u) = \tilde u$. Then, $\psi(xy - t^\ell u) = 0$, and we have an induced surjective local homomorphism $\phi\colon P/(xy - t^\ell u) \to \cO_{X,p}^\m{h}$. Let $K$ be its kernel. Since $\cO_{X,p}^\m{h}$ is flat over $\cO_{C,0}^\m{h}$, the embedding $K \to P/(xy - t^\ell u)$ is universally injective. In particular, $K/t \to P/(xy,t)$ is injective. Then, $K/t$ is the kernel of $\phi_0\colon P/(xy,t) \to \cO_{X,p}^\m{h}/(t)$, and hence $K/t = 0$ by Condition~$(\spadesuit)$. Now, $K$ is a finitely generated $P/(xy - t^\ell u)$-module, so Nakayama's lemma shows that $K = 0$. Thus, $\phi$ is an isomorphism.
	
	Finally, suppose that Condition~$(\spadesuit)$ holds, and that $\tilde u$ is invertible. Since $\psi$ is local, we have $u \in P^\times$. Then, we have an automorphism 
	\[\alpha\colon \enspace P \to P, \quad t \mapsto t, \enspace x \mapsto x/u, \enspace y \mapsto y, \enspace z_i \mapsto z_i,\]
	which satisfies $\alpha(xy - t^\ell) = xy/u - t^\ell$. Now, composing $\psi$ with $\alpha$ yields the desired isomorphism $\cp\colon R_{\ell;d}^\m{h} \to \cO_{X,p}^\m{h}$ in Definition~\ref{defn:dtd-point}.
\end{proof}

\subsection{The extension criterion for dtc degeneration points}\label{sect:ext-criterion-dtd-points}\note{sect:ext-criterion-dtd-points}

We simplify the criteria in Lemma~\ref{lem:dtd-intermediate} and restate them in terms of extension classes. Let us now assume that $f\colon X \to B$ is pre--gls. The primary object of our study is the short exact sequence 
\[
	\xymatrix{
		0 \ar[r] & f^*\Omega^1_B \ar[r] & \Omega^1_X \ar[r] & \Omega^1_{X/B} \ar[r] & 0 \\
		0 \ar[r] & \cO_X \ar[r] \ar[u]^{1 \mapsto f^*ds} & \Omega^1_X \ar[r] \ar@{=}[u] & \Omega^1_{X/B} \ar[r] \ar@{=}[u] & 0. \\
	}
\]
The left vertical arrow is an isomorphism by the assumptions in Situation~\ref{sitn:explicit-gls-candidate-kk-intro}. To see that the left horizontal maps are injective, note that they are injective on $\m{Reg}(X/B)$ so that injectivity on $X$ follows from Lemma~\ref{lem:reduced-fibers-purity}. We denote the class of this extension in $\Gamma(X,\cE xt^1(\Omega^1_{X/B},\cO_X))$ by $\eta(f)$. The same argument shows that the sequence remains exact after pullback along $c_k\colon X_k \to X$. We denote the class of the pullback in $\Gamma(X_k,\cE xt^1(\Omega^1_{X_k/B_k},\cO_{X_k}))$ by $\eta_k(f)$.

One easily sees that the preceding property holds for the pullback of all extensions of $\cO_X$, $\Omega^1_{X/B}$ along $c_k\colon X_k \to X$. Thus, this is an instance of the following condition on a morphism of schemes $f\colon X \to Y$ and two $\cO_Y$-modules $\cA$ and $\cB$:
\begin{itemize}
	\item[$(\bigstar)$] For every open subset $V \subseteq Y$ and every extension 
	\[e\colon \quad 0 \to \cA|_V \to \cE \to \cB|_V \to 0,\]
	the pullback $f^*(e)$ is a short exact sequence on $f^{-1}(V)$.
\end{itemize}
If this holds, we obtain a map of \emph{sets} $f^*\colon \m{Ext}^1_V(\cB|_V,\cA|_V) \to \m{Ext}^1_U((f^*\cB)|_U,(f^*\cA)|_U)$ for every open subset $V \subseteq Y$ and $U = f^{-1}(V)$. For lack of reference, we include a proof of the following expected result.

\begin{lem}\label{lem:extension-properties}\note{lem:extension-properties}
	Let $f\colon X \to Y$ be a morphism of schemes.
	\begin{enum:arabic}
		\item\label{item:ext-pushout} If $(\cB,\cA)$ and $(\cB,\cA')$ satisfy $(\bigstar)$, and if $\phi\colon \cA \to \cA'$ is a homomorphism, then we have a commutative diagram
		\[
		\xymatrix{
			\m{Ext}_V^1(\cB|_V,\cA|_V) \ar[r]^-{\phi_\bullet} \ar[d]_{f^*} & \m{Ext}_V^1(\cB|_V,\cA'|_V) \ar[d]^{f^*} \\
			\m{Ext}_U^1((f^*\cB)|_U,(f^*\cA)|_U) \ar[r]^-{(f^*\phi)_\bullet} & \m{Ext}_U^1((f^*\cB)|_U,(f^*\cA')|_U) \\
		}
		\]
		of sets for every open subset $V \subseteq Y$ and $U = f^{-1}(V)$.
		\item\label{item:ext-pullback} If $(\cB,\cA)$ and $(\cB',\cA)$ satisfy $(\bigstar)$, and if $\psi\colon\cB' \to \cB$ is a homomorphism, 
		then we have a commutative diagram
		\[
		\xymatrix{
			\m{Ext}_V^1(\cB|_V,\cA|_V) \ar[r]^{\psi^\bullet} \ar[d]_{f^*} & \m{Ext}_V^1(\cB'|_V,\cA|_V) \ar[d]^{f^*} \\
			\m{Ext}_U^1((f^*\cB)|_U,(f^*\cA)|_U) \ar[r]^-{(f^*\psi)^\bullet} & \m{Ext}_U^1((f^*\cB')|_U,(f^*\cA)|_U) \\
		}
		\]
		of sets for every open subset $V \subseteq Y$ and $U = f^{-1}(V)$.
		\item\label{item:ext-linear} If $(\cB,\cA)$ satisfies $(\bigstar)$, then, for two extensions $e_1$ and $e_2$, we have $f^*(e_1 + e_2) = f^*(e_1) + f^*(e_2)$ for the Baer sums in $\m{Ext}^1_U((f^*\cB)_U,(f^*\cA)_U)$. Furthermore, if $c \in \Gamma(V,\cO_Y)$, then $f^*(c \cdot e) = f^\sharp(c) \cdot f^*(e)$ in $\m{Ext}^1_U((f^*\cB)_U,(f^*\cA)|_U)$.
	\end{enum:arabic}
\end{lem}
\begin{proof}
	The formation of $\phi_\bullet$ and $(f^*\phi)_\bullet$ involves only a pushout, i.e., a colimit, so it is preserved under $f^*$, showing \ref{item:ext-pushout}. Since $c \cdot e = (\mu_c)_\bullet(e)$ for the multiplication map $\mu_c\colon \cA \to \cA$, this also shows $f^*(c \cdot e) = f^\sharp(c) \cdot f^*(e)$ in \ref{item:ext-linear}.
	
	We consider \ref{item:ext-pullback}. Let $e$ be an extension which gives a class in $\m{Ext}^1_V(\cB|_V,\cA|_V)$. By construction, we have a map of extensions $\psi^\bullet(e) \to e$. Since $(\bigstar)$ holds, we obtain a map of extensions $f^*(\psi^\bullet(e)) \to f^*(e)$ after pullback. However, this map factors as $f^*(\psi^\bullet(e)) \to (f^*\psi)^\bullet(f^*(e)) \to f^*(e)$. Now, the first map is an isomorphism because it is between extensions of the same two sheaves of modules.
	
	Additivity of $f^*$ in \ref{item:ext-linear} is now a formal consequence of \ref{item:ext-pushout} and \ref{item:ext-pullback} since $e_1 + e_2 = \Delta^\bullet\Sigma_\bullet(e_1 \oplus e_2)$ for the summation map $\Sigma\colon \cA \oplus \cA \to \cA$ and the diagonal $\Delta\colon \cB \to \cB \oplus \cB$.
\end{proof}

In particular, we have a restriction \emph{homomorphism}
\[\rho_{k}\colon \enspace\cE xt^1(\Omega^1_{X/B},\cO_X) \to (c_k)_*\cE xt^1(\Omega^1_{X_k/B_k},\cO_{X_k}).\]
The adjoint map on $X_k$ induces an isomorphism near dnc points. The analogous result holds for extension classes of $\Omega^1_X$ instead of $\Omega^1_{X/B}$.

\begin{lem}\label{lem:Omega-projdim-1}\note{lem:Omega-projdim-1}
	Assume that we are in Situation~\ref{sitn:explicit-gls-candidate-kk-intro}, and suppose that $f\colon X \to B$ is pre--gls. Let $p \in D$ be a closed dnc point. Then, both $\Omega^1_{X/B}$ and $\Omega^1_X$ have projective dimension $\leq 1$ in a neighborhood of $p$. Moreover, for both $\cA = \Omega^1_{X/B}$ and $\cA = \Omega^1_X$, the induced map 
	\[c_k^*\cE xt^1(\cA,\cO_X) \to \cE xt^1(c_k^*\cA,\cO_{X_k})\]
	is an isomorphism in a neighborhood of $p$.
\end{lem}
\begin{proof}
	When $g\colon Y \to T$ is a smooth morphism of schemes, and $i\colon D \to Y$ is a closed immersion defined by a sheaf of ideals $\cI \subset \cO_Y$, such that $g \circ i\colon D \to T$ is smooth, then 
	\[0 \to \cI/\cI^2 \to i^*\Omega^1_{Y/T} \to \Omega^1_{D/T} \to 0\]
	is exact and locally split. Now, if $\cI$ is a line bundle, and $g \circ i\colon D \to T$ is flat with reduced fibers, the sequence, although not necessarily locally split, is still exact. In particular, it gives a projective resolution of $\Omega^1_{D/T}$ of length $1$.
	
	In Situation~\ref{sitn:explicit-gls-candidate-kk-intro}, if $p \in D$ is a dnc point, condition $(\spadesuit)$ from Lemma~\ref{lem:dtd-intermediate} holds with $\ell = 1$, so we have an isomorphism $\cO_{X,p}^\m{h} \cong P/(xy - tu)$ for some $u \in P$. This allows us to write some \'etale neighborhood of $p$ in $X$ in the form considered in the previous paragraph. 
	
	Now, let either $\cA = \Omega^1_{X/B}$ or $\cA = \Omega^1_X$, and let $0 \to \cF_1 \to \cF_0 \to \cA \to 0$ be a coherent free resolution, locally defined in a neighborhood of $p$. Then, we obtain a resolution of $\cE xt^1(\cA,\cO_X)$. After pullback to $X_k$, we obtain
	\[\cH om(c_k^*\cF_0,\cO_{X_k}) \to \cH om(c_k^*\cF_1,\cO_{X_k}) \to c_k^*\cE xt^1(\cA,\cO_{X_k}) \to 0\]
	because $\cF_0$ and $\cF_1$ are finite locally free. However, since $\cA$ is locally free in relative codimension $0$, also $0 \to c_k^*\cF_1 \to c_k^*\cF_0 \to c_k^*\cA \to 0$ is exact. The long exact sequence for $\cE xt^1(-,\cO_{X_k})$ shows that the two comparison maps in the statement are isomorphisms.
\end{proof}

Let us now show that $\eta_{k - 1}(f)$ depends only on the finite-order deformation $f_k\colon X_k \to B_k$.

\begin{lem}\label{lem:eta-f-alternative}\note{lem:eta-f-alternative}
	Assume that we are in Situation~\ref{sitn:explicit-gls-candidate-kk-intro}, and suppose that $f\colon X \to B$ is pre--gls. For all $1 \leq k \leq n$, we have a commutative diagram
	\[
	\xymatrix{
		0 \ar[r] & (\cO_X \cdot ds)|_{X_{k - 1}} \ar[r] \ar[d]^\cong & \Omega^1_X|_{X_{k - 1}} \ar[d]^\cong \ar[r] & \Omega^1_{X/B}|_{X_{k - 1}} \ar[d]^\cong \ar[r] & 0 \\
		0\ar[r] & (f_n^*\Omega^1_{B_n})|_{X_{k - 1}} \ar[r] & \Omega^1_{X_n}|_{X_{k - 1}} \ar[r] & \Omega^1_{X_n/B_n}|_{X_{k - 1}} \ar[r] & 0 \\
	}
	\] 
	of coherent sheaves on $X_{k - 1}$. The two rows are exact. In particular, $\eta_{k - 1}(f)$ depends only on $f_k\colon X_k \to B_k$.
\end{lem}
\begin{proof}
	The commutativity of the diagram is clear. We have already seen that the upper row is exact, so it is sufficient to show that the three vertical maps are isomorphisms.
	
	To see that the left vertical map is an isomorphism, consider 
	\[(s^{n + 1})/(s^{2n + 2}) \to \Omega^1_B|_{B_n} \to \Omega^1_{B_n} \to 0.\]
	Now, we have $d(s^{n + 1}) = (n + 1)s^nds$ so that the right-hand surjection becomes an isomorphism after restricting to $B_{k - 1}$. The same argument also shows that the map in the middle is an isomorphism. The isomorphy of the right-hand map is clear.
\end{proof}

More generally, when $f_n\colon X_n \to B_n$ is pre--gls, we can define 
\[\eta_{k - 1}(f_n) \in \Gamma(X_k,\cE xt^1(\Omega^1_{X_k/B_k},\cO_{X_k}))\]
for any $1 \leq k \leq n$ by the lower row in the commutative diagram in the lemma.

\begin{lem}\label{lem:trivial-defo-characterization}\note{lem:trivial-defo-characterization}
	Let $n \geq 1$, and let $f_n\colon X_n \to B_n$ be \emph{affine} and pre--gls. Let $1 \leq k \leq n$. Then, $f_k\colon X_k \to B_k$ is a trivial flat deformation, i.e., $X_k \cong X_0 \times_{\Spec \kk} B_k$ over $B_k$ and compatibly with the two embeddings of $X_0$, if and only if $\eta_{k - 1}(f_n) = 0$.
\end{lem}
\begin{proof}
	First, assume that $f_k\colon X_k \to B_k$ is a trivial flat deformation. Then, the lower short exact sequence in Lemma~\ref{lem:eta-f-alternative} is (globally) split (for $n = k - 1$) so that we have $\eta_{k - 1}(f_n) = 0$.
	
	For the converse direction, we proceed by induction over $k$. The case $k = 1$ is well-known, see \cite[Thm.~2.4.1]{Sernesi2006}. For the induction step, suppose that we know the equivalence for some $k \geq 1$ with $k < n$, and assume that $\eta_k(f_n) = 0$. Then, we also have $\eta_{k - 1}(f_n) = 0$, and therefore an isomorphism $\cp_k\colon X_k \cong X_0 \times_{\Spec \kk} B_k$.
	
	Let us write $A_i = \Gamma(B_i,\cO_{B_i}) \cong \kk[s]/(s^{i + 1})$ and $R_i = \Gamma(X_i,\cO_{X_i})$ for brevity. We have an isomorphism 
	\[\chi\colon \enspace D_{k + 1} \coloneq A_{k + 1} \times_\kk \kk[\eps]/(\eps^2) \to A_{k + 1} \times_{A_{k}} A_{k + 1} \eqcolon E_{k + 1}\]
	of Artinian local rings given by $\chi(a,a_0 + \lambda\eps) = (a,a + \lambda s^{k + 1})$. Note that $D_0 = \kk[\eps]/(\eps^2)$. Let us write $\pi\colon E_{k + 1} \to A_{k + 1}$ for the first projection, and $\varpi\colon E_{k + 1} \to A_{k + 1}$ for the second projection.
	
	Let $e \in \m{Ext}^1(\Omega^1_{X_0/B_0},\cO_{X_0})$ be an extension class corresponding to a flat deformation $f_e\colon X_e \to \Spec(D_0)$ via \cite[Thm.~2.4.1]{Sernesi2006}. Then, we take the pushout with the trivial deformation over $B_{k + 1}$ and obtain a flat deformation 
	\[f'_e\colon \enspace X_e' = (\lvert X_0\rvert,(\cO_{X_0} \otimes_\kk A_{k + 1}) \times_{\cO_{X_0}} \cO_{X_e}) \to \Spec(D_{k + 1}).\]
	Now, we also have $D_0 = D_{k + 1} \times_{A_{k + 1}} \kk$ so that we have two Cartesian squares 
	\[
	\xymatrix{
		X_e \ar[r] \ar[d] & X_e' \ar[r] \ar[d] & X_e \ar[d] \\
		\Spec(D_0) \ar[r] & \Spec(D_{k + 1}) \ar[r] & \Spec(D_0) \\
	}
	\]
	in the category of schemes because the restriction of $X_e'$ to $\Spec(A_{k + 1})$ is trivial. We expand the right-hand square to a diagram
	\[
	\xymatrixcolsep{2.5em}
	\xymatrix{
		Y_e \ar@{=}[r] \ar[d]_{g_e} & X_e'' \times^\varpi_{\Spec(E_{k + 1})} B_{k + 1} \ar[d] \ar[r] & X_e'' \ar[d] \ar[r] & X_e' \ar[d]^{f_e'} \ar[r] & X_e \ar[d]^{f_e}  \\
		C_{k + 1} \ar@{=}[r] & \Spec(A_{k + 1}) \ar[r]^-{\Spec(\varpi)} & \Spec(E_{k + 1}) \ar[r]^-{\Spec(\chi)}_-\cong & \Spec(D_{k + 1}) \ar[r] & \Spec(D_0) \\
	}
	\]
	with all squares Cartesian. 
	
	Explicitly, the composition $D_0 \to A_{k + 1}$ in the lower row is given by $a_0 + \lambda\eps \mapsto a_0 + \lambda s^{k + 1}$. Thus, when we further compose with $A_{k + 1} \to A_k$, the composition factors as $D_0 \to \kk \to A_k$. In particular, $Y_e \times_{B_{k + 1}} B_k$ is a trivial deformation, which we can identify with $X_k$.
	
	Let $q\colon Y_e \to X_e$ be the outer morphism in the upper row. Then, we obtain a commutative diagram
	\[
	\xymatrix{
		(q^*f_e^*\Omega^1_{\Spec(D_0)})|_{X_k} \ar[r] \ar[d] & (q^*\Omega^1_{X_e})|_{X_k} \ar[r] \ar[d] & (q^*\Omega^1_{X_e/\Spec(D_0)})|_{X_k} \ar[r] \ar[d]^\cong & 0 \\
		(g_e^*\Omega^1_{B_{k + 1}})|_{X_k} \ar[r] & (\Omega^1_{Y_e})|_{X_k} \ar[r] & (\Omega^1_{Y_e/B_{k + 1}})|_{X_k} \ar[r] & 0 \\
	}
	\]
	of coherent sheaves on $X_k$. The rows are exact. On the smooth locus of $X_0$, the two maps on the left are injective and locally split. Furthermore, $\Omega^1_{B_{k + 1}}|_{B_k}$ is free of rank one with generator $(ds)|_{B_k}$ so that the lower left map is injective because $\Gamma(X_k,\cO_{X_k}) \to \Gamma(\m{Reg}(X_k/B_k),\cO_{X_k})$ is injective. Similarly, $\Omega^1_{\Spec(D_0)}|_{B_0}$ is a $\kk$-vector space of dimension one so that also the upper left map is injective.
	
	The map $X_k \to Y_e \to X_e$ factors as $X_k \xrightarrow{h} X_0 \to X_e$, where $h$ arises from the trivialization. Thus, the extension in the upper row is isomorphic to $h^*(e)$, the pullback of the original extension class defining $f_e\colon X_e \to \Spec(D_0)$.
	
	Since $D_0 \to A_{k + 1}$ is given by $a_0 + \lambda\eps \mapsto a_0 + \lambda s^{k + 1}$, we have $d\eps \mapsto (k + 1)s^kds$. Thus, the map $(q^*f_e^*\Omega^1_{\Spec(D_0)})|_{X_k} \to (g_e^*\Omega^1_{B_{k + 1}})|_{X_k}$ is given by $(d\eps)|_{X_k} \mapsto (k + 1)f^\sharp(s)^k(ds)|_{X_k}$. Under our chosen trivializations of the terms on the left, this shows $\eta_k(g_e) = (k + 1)f^\sharp(s)^kh^*(e)$.
	
	Because $f_k\colon X_k \to B_k$ is trivial, $\cE xt^1(\Omega^1_{X_k/B_k},\cO_{X_k})$ is flat over $B_k$, and we have a short exact sequence
	\[0 \to \cE xt^1(\Omega^1_{X_0/B_0},\cO_{X_0}) \otimes (s^k) \xrightarrow{\sigma} \cE xt^1(\Omega^1_{X_k/B_k},\cO_{X_k}) \to \cE xt^1(\Omega^1_{X_{k - 1}/B_{k - 1}},\cO_{X_{k - 1}}) \to 0.\]
	Under this sequence, $\eta_k(g_e) = (k + 1) \cdot \sigma(e \otimes s^k)$.
	
	We consider $f_{k + 1}\colon X_{k + 1} \to B_{k + 1}$. We obtain a deformation $f'_{k + 1}\colon X_{k + 1}' \to \Spec(E_{k + 1})$ as the pushout of 
	\[X_0 \times_{\Spec \kk} B_{k + 1} \leftarrow X_0 \times_{\Spec \kk} B_k \cong X_k \to X_{k + 1}.\]
	We take the trivial deformation as the first summand and $f_{k + 1}$ as the second summand (recall the different roles of the two projections $\pi$ and $\varpi$ above). Since 
	\[\m{Def}(X_0,E_{k + 1}) = \m{Def}(X_0,D_{k + 1}) = \m{Def}(X_0,A_{k + 1}) \times \m{Def}(X_0,D_0),\]
	and since the first component of the image of $f_{k + 1}'$ is the trivial deformation, there is some $e \in \m{Ext}^1(\Omega^1_{X_0/B_0},\cO_{X_0})$ with $f'_{k + 1} \cong f'_e$. Then, we have $0 = \eta_k(f_n) = \eta_k(g_e) = (k + 1) \cdot \sigma(e \otimes s^k)$. Since $\sigma$ is injective and $\m{char}(\kk) = 0$, this shows $e = 0$. Therefore, $f_e$ is the trivial deformation, and so is $g_e \cong f_{k + 1}$.
\end{proof}

\begin{cor}\label{cor:condition-a-criterion}\note{cor:condition-a-criterion}
	Assume that we are in Situation~\ref{sitn:explicit-gls-candidate-kk-intro}, and suppose that $f\colon X \to B$ is pre--gls. Consider a closed point $p \in D$. Assume that $p$ is a dnc point (see Definition~\ref{defn:dnc-point}). Fix $\ell \geq 1$. If $\ell = 1$, then $p$ satisfies Condition~$(\spadesuit)$ in Lemma~\ref{lem:dtd-intermediate}. If $\ell \geq 2$, then $p$ satisfies Condition~$(\spadesuit)$ if and only if $\eta_{\ell - 2}(f)|_p = 0$ in the stalk $\cE xt^1(\Omega^1_{X_{\ell - 2}/B_{\ell - 2}},\cO_{X_{\ell - 2}})_p$.
\end{cor}
\begin{proof}
	If $\ell = 1$, Condition~$(\spadesuit)$ says nothing but that $p$ is a dnc point, which holds by assumption.
	
	Let $\ell \geq 2$. First, assume that Condition~$(\spadesuit)$ is satisfied. Let $\pi\colon W \to X$ be an \'etale morphism with $W = \Spec R_W$ affine, and let $w \in W$ be a point with $\pi(w) = p$, such that $\tilde x$, $\tilde y$, $\tilde z_i$ are induced by functions $\check x$, $\check y$, $\check z_i$ on $W$. By shrinking $W$ in the \'etale topology, we may assume that we have a ring homomorphism 
	\[\phi\colon \enspace \kk[t,x,y,z_1,\ldots,z_{d - 1}]/(t^\ell,xy) \to R_W/(f^\sharp(s)^\ell)\]
	which induces the isomorphism in Condition~$(\spadesuit)$ on Henselian stalks. By shrinking $W$ in the Zariski topology, we may assume that $\Spec(\phi)$ is \'etale.
	
	We have another ring homomorphism 
	\[\phi'\colon \enspace \kk[t,x,y,z_1,\ldots,z_{d - 1}]/(t^\ell,xy) \to R_W/(f^\sharp(s)) \otimes_\kk \kk[t]/(t^\ell),\]
	which is induced by $\phi_0$ and hence \'etale as well. Then, the infinitesimal lifting property of the \'etale morphism $\Spec(\phi')$ yields an isomorphism 
	\[\psi\colon \enspace R_W/(f^\sharp(s)) \otimes_\kk \kk[t]/(t^\ell) \xrightarrow{\cong} R_W/(f^\sharp(s)^\ell)\]
	which is compatible with $\phi$ and $\phi'$ as well as with $t \mapsto f^\sharp(s)$, and which induces the identity modulo $t$ respectively $f^\sharp(s)$. Now, $W_{\ell - 1} \to B_{\ell - 1}$ is a trivial deformation so that the lower exact sequence in Lemma~\ref{lem:eta-f-alternative} splits. But then, we have $\eta_{\ell - 2}(f)|_p = 0$.
	
	Conversely, assume that $\eta_{\ell - 2}(f)|_p = 0$. After shrinking $X$ in the Zariski topology, we can assume that $\eta_{\ell - 2}(f) = 0$, and that $X$ is affine. Then, Lemma~\ref{lem:trivial-defo-characterization} yields that $f_{\ell - 1}\colon X_{\ell - 1} \to B_{\ell - 1}$ is a trivial flat deformation. The trivialization and the isomorphism which establishes that $p$ is a dnc point yield the desired isomorphism in Condition~$(\spadesuit)$, and the elements $\tilde x$, $\tilde y$, $\tilde z_1$, \ldots, $\tilde z_{d - 1} \in \fm_{X,p}^\m{h}$ are obtained as lifts along $\fm_{X,p}^\m{h} \to \fm_{X_{\ell - 1},p}^\m{h}$.
\end{proof}

We now turn to the question if $u \in P$ is a unit in Lemma~\ref{lem:dtd-intermediate} when Condition~$(\spadesuit)$ is satisfied. To this end, we turn our attention to $\eta_{\ell - 1}(f)|_p$. Its restriction to $X_{\ell - 2}$ is zero, so it lies in the kernel of that restriction map. This kernel has a simple description.

\begin{lem}\label{lem:ses-at-lt-point}\note{lem:ses-at-lt-point}
	Let $k \geq 2$, and let $f_{k - 1}\colon X_{k - 1} \to B_{k - 1}$ be pre--gls. Assume that $\eta_{k - 2}(f_{k - 1}) = 0$. Then, we have a short exact sequence 
	\begin{multline*}
		0 \to \cE xt^1(\Omega^1_{X_0/B_0},\cO_{X_0}) \otimes_\kk (s^{k - 1}) \\ \to \cE xt^1(\Omega^1_{X_{k - 1}/B_{k - 1}},\cO_{X_{k - 1}}) \xrightarrow{\rho_{k - 1;k - 2}} \cE xt^1(\Omega^1_{X_{k - 2}/B_{k - 2}},\cO_{X_{k - 2}}) \to 0,
	\end{multline*}
	where $\rho_{k - 1;k - 2}$ is the restriction map, and where $(s^{k - 1}) \subset \kk[s]/(s^k)$ is the ideal generated by $s^{k - 1}$.
\end{lem}
\begin{proof}
	This follows from local triviality of $f_{k - 1}\colon X_{k - 1} \to B_{k - 1}$. The image of $e \otimes s^{k - 1}$ under the injection on the left is given by choosing an arbitrary lift of $e \in \cE xt^1(\Omega^1_{X_0/B_0},\cO_{X_0})$ along the surjection from $\cE xt^1(\Omega^1_{X_{k - 1}/B_{k - 1}},\cO_{X_{k - 1}})$ and multiplying with $s^{k - 1}$. In particular, the sequence is independent of the trivialization.
\end{proof}

Note that on $\cU_1^\times V$, the sheaf $\cE xt^1(\Omega^1_{X_0/B_0},\cO_{X_0})$ is a line bundle on $\cS_1^\times V$, i.e., locally isomorphic to $\cO_{D}$ as a sheaf of $\cO_{V}$-modules.

\begin{prop}\label{prop:effective-dtd-criterion}\note{prop:effective-dtd-criterion}
	Assume that we are in Situation~\ref{sitn:explicit-gls-candidate-kk-intro}, and suppose that $f\colon X \to B$ is pre--gls. Let $p \in D$ be a closed dnc point. Then, $p$ is a dtc degeneration point of kink $\ell = 1$ if and only if $\eta_0(f)|_p$ is a generator of $\cE xt^1(\Omega^1_{V},\cO_{V})_p$. It is a dtc degeneration point of kink $\ell \geq 2$ if and only if the following conditions hold:
	\begin{enum:alph}
		\item $\eta_{\ell - 2}(f)|_p = 0$; in particular, for the stalks at $p$, we have a short exact sequence as in Lemma~\ref{lem:ses-at-lt-point} with $k = \ell$;
		
		\item $\eta_{\ell - 1}(f)|_p$ is a generator of the kernel in this exact sequence.
	\end{enum:alph}
\end{prop}
\begin{proof}
	We start with a preparatory computation of extension classes. Let
	\[\chi\colon \enspace P = \kk[t,x,y,z_1,\ldots,z_{d - 1}] \to R_A\] 
	be of finite type and \'etale, and write $\check x = \chi(x)$, $\check y = \chi(y)$, $\check z_i = \chi(z_i)$. Let $\check u \in R_A$, and let $R_X = R_A/(\check x\check y - t^\ell \check u)$. Let us assume that the induced map $X = \Spec R_X \to \IA^1_t$ is an instance of Situation~\ref{sitn:explicit-gls-candidate-kk-intro} which satisfies Condition~(P).
	
	Let $g_{\ell - 1}\colon M_{\ell - 1} = \Spec P/(t^\ell,xy) \to B_{\ell - 1}$ be the trivial deformation. Then, we have a short exact sequence
	\[0 \to \cO_{M_{\ell - 1}} \xrightarrow{\begin{pmatrix} y \\ x \end{pmatrix}} \cO_{M_{\ell - 1}}^{\oplus 2} \xrightarrow{\begin{pmatrix} dx & dy \end{pmatrix}} \Omega^1_{M_{\ell - 1}/B_{\ell - 1}} \to 0.\]
	Since $X_{\ell - 1} \to M_{\ell - 1}$ is \'etale, we can pull the sequence back to $X_{\ell - 1}$ and still have a short exact sequence. Furthermore, this new sequence fits into a commutative diagram
	\[
	\xymatrixcolsep{5em}
	\xymatrix{
		0 \ar[r] & \cO_{X_{\ell - 1}} \ar[r]^-{\begin{pmatrix} \check y \\ \check x \end{pmatrix}} \ar[d]_{1 \mapsto \check u\ell t^{\ell - 1}} & \cO_{X_{\ell - 1}}^{\oplus 2} \ar[r]^-{\begin{pmatrix} d\check x & d\check y \end{pmatrix}} \ar[d]^{\begin{pmatrix} d\check x & d\check y \end{pmatrix}} & \Omega^1_{X_{\ell - 1}/B_{\ell - 1}} \ar[r] \ar@{=}[d] & 0 \\
		0 \ar[r] & \cO_{X_{\ell - 1}} \ar[r]^-{1 \mapsto dt} & \Omega^1_X|_{X_{\ell - 1}} \ar[r] & \Omega^1_{X_{\ell - 1}/B_{\ell - 1}} \ar[r] & 0 \\
	}
	\]
	because $\check xd\check y + \check yd\check x = \check u\ell t^{\ell - 1}dt + t^\ell d\check u$.
	
	The upper row gives a presentation 
	\[\cH om(\cO_{X_{\ell - 1}}^{\oplus 2}, \cO_{X_{\ell - 1}}) \to \cH om(\cO_{X_{\ell - 1}}, \cO_{X_{\ell - 1}}) \to \cE xt^1(\Omega^1_{X_{\ell - 1}/B_{\ell - 1}},\cO_{X_{\ell - 1}}) \to 0\]
	such that the image $e_0$ of the identity in the middle is the extension class of the upper row. In particular, $e_0$ is a generator of $\cE xt^1(\Omega^1_{X_{\ell - 1}/B_{\ell - 1}},\cO_{X_{\ell - 1}})$, and we have $\eta_{\ell - 1}(f) = \check u\ell t^{\ell - 1}e_0$. 
	
	If $\ell = 1$, then $\eta_0(f) = \check ue_0$ is a generator of $\cE xt^1(\Omega^1_{X_0/B_0},\cO_{X_0})$ in a point $p \in D$ if and only if $(\check x,\check y,\check u) = \cO_{X_{\ell - 1},p}$, which is the case if and only if $\check u \in \cO_{A,p}$ is invertible.
	
	Let $\ell \geq 2$. Then, we also have the analogous presentation of $\cE xt^1(\Omega^1_{X_{\ell - 2}/B_{\ell - 2}},\cO_{X_{\ell - 2}})$. Let $e_0'$ be the image of the identity in the middle term. Then, we have $e_0 \mapsto e_0'$ under the restriction map of extension classes $\rho_{\ell - 1;\ell - 2}$. For $a \in \cO_{X_{\ell - 1}}$, we have $\rho_{\ell - 1;\ell - 2}(ae_0) = 0$ if and only if $a \in (\check x,\check y,t^{\ell - 1})$. Therefore, the kernel of $\rho_{\ell - 1;\ell - 2}$ is generated by $t^{\ell - 1}e_0$. But then, $\check u\ell t^{\ell - 1}e_0$ generates the kernel if and only if $\check u \in \cO_{A,p}$ is invertible.
	
	We proceed to the proof of the stated result. First, we consider the case $\ell = 1$. If $p$ is a dtc degeneration point of kink $\ell = 1$, then $X$ is \'etale locally around $p$ of the form considered in the preceding paragraphs with $\check u = 1$. Thus, $\eta_0(f)|_p$ is a generator of $\cE xt^1(\Omega^1_{X_0/B_0},\cO_{X_0})_p$. 
	
	Next, we consider $\ell \geq 2$. By Corollary~\ref{cor:condition-a-criterion}, we have $\eta_{\ell - 2}(f)|_p = 0$. By Lemma~\ref{lem:trivial-defo-characterization}, $f_{\ell - 1}\colon X_{\ell - 1} \to B_{\ell - 1}$ is a trivial deformation after shrinking $X$ to a suitable Zariski neighborhood of $p$. Thus, we have the exact sequence as claimed. Furthermore, we have a local model as in the preceding paragraphs with $\check u = 1$, and $\eta_{\ell - 1}(f)|_p$ is a generator of the kernel of $\rho_{\ell - 1;\ell - 2}$ at $p$.
	
	For the converse direction, we start with $\ell = 1$. Assume that $\eta_0(f)|_p$ is a generator of $\cE xt^1(\Omega^1_{X_0/B_0},\cO_{X_0})_p$. Condition~$(\spadesuit)$ in Lemma~\ref{lem:dtd-intermediate} holds so that we can write the (strictly) Henselian stalk as $\cO_{X,p}^\m{h} \cong P/(xy - tu)$. But then, we can write $f\colon X \to B$ \'etale locally around $p$ in the form considered at the beginning of this proof. Since $ue_0$ is a generator, we find that $u \in P^\times$, and $p$ is a dtc degeneration point of kink $\ell = 1$.
	
	For $\ell \geq 2$, assume the two conditions stated above. Again, Lemma~\ref{lem:trivial-defo-characterization} guarantees that we have the claimed exact sequence. Furthermore, by Corollary~\ref{cor:condition-a-criterion}, Condition~$(\spadesuit)$ is satisfied. Thus, we can write the (strictly) Henselian stalk as $\cO_{X,p} \cong P/(xy - t^\ell u)$, and again, we have the local model. Since $\eta_{\ell - 1}(f)|_p$ generates the kernel by assumption, we find that $u \in P$ is invertible. But then, Lemma~\ref{lem:dtd-intermediate} shows that $p$ is a dtc degeneration point.
\end{proof}

We reformulate Proposition~\ref{prop:effective-dtd-criterion} in a form that is very suitable for computations.

\begin{cor}\label{cor:dtc-computable-condition}\note{cor:dtc-computable-condition}
	Assume that we are in Situation~\ref{sitn:explicit-gls-candidate-kk-intro}, and assume that $f\colon X \to B$ is pre--gls. Consider the short exact sequence
	\[\eta(f)\colon \quad 0 \to \cO_X \xrightarrow{\iota} \Omega^1_X \xrightarrow{\pi} \Omega^1_{X/B} \to 0.\]
	Let $p \in D$ be a closed dnc point. Then, $p$ is a dtc degeneration point of kink $\ell = 1$ if and only if $\cE xt^1(\Omega^1_X|_V,\cO_{V})_p = 0$ for the stalk at $p$. It is a dtc degeneration point of kink $\ell \geq 2$ if and only if the following conditions hold:
	\begin{enum:alph}
		\item the map 
		\[\xi_{\ell - 2}\colon \enspace \cE xt^1(\Omega^1_{X_{\ell - 2}/B_{\ell - 2}},\cO_{X_{\ell - 2}}) \to \cE xt^1(\Omega^1_X|_{X_{\ell - 2}},\cO_{X_{\ell - 2}})\]
		is injective at $p$;
		\item the multiplication map 
		\[\sigma_{\ell - 1}\colon \enspace \cE xt^1(\Omega^1_X|_{X_{\ell - 1}},\cO_{X_{\ell - 1}}) \xrightarrow{(-) \cdot f^\sharp(s)^{\ell - 1}} \cE xt^1(\Omega^1_X|_{X_{\ell - 1}},\cO_{X_{\ell - 1}})\]
		is the zero homomorphism at $p$.
	\end{enum:alph}
\end{cor}
\begin{proof}
	First, we consider the case $\ell = 1$. Then, we have an exact sequence 
	\[\cH om(\Omega^1_X|_{X_0},\cO_{X_0}) \to \cH om(\cO_{X_0},\cO_{X_0}) \to \cE xt^1(\Omega^1_{X_0/B_0},\cO_{X_0}) \to \cE xt^1(\Omega^1_X|_{X_0},\cO_{X_0}) \to 0.\]
	Thus, $\eta_0(f)|_p$ is a generator if and only if the last term vanishes at $p$.
	
	We proceed to $\ell \geq 2$. We have an analogous exact sequence for $\ell - 2$, which shows that $\eta_{\ell - 2}(f)|_p = 0$ if and only if the last map is injective. If this is the case, we consider the diagram
	\[
		\resizebox{138mm}{!}{
		\xymatrix{
			\cH om(\cO_{X_{\ell - 1}},\cO_{X_{\ell - 1}})_p \ar[r]^-{\eta_{\ell - 1}(f)} \ar[d] & \cE xt^1(\Omega^1_{X_{\ell - 1}/B_{\ell - 1}},\cO_{X_{\ell - 1}})_p \ar[r]^{\xi_{\ell - 1}} \ar[d]^{\rho_{\ell - 1;\ell - 2}} & \cE xt^1(\Omega^1_X|_{X_{\ell - 1}},\cO_{X_{\ell - 1}})_p \ar[r] \ar[d]^{\tilde\rho_{\ell- 1;\ell - 2}} & 0 \\
			\cH om(\cO_{X_{\ell - 2}},\cO_{X_{\ell - 2}})_p \ar[r]^-0 & \cE xt^1(\Omega^1_{X_{\ell - 2}/B_{\ell - 2}},\cO_{X_{\ell - 2}})_p \ar[r]^{\xi_{\ell - 2}}_\cong & \cE xt^1(\Omega^1_X|_{X_{\ell - 2}},\cO_{X_{\ell - 2}})_p \ar[r] & 0 \\
		}
		}
	\]
	with exact rows. The middle and right vertical map are the pullbacks as considered in Lemma~\ref{lem:extension-properties}, and this lemma also shows that the right square is commutative. The left square is commutative because $\rho_{\ell - 1;\ell - 2}(\eta_{\ell - 1}(f)) = \eta_{\ell - 2}(f) = 0$. 
	
	Now, $\eta_{\ell - 1}(f)$ is a generator of the kernel of $\rho_{\ell - 1;\ell - 2}$ if and only if $\m{ker}(\rho_{\ell - 1;\ell - 2}) = \m{ker}(\xi_{\ell - 1})$. Since $\xi_{\ell - 1}$ is surjective, this is equivalent to $\tilde\rho_{\ell - 1;\ell - 2}$ being injective.
	
	By Lemma~\ref{lem:Omega-projdim-1}, $\tilde\rho_{\ell - 1;\ell - 2}$ induces an isomorphism after restriction of the source to $X_{\ell - 2}$. Thus, $\tilde\rho_{\ell - 1;\ell - 2}$ is injective if and only if 
	\[\cE xt^1(\Omega^1_X|_{X_{\ell - 1}},\cO_{X_{\ell - 1}})_p \otimes \m{ker}(\cO_{X_{\ell - 1}} \to \cO_{X_{\ell - 2}})_p \to \cE xt^1(\Omega^1_X|_{X_{\ell - 1}},\cO_{X_{\ell - 1}})_p\]
	is the zero map. Since $\m{ker}(\cO_{X_{\ell - 1}} \to \cO_{X_{\ell - 2}})$ is generated by $f^\sharp(s)^{\ell - 1}$, this is equivalent to multiplication with $f^\sharp(s)^{\ell - 1}$ being the zero map at $p$.
\end{proof}

\subsection{The proof of Theorem~\ref{thm:explicit-gls-criterion-kk}}

In this section, we prove Theorem~\ref{thm:explicit-gls-criterion-kk}. We assume throughout that we are in Situation~\ref{sitn:explicit-gls-candidate-kk-intro}, and that $f\colon X \to B$ is pre--gls.

\begin{clm}\label{clm:injectivity-descends}\note{clm:injectivity-descends}
	Let $E \in [\cS_1^\times V]$ be an irreducible component of $\cS_1^\times V$, and let $k \geq 0$. If 
	\[\xi_i\colon \enspace \cE xt^1(\Omega^1_{X_i/B_i},\cO_{X_i}) \to \cE xt^1(\Omega^1_X|_{X_i},\cO_{X_i})\]
	is injective for $i = k$ at some $p \in E$, then it is injective for all $0 \leq i \leq k$ and at all $p \in E$.
\end{clm}
\begin{proof}
	We saw in the proof of Corollary~\ref{cor:dtc-computable-condition} that injectivity of $\xi_i$ at $p$ is equivalent to $\eta_i(f)|_p = 0$. In particular, if $\eta_k(f)|p = 0$, then we also have $\eta_i(f)|_p = 0$ for all $0 \leq i \leq k$.
	
	We show that $\xi_i$ is injective at all $q \in E$ by induction on $i$. For $i = 0$, its source $\cE xt^1(\Omega^1_{X_0/B_0},\cO_{X_0})$ is a line bundle on $\cS_1^\times V$, and in particular on $E$ in a neighborhood of $E$. Since $E$ is smooth and irreducible, injectivity at some point $p \in E$ implies injectivity everywhere.
	
	Now, suppose that we know injectivity of $\xi_{i - 1}$ everywhere. Then, we have $\eta_{i - 1}(f) = 0$ in a neighborhood of $E$, and therefore the exact sequence of Lemma~\ref{lem:ses-at-lt-point}. Since the kernel of $\xi_i$ is contained in the kernel of $\rho_{i;i - 1}$, it is contained in $\cE xt^1(\Omega^1_{X_0/B_0},\cO_{X_0}) \otimes_\kk (s^i)$. But then, the same argument as for $\xi_0$ shows that $\xi_i$ is injective.
\end{proof}

\begin{clm}\label{clm:injectivity-stops}\note{clm:injectivity-stops}
	Let $p \in \cS_1^\times V$, and assume that
	\[\sigma_{k}\colon \quad \cE xt^1(\Omega^1_X|_{X_{k}},\cO_{X_{k}}) \xrightarrow{(-)\cdot f^\sharp(s)^{k}}\cE xt^1(\Omega^1_X|_{X_{k}},\cO_{X_{k}})\]
	is the zero map at $p$. Also assume that $\xi_{k - 1}$ is injective at $p$. Then, $\xi_{k}$ is not injective at $p$.
\end{clm}
\begin{proof}
	If $\sigma_k$ is the zero map at $p$, then the map $\tilde\rho_{k;k - 1}$ in the proof of Corollary~\ref{cor:dtc-computable-condition} is an isomorphism. Thus, if $\xi_k$ were injective, also $\rho_{k;k - 1}$ would be injective. However, since $\xi_{k - 1}$ is injective, Lemma~\ref{lem:ses-at-lt-point} shows that $\rho_{k;k - 1}$ is not injective.
\end{proof}

Together, these two claims show that $\ell(E)$ in the statement of Theorem~\ref{thm:explicit-gls-criterion-kk} is unique if it exists.

\begin{clm}\label{clm:conditions-are-satisfied}\note{clm:conditions-are-satisfied}
	Assume that there is some $Z_X \subset X$ such that $f\colon (X,X \setminus Z_X|V) \to (B|0)$ is gls. Then, the conditions listed in Theorem~\ref{thm:explicit-gls-criterion-kk} are satisfied.
\end{clm}
\begin{proof}
	The log structures on $B^\star$ and $X^\star$ are trivial so that $f\colon X^\star \to B^\star$ is log smooth around a point $x \in X^\star$ if and only if $f$ is smooth around the point. Thus, we have $\m{Sing}(X^\star/B^\star;d) \subset Z_X^\star$, and therefore $\m{dim}(\m{Sing}(X_b;d)) \leq d - 2$ for $b \not= 0$. Since $X_b$ is pure of dimension $d$ and satisfies Serre's condition $(S_2)$, it is normal.
	
	Let $U = X \setminus Z_X$. We apply Theorem~\ref{thm:toroidal-characterization} to $f\colon (U|\itUpsilon) \to (B|0)$ and obtain a stratification $\itUpsilon = \bigsqcup_{k \geq 0} \cS_k\itUpsilon$ of the central fiber. The closed points $x \in \cS_0\itUpsilon$ are smooth, and the closed points $x \in \cS_1\itUpsilon$ are dnc points. Furthermore, we have $\m{dim}(\itUpsilon \setminus (\cS_0\itUpsilon \cup \cS_1\itUpsilon)) \leq d - 2$. Since $\cC_2^\times V \subset Z_V \cup (\itUpsilon \setminus (\cS_0\itUpsilon \cup \cS_1\itUpsilon))$, we find $\m{dim}(\cC_2^\times V) \leq d - 2$.
	
	Let $p \in D \setminus (\cC_2^\times V \cup Z_V)$ be a closed point. Then, we have $p \in \itUpsilon \cap \cS_1^\times V$, and hence $p \in \cS_1\itUpsilon$ by Proposition~\ref{prop:strata-dnc-rank}. Thus, Theorem~\ref{thm:toroidal-characterization} shows that $p$ is a dtc degeneration point of some kink $\ell = \ell(p) \geq 1$. Now, Corollary~\ref{cor:dtc-computable-condition} shows that $\xi_{\ell - 2}$ is injective at $p$, and that $\sigma_{\ell - 1}$ is the zero map at $p$.
	
	Let $E \in [\cS_1^\times V]$. Since $\m{dim}(E) = d - 1$ and $\m{dim}(Z_V) \leq d - 2$, there is a closed point $p \in E \setminus Z_V$; we find that $\xi_k$ is injective for all $0 \leq k \leq \ell(p) - 2$ and at all $q \in E$, and that $\sigma_{\ell(p) - 1}$ is the zero map at $p$. In particular, $p$ is not contained in the support $Z^{(\ell(p))}$ of the image of $\sigma_{\ell(p) - 1}$, and neither is $E$.
\end{proof}

\begin{clm}
	Assume that the conditions in Theorem~\ref{thm:explicit-gls-criterion-kk} are satisfied, and let $Z_X \subset X$ be given by Formula~\eqref{eqn:clubsuit} in Theorem~\ref{thm:explicit-gls-criterion-kk}. Then, $f$ is gls.
\end{clm}
\begin{proof}
	Let $E \in [\cS_1^\times V]$, and let $p \in E \setminus (\cC_2^\times V \cup Z_V^\circ)$. Let $\ell = \ell(E)$. Then, $\xi_{\ell - 2}$ is injective at $p$, and $\sigma_{\ell - 1}$ is the zero map at $p$ so that $p$ is a dtc degeneration point of kink $\ell$.
	
	We show that $Z_X \subset X$ is closed. Note that $Z_V^\circ,\cC_2^\times V \subset \m{Sing}(X/B;d)$ so that $Z_X \subset \m{Sing}(X/B;d)$. Let $p \notin Z_X$. Then, we have $p \in D \setminus (\cC_2^\times V \cup Z_V^\circ)$ so that $p$ has an \'etale neighborhood where $V$ is the only singular fiber.
	
	Let $E \in [\cS_1^\times V]$. Then, $Z_E \subset E$ is a proper closed subset, and hence $\m{dim}(Z_E) \leq d - 2$. Thus, $Z_V \subset V$ has codimension $\geq 2$. For $0 \not= b \in B$, the subset $Z_b \subset X_b$ has codimension $\geq 2$ because $X_b$ is normal.
	
	On $X \setminus (V \cup Z_X)$, the log structure is trivial, and $f\colon X \to B$ is smooth. Thus, the log structure is fs, and $f\colon (X|V) \to (B|0)$ is log smooth and saturated around such a point.
	
	Let $p \in V \setminus D$. Then, $f$ is smooth in a neighborhood of $p$, hence strict, and therefore log smooth and saturated. Again, the log structure is fs.
	
	Let $p \in D \setminus Z_X$, i.e., $p \in D \setminus (\cC_2^\times V \cup Z_V^\circ)$. Then, $p$ is a dtc degeneration point, and the \'etale local model shows that the log structure on $(X|V)$ is fs near $p$, and that $f$ is log smooth and saturated near $p$.
\end{proof}

\begin{clm}
	Let $p \in \cU_1^\times V$ be such that $f\colon (X|V) \to (B|0)$ is log smooth and saturated at $p$. Then, we have $p \notin Z_V^\circ$.
\end{clm}
\begin{proof}
	Assume to the contrary that there is an irreducible component $E \in [\cS_1^\times V]$ and a point $p \in Z_E = E \cap Z^{(\ell)} \subset \cS_1^\times V$ such that $f\colon (X|V) \to (B|0)$ is log smooth and saturated at $p$. Here, $\ell = \ell(E)$ is the kink in the sense of Condition~\ref{item:ell-exists}. 
	
	Let $U \subseteq X$ be an open neighborhood of $p$ where the map is log smooth and saturated. By Theorem~\ref{thm:toroidal-characterization}, we have a rank stratification $\itUpsilon = \bigsqcup_{k \geq 0} \cS_k\itUpsilon$, and by Proposition~\ref{prop:strata-dnc-rank}, we have $p \in \cS_1\itUpsilon$ since $\cS_1^\times V \cap \itUpsilon = \cS_1\itUpsilon$ and $p \in \cS_1^\times V$. Thus, the local model guaranteed by Theorem~\ref{thm:toroidal-characterization} is $\mu_{\ell';d}\colon M_{\ell';d} \to \IA^1_t$ for some $\ell' \geq 1$.
	
	Corollary~\ref{cor:dtc-computable-condition} shows that $\xi_{\ell' - 2}$ is injective at $p$, and that $\sigma_{\ell' - 1}$ is the zero map at $p$. By Claim~\ref{clm:injectivity-descends}, $\xi_k$ is injective at all $q \in E$ for all $k \leq \ell' - 1$. Furthermore, Claim~\ref{clm:injectivity-stops} guarantees that $\xi_{\ell' - 1}$ is not injective at $p$. Now, let $q \in E \setminus Z_E$ be a dtc degeneration point of kink $\ell$. Then, $\xi_{\ell - 2}$ is injective at $q$, but $\xi_{\ell - 1}$ is not injective at $q$. Thus, we have $\ell = \ell'$.
	
	On the one hand, we have that $\sigma_{\ell - 1}$ is the zero map at $p$, and on the other hand, we have $p \in Z^{(\ell)}$ so that $\sigma_{\ell - 1}$ is not the zero map at $p$. Therefore, we have a contradiction.
\end{proof}

\subsection{The proof of Proposition~\ref{prop:infinitesimal-gls-criterion}}

In this section, we prove Proposition~\ref{prop:infinitesimal-gls-criterion}. We assume that we are in the setting of that proposition throughout. 

\begin{clm}
	There is an open subset $X' \subseteq X$ with $V \subset X'$ such that $X' \to B$ is pre--gls.
\end{clm}
\begin{proof}
	By \cite[Tag~045U]{stacks}, the locus where $f\colon X \to B$ is Cohen--Macaulay is open. By assumption, it contains $V$. In a first step, we shrink $X$ to its Cohen--Macaulay locus. By \cite[Tag~02NM]{stacks}, we can decompose $X = \bigsqcup_e X^{(e)}$ such that each $X^{(e)} \subseteq X$ is open and closed, and such that $X^{(e)} \to B$ has relative dimension $e$. Now, $V$ is pure of dimension $d$, so we have $V \subset X^{(d)} \subseteq X$, and after further shrinking $X$, we can assume that $f\colon X \to B$ has relative dimension $d$.
	
	Since $V$ is reduced, we have $\m{dim}(\m{Sing}(X/B;d) \cap V) \leq d - 1$. Let $A \subset \m{Sing}(X/B;d)$ be an irreducible component with $A \cap V \not= \varnothing$. Then, for a closed point $x \in A \cap V$, we have $\m{dim}(\cO_{A,x}) \leq 1 + \m{dim}(\cO_{V \cap A,x}) \leq d$ by \cite[Tag~00OM]{stacks}. But then, $\m{dim}(A) \leq d$. 
	
	Let $0 \not= b \in B$ be a closed point. Since $A \cap V \not=\varnothing$, we have $A \cap X_b \subsetneq A$ so that $\m{dim}(A \cap X_b) \leq d - 1$ because $A$ is irreducible. Now, the locus where $A \to B$ has fiber dimension $\leq d - 1$ is open in $A$ by \cite[Tag~02FZ]{stacks}, so we also have $\m{dim}(A \cap X_\eta) \leq d - 1$ in the generic fiber $X_\eta$.
	
	We obtain $X' \subseteq X$ by further removing all irreducible components of $\m{Sing}(X/B;d)$ with $A \cap V = \varnothing$. Then, for every $b \in B$, the regular (= smooth) locus $\m{Reg}(X'/B) \cap X'_b$ contains all points of codimension $0$, and $X'_b$ is reduced. Since we are working in characteristic $0$, this implies that all fibers are geometrically reduced.
\end{proof}

Now, Lemma~\ref{lem:eta-f-alternative} allows us to compute $\xi_k$ and $\sigma_{\ell - 1}$ by using $\pi_{k + 1}$ respectively $\Omega^1_{X_{\ell}}|_{X_{\ell - 1}}$. Therefore, the conditions of Proposition~\ref{prop:infinitesimal-gls-criterion} imply that Condition~\ref{item:ell-exists} in Theorem~\ref{thm:explicit-gls-criterion-kk} is satisfied. Thus, the following claim is sufficient to conclude the proof of Proposition~\ref{prop:infinitesimal-gls-criterion}.

\begin{clm}
	With the open subset $X' \subseteq X$ from above, $X'_b$ is normal for all $0 \not= b \in B$.
\end{clm}
\begin{proof}
	First, assume that there is a closed point $0 \not= b \in B$ such that $X'_b$ is not normal. Since $X_b'$ is $(S_2)$, the Serre's condition $(R_1)$ is violated. Then, there is an irreducible component $A \subset \m{Sing}(X'/B;d)$ with $\m{dim}(A \cap X_b') = d - 1$. Since $A \cap V \not= \varnothing$, we have $A \cap X_b' \subsetneq A$, and therefore $\m{dim}(A) > d - 1$. On the other hand, we already saw that $\m{dim}(A) \leq d$, so $\m{dim}(A) = d$. Then, for a closed point $p \in A \cap V$, we have $d = \m{dim}(\cO_{A,p}) \leq 1 + \m{dim}(\cO_{A \cap V,p}) \leq d$, so $\m{dim}(A \cap V) = d - 1$.
	
	Since $A \cap V \subset D$ and $\m{dim}(\cC_2^\times V) \leq d - 2$, there must be an irreducible component $E \in [\cS_1^\times V]$ with $E \subset A \cap V$. Let $\ell = \ell(E) \geq 1$ be its kink in the sense of Condition~\ref{item:ell-exists} in Theorem~\ref{thm:explicit-gls-criterion-kk}. Since $E$ contains a dtc degeneration point of kink $\ell$, we can find a dtc degeneration point $p \in A \cap V$. For the local model $M_{\ell;d} \to \IA^1_t$ of a dtc degeneration point, we have $\m{dim}(\m{Sing}(M_{\ell;d}/\IA^1_t;d)) = d - 1$, contradicting the presence of the $d$-dimensional component $A$ of $\m{Sing}(X'/B;d)$ near $p$.
	
	Now, \cite[Tag~02JS]{stacks} shows that $\m{dim}(\m{Sing}(X'^\star/B^\star;d)) \leq d - 1$ so that $\m{codim}(Z_\eta,X'_\eta) \geq 2$ for the generic point $\eta \in B$. Thus, $X_\eta$ satisfies Serre's condition $(R_1)$ and is therefore normal.
\end{proof}

\section{Relation with toroidal crossing schemes}

In this section, we discuss various results that pertain to toroidal crossing schemes.

\subsection{The map $\eta_V^{(\ell)}\colon \cL\cS_V \to \cE xt^1(\Omega^1_{V},\cO_{V})$}

We prove Theorem~\ref{thm:ell-th-classifying-map}. Recall the monoid homomorphism $\theta_\ell\colon \NN \to P_\ell$ from Definition~\ref{defn:dtd-point}.

\begin{defn}
	Let $(V,\cP,\bar\rho)$ be a toroidal crossing scheme. A closed point $v \in V$ is a \emph{double toroidal crossing (dtc) point of kink $\ell \geq 1$} if there is an isomorphism $\phi\colon \cP_{V,\bar v} \cong P_\ell$ such that $\phi(\bar\rho|_{\bar v}) = \bar t$. \qeddef
\end{defn}

The dtc points of kink $\ell \geq 1$ are precisely the closed points in $\cS_1^{(\ell)}V$. As in the introduction, we set $\cU_1^{(\ell)}V \coloneq \cU_0V \sqcup \cS_1^{(\ell)}V$. This is an open subset of $V$.

\begin{constr}\label{constr:eta-ell}\note{constr:eta-ell}
	Fix $\ell \geq 2$. We construct a map
	\[\eta_V^{(\ell)}\colon \enspace \cL\cS_V \to \cE xt^1(\Omega^1_V,\cO_V)\]
	on $\cU_1^{(\ell)}V$ as follows. Let $W \to \cU_1^{(\ell)}V$ be an affine \'etale open subset, and let $\lambda \in \Gamma(W,\cL\cS_V)$. Then, we obtain a log smooth morphism $f_0\colon (W,\cM_0) \to S_0$ and hence a log smooth deformation $f_\ell\colon (W_\ell,\cM_\ell) \to S_\ell$ which is unique up to non-unique isomorphism.
	
	The affine morphism $f_\ell\colon W_\ell \to S_\ell$ is pre--gls and induces a class 
	\[\eta_{\ell - 1}(f_\ell) \in \Gamma(W_{\ell - 1},\cE xt^1(\Omega^1_{W_{\ell - 1}/S_{\ell - 1}},\cO_{W_{\ell - 1}}))\] 
	as discussed before Lemma~\ref{lem:trivial-defo-characterization}. Since $f_\ell$ is \'etale locally isomorphic to \[\Spec\kk[t,x,y,z_1,\ldots,z_{d - 1}]/(xy - t^\ell,t^{\ell + 1}) \to \Spec \kk[t]/(t^{\ell + 1}),\]
	its base change $f_{\ell - 1}\colon W_{\ell - 1} \to S_{\ell - 1}$ is locally trivial as a flat deformation. Thus, we have $\eta_{\ell - 2}(f_\ell) = 0$ by Lemma~\ref{lem:trivial-defo-characterization}, and therefore $\rho_{\ell - 1;\ell - 2}(\eta_{\ell - 1}(f_\ell)) = 0$. Then, we obtain an element $\eta_V^{(\ell)}(f_\ell) \in \Gamma(W,\cE xt^1(\Omega^1_V,\cO_V))$ by means of the exact sequence in Lemma~\ref{lem:ses-at-lt-point}.\footnote{Note that the identification of $\cE xt^1(\Omega^1_V,\cO_V)$ with the kernel is canonical once a chart $\NN \to \cM_{S_\ell}$ and hence a generator of the kernel of $\cO_{S_{\ell - 1}} \to \cO_{S_{\ell - 2}}$ is chosen.}
	
	Let $f'_\ell\colon (W_\ell',\cM_\ell') \to S_\ell$ be another log smooth deformation. Any chosen isomorphism between them identifies $\eta_{\ell - 1}(f_\ell)$ with $\eta_{\ell - 1}(f'_\ell)$. The identification of the kernel of $\rho_{\ell - 1;\ell - 2}$ with $\cE xt^1(\Omega^1_W,\cO_W)$ is compatible with this isomorphism, so we have $\eta_V^{(\ell)}(f_\ell) = \eta_V^{(\ell)}(f_\ell')$, and hence a well-defined element $\eta_V^{(\ell)}(\lambda) \in \Gamma(W,\cE xt^1(\Omega^1_V,\cO_V))$. This construction is compatible with restrictions to smaller open subsets, so we have a sheaf map. \qedloz
\end{constr}
\begin{rem}
	Note that we cannot extend $\eta_V^{(\ell)}$ to the whole of $V$ since log smooth deformations to order $\ell - 1$ are not locally trivial outside $\cU_1^{(\ell)}V$. \qedloz
\end{rem}

We make the local model for a log smooth morphism arising from $\lambda \in \Gamma(W,\cL\cS_V)$ explicit. This is a variant and special case of Theorem~\ref{thm:toroidal-characterization}. 

\begin{lem}\label{lem:U1ell-local-model}\note{lem:U1ell-local-model}
	Let $W \to \cU_1^{(\ell)}V$ be \'etale, let $\bar w \to W$ be a $\kk$-valued point lying in $D^{(\ell)} = \m{Sing}(\cU_1^{(\ell)}V;d)$, and let $\lambda \in \Gamma(W,\cL\cS_V)$ be represented by $(\cM_0,\alpha,q,\rho)$. Let $f_0\colon (W,\cM_0) \to S_0$ be the associated log smooth morphism. Then, $\bar w$ admits an \'etale neighborhood $\bar w \to \widetilde W \to W$ and
	\begin{enum:alph}
		\item $\xi,v \in \Gamma(\widetilde W,\cM_0)$ with $\xi v = \rho^\ell$ and $\tilde x|_{\bar w},\tilde y|_{\bar w} \in \fm_{V,\bar w}$ for $\tilde x \coloneq \alpha(\xi)$ and $\tilde y \coloneq \alpha(v)$,
		\item $\tilde z_1,\ldots,\tilde z_{d - 1} \in \Gamma(\widetilde W,\cO_V)$ with $\tilde z_i|_{\bar w} \in \fm_{V,\bar w}$,
	\end{enum:alph}
	such that the morphism of log schemes 
	\[\phi\colon \enspace (\widetilde W,\cM_0) \to \Spec(P_\ell \to \kk[x,y,z_1,\ldots,z_{d - 1}]/(xy))\]
	given by 
	\[\phi^\sharp(x) = \tilde x, \enspace \phi^\sharp(y) = \tilde y, \enspace \phi^\sharp(z_i) = \tilde z_i, \quad \phi^\flat(\bar x) = \xi, \enspace \phi^\flat(\bar y) = v, \enspace \phi^\flat(\bar t) = \rho,\]
	is strict and \'etale.
\end{lem}
\begin{proof}
	This is a special case of \cite[Thm.~IV.3.3.3]{LoAG2018}.
\end{proof}

When one local model is chosen, we can choose a closely related local model for a second section of $\cL\cS_V$.

\begin{prop}\label{prop:second-section-local-model}\note{prop:second-section-local-model}
	In the situation of Lemma~\ref{lem:U1ell-local-model}, assume that the data in its conclusion have been chosen. Let $\lambda' \in \Gamma(\widetilde W,\cL\cS_V)$ be another section, represented by $(\cM_0',\alpha',q',\rho')$. Let $f_0'\colon (\widetilde W,\cM_0') \to S_0$ be the corresponding log smooth morphism. Then, $\bar w$ admits an \'etale neighborhood $\bar w \to \widehat W \to \widetilde W$ and
	\begin{enum:alph}
		\item two invertible functions $u_x, u_y \in \Gamma(\widehat W,\cO_V^\times)$,
		\item another invertible function $u_t \in \Gamma(\widehat W,\cO_V^\times)$ with $u_xu_y = u_t^\ell$,
		\item two sections $\xi',v' \in \Gamma(\widehat W,\cM_0')$ such that $\alpha'(\xi') = u_x \cdot \alpha(\xi)$ and $\alpha'(v') = u_y \cdot \alpha(v)$, as well as $\xi' v' = (\rho')^\ell$,
	\end{enum:alph}
	such that:
	\begin{enum:arabic}
		\item We have an isomorphism $\psi\colon (\widehat W,\cM_0) \to (\widehat W,\cM_0')$ of log schemes which is the identity on underlying schemes, and which satisfies $\psi^\flat(\xi') = u_x\xi$, $\psi^\flat(v') = u_y v$, and $\psi^\flat(\rho') = u_t\rho$. In general, we have $f_0' \circ \psi \not= f_0$.
		\item The morphism 
		\[\phi'\colon \enspace \widehat W \to \Spec(P_\ell \to \kk[x,y,z_1,\ldots,z_{d - 1}]/(xy))\]
		of log schemes over $S_0$ given by 
		\[\phi'^\sharp(x) = u_x \cdot \tilde x, \enspace \phi'^\sharp(y) = u_y \cdot \tilde y, \enspace \phi'^\sharp(z_i) = \tilde z_i, \quad \phi'^\flat(\bar x) = \xi', \enspace \phi'^\flat(\bar y) = v', \enspace \phi'^\flat(\bar t) = \rho',\]
		is strict and \'etale. 
	\end{enum:arabic} 
\end{prop}
\begin{proof}
	Let $\gamma\colon P_\ell \to \cM_0$ be the chart for the log structure given by $\gamma(\bar x) = \xi$, $\gamma(\bar y) = v$, and $\gamma(\bar t) = \rho$. This chart is neat at $\bar w$ and gives an identification $P_\ell \cong \cP_{\bar w}$.
	
	On the base $S_0$, we have a neat chart $\beta\colon \NN \to \cM_{S_0}$ given by $\beta(1) = \tau_0$. By \cite[Thm.~III.1.2.7]{LoAG2018}, after shrinking $\widetilde W$ around $\bar w$, we can find a monoid $P'$, a monoid homomorphism $\theta\colon \NN \to P'$, and a chart $\gamma'\colon P' \to \cM_0'$ such that $\rho' = (f_0')^\flat(\tau_0) = \gamma'(\theta(1))$, and such that $\gamma'$ induces an isomorphism $P' \cong \cP_{\bar w}$. Then, we can compose with the identification $\cP_{\bar w} \cong P_\ell$ induced by $\gamma$ and consider $\gamma'$ as a chart $\gamma'\colon P_\ell \to \cM_0'$. Under this identification, we have $\theta(1) = \bar t$. In particular, we have $\gamma'(\bar t) = \rho'$.
	
	Let $\xi' = \gamma'(\bar x)$ and $v' = \gamma'(\bar y)$. Since $\gamma'$ is a monoid homomorphism, we have $\xi'v' = (\rho')^\ell$.
	
	Because $\alpha\colon \cM_0 \to \cO_{\widetilde W}$ and $\alpha'\colon \cM_0' \to \cO_{\widetilde W}$ are log structures of the same type, the two induced maps $\bar\alpha,\bar\alpha'\colon \cP \rightrightarrows \cO_{\widetilde W}/\cO_{\widetilde W}^\times$ coincide by \cite[Lemma~9.23]{FeltenGLDT}. Thus, after shrinking $\widetilde W$ in the \'etale topology, we can assume that we have $u_x,u_y \in \Gamma(\widetilde W,\cO_{\widetilde W}^\times)$ with $\alpha'(\xi') = u_x \cdot \alpha(\xi)$ and $\alpha'(v') = u_y \cdot \alpha(v)$.
	
	In a strictly Henselian local ring whose residue field has characteristic $0$, we can take $\ell$-th roots of invertible elements. Thus, after further shrinking $\widetilde W$, we can assume that there is some $u_t \in \Gamma(\widetilde W,\cO_{\widetilde W}^\times)$ with $u_xu_y = u_t^\ell$.
	
	We define a new homomorphism $\check\gamma\colon P_\ell \to \cM_0$ by $\check\gamma(\bar x) = u_x\xi$, $\check\gamma(\bar y) = u_yv$, and $\check\gamma(\bar t) = u_t\rho$. Since $\xi v = \rho^\ell$ and $u_xu_y = u_t^\ell$, this is well-defined. 
	
	Let $\cN_0 \to \cO_{\widetilde W}$ be the log structure associated with $\check\gamma$. Then, we have a homomorphism of log structures $\chi^\flat\colon \cN_0 \to \cM_0$. For a point $\bar u \in \widetilde W$, we have $\check\gamma^{-1}(\cO_{\widetilde W,\bar u}) = \gamma^{-1}(\cO_{\widetilde W,\bar u})$ because $\gamma$ and $\check\gamma$ differ only by multiplication with a unit. Thus, $\chi^\splus\colon \overline\cN_0 \to \overline\cM_0$ is an isomorphism, and so is $\chi^\flat$ because both $\cN_0$ and $\cM_0$ are integral log structures.
	
	We have $\alpha(\check\gamma(\bar x)) = u_x\alpha(\xi) = \alpha'(\gamma'(\bar x))$, $\alpha(\check\gamma(\bar y)) = u_y\alpha(v) = \alpha'(\gamma'(\bar y))$, and $\alpha(\check\gamma(\bar t)) = u_t\alpha(\rho) = 0 = \alpha'(\gamma'(\bar t))$. Since $\gamma'$ is a chart for $\cM_0'$, this shows that we have an induced map of log structures $\psi^\flat\colon \cM_0' \to \cM_0$. Since $\check\gamma$ is a chart for $\cM_0$, and since $\alpha \circ \check\gamma = \alpha' \circ \gamma'$, $\psi^\flat$ is an isomorphism. This yields the claimed isomorphism of log schemes $\psi\colon (\widetilde W,\cM_0) \to (\widetilde W,\cM_0')$ whose underlying morphism of schemes is the identity.
	
	With the elements chosen above, the morphism 
	\[\phi'\colon \widetilde W \to \Spec(P_\ell \to \kk[x,y,z_1,\ldots,z_{d - 1}]/(xy))\]
	defined in the statement is strict. Furthermore, we have $\phi'(\bar w) = 0$.
	
	Recall the ring $R_{\times,d}^\m{h}$ from Definition~\ref{defn:dnc-point}. Since the morphism $\phi$ from Lemma~\ref{lem:U1ell-local-model} is \'etale and $\phi(\bar w) = 0$, we have an induced isomorphism $\phi^\sharp\colon R_{\times,d}^\m{h} \to \cO_{V,\bar w}$. In particular, $\fm_{V,\bar w}/\fm_{V,\bar w}^2$ has a basis given by $[\tilde x]$, $[\tilde y]$, $[\tilde z_1]$, \ldots, $[\tilde z_{d - 1}]$. But then, also $[u_x\tilde x]$, $[u_y\tilde y]$, $[\tilde z_1]$, \ldots, $[\tilde z_{d - 1}]$ is a basis, and $d\phi'^*\colon \fm_{\times,d}^\m{h}/(\fm_{\times,d}^\m{h})^2 \to \fm_{V,\bar w}/\fm_{V,\bar w}^2$ is an isomorphism. 
	
	After completion, we find that $(\phi')^\sharp\colon \widehat R_{\times,d}^\m{h} \to \widehat\cO_{V,\bar w}$ is surjective, so Lemma~\ref{lem:complete-Henselian} yields that $\phi'^\sharp\colon R_{\times,d}^\m{h} \to \cO_{V,\bar w}$ is surjective. Since $\cO_{V,\bar w} \cong R_{\times,d}^\m{h}$ by means of $\phi^\sharp$, both rings have precisely two minimal prime ideals. Let $\fp_1,\fp_2 \subset \cO_{V,\bar w}$ be the minimal primes. Since both rings have the same dimension $d$, the preimage $\fq_i$ of $\fp_i$ is a minimal prime in $R_{\times,d}^\m{h}$, and since $(\phi')^\sharp$ is surjective, we have $\fq_1 \not=\fq_2$. Then, we have $(\phi'^\sharp)^{-1}(0) = (\phi'^\sharp)^{-1}(\fp_1 \cap \fp_2) = \fq_1 \cap \fq_2 = (0)$ because both rings are reduced. But then, $\phi'^\sharp\colon R_{\times,d}^\m{h} \to \cO_{V,\bar w}$ is an isomorphism, and $\phi'$ becomes \'etale after further shrinking $\widetilde W$. We denote the final $\widetilde W$ by $\widehat W$.
\end{proof}

We establish the relation between $\eta_V^{(\ell)}(\lambda)$ and $\eta_V^{(\ell)}(\lambda')$. To this end, we need the following result.

\begin{lem}
	In Proposition~\ref{prop:second-section-local-model}, let $M_0 = \Spec\kk[x,y,z_1,\ldots,z_{d - 1}]/(xy)$, and let 
	\[e\colon \quad 0 \to \cO_{M_0} \to \cE_0 \to \Omega^1_{M_0} \to 0\]
	be a short exact sequence. Then, we have $\phi^*(e) = u_xu_y \cdot \phi'^*(e)$ in $\Gamma(\widehat W,\cE xt^1(\Omega^1_{\widehat W},\cO_{\widehat W}))$.
\end{lem}
\begin{proof}
	First, we consider the short exact sequence 
	\[e_0\colon \quad 0 \to \cO_{M_0} \xrightarrow{\iota} \cO_{M_0} \cdot \delta_x \oplus \cO_{M_0} \cdot \delta_y \xrightarrow{\pi} \Omega^1_{M_0} \to 0\]
	given by $\iota(1) = y\delta_x + x\delta_y$ and $\pi(\delta_x) = dx$, $\pi(\delta_y) = dy$. After pulling back to $\widehat W$, we obtain a solid diagram
	\[
		\xymatrix{
			0 \ar[r] & \phi^*\cO_{M_0} \ar[r] & \phi^*\cO_{M_0} \cdot \delta_x \oplus \phi^*\cO_{M_0} \cdot \delta_y \ar[r] & \phi^*\Omega^1_{M_0} \ar[r] \ar[d]^\cong & 0 \\
			& \cO_{\widehat W} \ar[u]^\omega & & \Omega^1_{\widehat W} & \\
			0 \ar[r] & \phi'^*\cO_{M_0} \ar[r] \ar[u]_\cong & \phi'^*\cO_{M_0} \cdot \delta_x \oplus \phi'^*\cO_{M_0} \cdot \delta_y \ar[r] \ar@{-->}[uu]_\xi & \phi'^*\Omega^1_{M_0} \ar[r] \ar[u]_\cong & 0 \\
		}
	\]
	of coherent sheaves on $\widehat W$. The map $\omega\colon \cO_{\widehat W} \to \phi^*\cO_{M_0}$ is given by $\omega(1) = u_xu_y\phi^*(1)$. We will see in a moment why we choose this map like this.
	
	Assume that $\widehat W$ is affine, and let $\eps_x,\eps_y \in \Gamma(\widehat W,\phi^*\cO_{M_0} \cdot \delta_x \oplus \phi^*\cO_{M_0} \cdot \delta_y)$ be preimages of $du_x,du_y \in \Gamma(\widehat W,\Omega^1_{\widehat W})$. Then, we define the dashed vertical homomorphism by $\xi(\delta_x) = u_x\delta_x + \tilde x \eps_x$ and $\xi(\delta_y) = u_y\delta_y + \tilde y\eps_y$. An easy computation shows that the right-hand square is commutative. However, due to our choice of $\omega$, also the left-hand square is commutative. Namely, we have $\xi(u_y\tilde y\delta_x + u_x\tilde x\delta_y) = u_y\tilde y(u_x\delta_x + \tilde x\eps_x) + u_x\tilde x(u_y\delta_y + \tilde y\eps_y) = u_xu_y(\tilde y\delta_x + \tilde x\delta_y)$. This isomorphism of short exact sequences shows that $\phi^*(e_0) = u_xu_y\cdot \phi'^*(e_0)$.
	
	Since $e_0$ is a generator of $\Gamma(M_0,\cE xt^1(\Omega^1_{M_0},\cO_{M_0}))$, we can find $c \in \Gamma(M_0,\cO_{M_0})$ with $e = c \cdot e_0$ in the general case. By Lemma~\ref{lem:extension-properties}, we then have, on the one hand, $\phi^*(e) = \phi^\sharp(c) \cdot \phi^*(e_0) = \phi^\sharp(c) \cdot u_xu_y \cdot \phi'^*(e_0)$, and on the other hand, $\phi'^*(e) = \phi'^*(c \cdot e_0) = \phi'^\sharp(c) \cdot \phi'^*(e_0)$. Since $\phi^\sharp(c) - \phi'^\sharp(c) \in (\tilde x,\tilde y)$ and $\tilde x\cdot \tilde e = \tilde y \cdot \tilde e = 0$ for every $\tilde e \in \Gamma(\widehat W,\cE xt^1(\Omega^1_{\widehat W},\cO_{\widehat W}))$, we find that $\phi^*(e) = u_xu_y \cdot \phi'^*(e)$.
\end{proof}

\begin{cor}\label{cor:eta-lambda-difference}\note{cor:eta-lambda-difference}
	In Proposition~\ref{prop:second-section-local-model}, we have $\eta_V^{(\ell)}(\lambda) = u_xu_y \cdot \eta_V^{(\ell)}(\lambda')$.
\end{cor}
\begin{proof}
	Let $\lambda_0 \in \Gamma(M_0,\cL\cS_{M_0})$ be the canonical log smooth structure. Since $\eta_V^{(\ell)}$ is a map of sheaves in the \'etale topology, we have $\eta_V^{(\ell)}(\lambda) = \phi^*(\eta_{M_0}^{(\ell)}(\lambda_0)) = u_xu_y \cdot \phi'^*(\eta_{M_0}^{(\ell)}(\lambda_0)) = u_xu_y \cdot \eta_V^{(\ell)}(\lambda')$.
\end{proof}

Next, we answer the question when we have $\lambda = \lambda'$ in terms of the data in Proposition~\ref{prop:second-section-local-model}. From this, we will see that $\eta_V^{(\ell)}$ is injective.

\begin{lem}\label{lem:lambda-difference}\note{lem:lambda-difference}
	In Proposition~\ref{prop:second-section-local-model}, let $\widehat D = \{\tilde x = \tilde y = 0\} \subset \widehat W$. Then, we have $\lambda = \lambda'$ if and only if $u_xu_y|_{\widehat D} = 1$ in $\cO_{\widehat D}$.
\end{lem}
\begin{proof}
	First, we assume that $\lambda = \lambda'$. Then, we have an isomorphism $\chi\colon (\widehat W,\cM_0) \to (\widehat W,\cM_0')$ which is compatible with $q,q'$, $\alpha,\alpha'$, and $\rho,\rho'$. There are $\tilde u_x,\tilde u_y \in \Gamma(\widehat W,\cO_V^\times)$ with $\chi^\flat(\xi') = \tilde u_x\xi$ and $\chi^\flat(v') = \tilde u_yv$. Then, we have $u_x \cdot \alpha(\xi) = \alpha'(\xi') = \alpha(\tilde u_x\xi) = \tilde u_x \cdot \alpha(\xi)$ and similarly $u_y \cdot \alpha(v) = \tilde u_y \cdot \alpha(v)$. 
	
	Now, $\alpha(\xi)$ is invertible on $(\phi')^{-1}(\{x \not= 0\})$, and $\alpha(v)$ is invertible on $(\phi')^{-1}(\{y \not= 0\})$. Thus, we have $u_x = \tilde u_x$ on $(\phi')^{-1}(\{x \not= 0\})$, and $u_y = \tilde u_y$ on $(\phi')^{-1}(\{y \not= 0\})$. The equalities extend to the closures of the open subsets, and we find $u_xu_y = \tilde u_x\tilde u_y$ after restriction to $D^{(\ell)}$.
	
	On the one hand, we have $\chi^\flat(\xi'v') = \chi^\flat((\rho')^\ell) = \rho^\ell = \xi v$. On the other hand, we have $\chi^\flat(\xi'v') = \tilde u_x\xi \cdot \tilde u_y v$. Thus, we have $\tilde u_x\tilde u_y = 1$, and therefore $u_xu_y|_{D^{(\ell)}} = 1$.
	
	For the converse, let $\widehat C(x) = \{\tilde y = 0\}$, $\widehat C(y) = \{\tilde x = 0\}$, and $\widehat D = \{\tilde x = \tilde y = 0\}$. Then, we have an isomorphism
	\[\cO_{\widehat W} \to \cO_{\widehat C(x)} \times_{\cO_{\widehat D}} \cO_{\widehat C(y)}.\] 
	
	Since $u_xu_y = u_t^\ell$, we find $(u_t^\ell)|_{\widehat D} = 1$. Then, we can find elements $\tilde u_x,\tilde u_y \in \cO_{\widehat W}^\times$ with $\tilde u_y|_{\widehat C(x)} = u_t^{-\ell}|_{\widehat C(x)}$, $\tilde u_y|_{\widehat C(y)} = 1$, $\tilde u_x|_{\widehat C(x)} = 1$, and $\tilde u_x|_{\widehat C(y)} = u_t^{-\ell}|_{\widehat C(y)}$. For these elements, we have $\tilde u_x\tilde u_y = u_t^{-\ell}$.
	
	We define an automorphism $\gamma\colon (\widehat W,\cM_0) \to (\widehat W,\cM_0)$ which is the identity on underlying schemes as follows. We set $\gamma^\flat(\xi) = \tilde u_x\xi$, $\gamma^\flat(v) = \tilde u_yv$, and $\gamma^\flat(\rho) = u_t^{-1}\rho$. Indeed, these specifications are compatible with $\alpha$ and $q$. 
	
	Just as $\psi$, we have $f_0 \circ \gamma \not= f_0$ in general. However, we do have $(\gamma \circ \psi)^\flat(\rho') = \gamma^\flat(u_t\rho) = u_tu_t^{-1}\rho = \rho$. Thus, $\gamma \circ \psi\colon (\widehat W,\cM_0') \to (\widehat W,\cM_0)$ is an isomorphism which establishes $\lambda = \lambda'$.
\end{proof}

\begin{cor}
	The map $\eta_V^{(\ell)}\colon \cL\cS_V \to \cE xt^1(\Omega^1_V,\cO_V)$ is injective.
\end{cor}
\begin{proof}
	Assume to the contrary that we have two sections of the stalk $\cL\cS_{V,\bar v}$ which have the same image under $\eta_V^{(\ell)}$. After shrinking the locus of definition of their representatives, we can assume that we have data as in the conclusion of Proposition~\ref{prop:second-section-local-model}. 
	
	We use the notation from the proof of Corollary~\ref{cor:eta-lambda-difference}. Then, $\eta_V^{(\ell)}(\lambda) = \eta_V^{(\ell)}(\lambda')$ implies $u_xu_y \cdot \phi'^*(\eta_{M_0}^{(\ell)}(\lambda_0)) = \phi'^*(\eta_{M_0}^{(\ell)}(\lambda_0))$. Since $\cE xt^1(\Omega^1_{M_0},\cO_{M_0})$ is a line bundle on $\{x = y = 0\}$ which is trivialized by $\eta_{M_0}^{(\ell)}(\lambda_0)$ (cf.~Proposition~\ref{prop:effective-dtd-criterion}), this implies that $u_xu_y|_{\widehat D} = 1$ in the notation of Lemma~\ref{lem:lambda-difference}. But then, Lemma~\ref{lem:lambda-difference} shows that $\lambda = \lambda'$.
\end{proof}

\begin{cor}
	Let $W \to \cU_1^{(\ell)}V$ be an \'etale open subset, let $\lambda \in \Gamma(W,\cL\cS_V)$, and let $c \in \Gamma(W,\cO_V^\times)$. Then, we have $\eta_V^{(\ell)}(c \cdot \lambda) = c^\ell \cdot \eta_V^{(\ell)}(\lambda)$.
\end{cor}
\begin{proof}
	Again, we can assume that we have the data as in the conclusion of Proposition~\ref{prop:second-section-local-model}. In particular, we have elements $u_x,u_y,u_t \in \Gamma(\widehat W,\cO_V^\times)$ which define an isomorphism 
	\[\psi\colon \enspace (\widehat W,\cM_0) \to (\widehat W,\cM_0)\]
	with $\psi^\flat(c \cdot \rho) = u_t\rho$ as well as $\psi^\flat(\xi') = u_x \cdot \xi$ and $\psi^\flat(v') = u_y \cdot v$.
	
	Let $\tilde u_x,\tilde u_y$ be such that $\psi^\flat(\xi) = \tilde u_x \cdot \xi$ and $\psi^\flat(v) = \tilde u_y \cdot v$. This shows that $\xi' = u_x/\tilde u_x \cdot \xi$ and $v' = u_y/\tilde u_y \cdot v$. Then, we have
	\[c^{-\ell} \cdot \psi^\flat(\rho^\ell) = \psi^\flat(\rho'^\ell) = \psi^\flat(\xi'v') = \psi^\flat(u_x/\tilde u_x \cdot u_y/\tilde u_y \cdot \xi v) = u_x/\tilde u_x \cdot u_y/\tilde u_y \cdot \psi^\flat(\rho^\ell)\]
	so that $c^{-\ell} = u_x/\tilde u_x \cdot u_y/\tilde u_y$.
	
	Since $\psi^\flat$ is compatible with $\alpha$, we have $\tilde x = \alpha(\xi) = \alpha \circ \psi^\flat(\xi) = \tilde u_x \cdot \tilde x$ and $\tilde y = \tilde u_y \cdot \tilde y$. In particular, we have $\tilde u_x|_{\widehat D} = 1$ and $\tilde u_y|_{\widehat D} = 1$. But then, we have $\eta_V^{(\ell)}(\lambda) = u_xu_y \cdot \eta_V^{(\ell)}(\lambda') = u_x/\tilde u_x \cdot u_y/\tilde u_y \cdot \eta_V^{(\ell)}(c \cdot \lambda) = c^{-\ell} \cdot \eta_V^{(\ell)}(c \cdot \lambda)$.
\end{proof}

\begin{cor}
	The image of the map $\eta_V^{(\ell)}\colon \cL\cS_V \to \cE xt^1(\Omega^1_V,\cO_V)$ is given by those elements of $\cE xt^1(\Omega^1_V,\cO_V)$ which are local generators.
\end{cor}
\begin{proof}
	Since $\eta_V^{(\ell)}(\lambda) = \phi^*\eta_{M_0}^{(\ell)}(\lambda_0)$, the image $\eta_V^{(\ell)}(\lambda)$ is a local generator. Now, if $e \in \Gamma(\widetilde W,\cE xt^1(\Omega^1_V,\cO_V))$ locally generates $\cE xt^1(\Omega^1_V,\cO_V)$ for some $\widetilde W \to V$ arising in Lemma~\ref{lem:U1ell-local-model}, then we can write $e = c' \cdot \eta_V^{(\ell)}(\lambda)$ for some $c' \in \Gamma(\widetilde W,\cO_V^\times)$. After passage to an appropriate \'etale cover $\widehat W \to \widetilde W$, we can find $c \in \Gamma(\widehat W,\cO_V^\times)$ with $c^\ell = c'|_{\widehat W}$. Then, we have $\eta_V^{(\ell)}(c \cdot \lambda|_{\widehat W}) = e|_{\widehat W}$. Since $\eta_V^{(\ell)}$ is injective, we can apply \'etale descent and obtain a preimage of $e$ in $\Gamma(\widetilde W,\cL\cS_V)$.
\end{proof}

\subsection{The central fiber as a section of $\cL\cS_V$}

In this section, we prove Lemma~\ref{lem:explicit-gls-as-section-of-LS}.

\begin{proof}[Proof of Lemma~\ref{lem:explicit-gls-as-section-of-LS}]
	By the discussion in \cite[\S9.10]{FeltenGLDT}, there is at most one isomorphism $(\cP,\bar\rho) \cong (\overline\cM_{V},f^\splus(\bar\tau_0))$ of sheaves of monoids with a distinguished section. Thus, it is sufficient to construct it \'etale locally. For $p \in V \setminus D$, both sheaves are constant with stalk $\NN$ and distinguished element $1$.
	
	Now, let $p \in D \setminus (\cC_2^\times V \cup Z_V^\circ)$. In particular, there is an irreducible component $E \in [\cS_1^\times V]$ with $p \in E \setminus Z_V^\circ$. Let $\ell = \ell(E) \geq 1$ be the kink. We consider a (possibly non-commutative) diagram 
	\[
	\xymatrix{
		V \ar@{=}[dd] & W \ar[l] \ar[dr] & & & \\
		& Y \ar[d] \ar[u] \ar@/_/[r]_\pi \ar@/^/[r]^\rho & L \ar[rr] \ar[dd]^\mu & & M_{\ell;d} \ar[dd]^{\mu_{\ell;d}} \\
		V \ar[d]_{f_0} & V' \ar[l] \ar[ur] & & & \\
		B_0 \ar@{=}[rr] & & B_0 \ar[r] & B \ar[r]^{\ms{b}} & \IA^1_t \\
	}
	\]
	in the category of schemes. Here, the right-hand square is commutative and Cartesian. The lower left pentagon involving $V'$ and $B_0$ is also commutative and arises from $p$ being a dtc degeneration point of kink $\ell$. In particular, $V' \to V$ is an \'etale neighborhood of $p \in V$ for a distinguished closed point $p' \in V'$ with $p' \mapsto p$, and $V' \to L$ is an \'etale neighborhood of $0 \in L$ with $p' \mapsto 0$.
	
	The left outer pentagon involving $W$ and $B_0$ is also commutative. It arises from the definition of a toroidal crossing scheme in the formulation of \cite[Defn.~9.36]{FeltenGLDT} after replacing the smooth morphism with an \'etale morphism to $V(\sigma) \times \IA^r = M_{\ell;d} \times_{\IA^1_t} B_0$. In particular, $W \to V$ is an \'etale neighborhood of $p \in V$ with a distinguished closed point $w \in W$ with $w \mapsto p$, and $W \to L$ is an \'etale neighborhood of $0 \in L$ with $w \mapsto 0$.
	
	We set $Y \coloneq V' \times_{V} W$ and obtain a distinguished point $y \in Y$ with $y \mapsto p'$ and $y \mapsto w$. In particular, both $Y \to V'$ and $Y \to W$ are \'etale, and the upper left rectangle is commutative. It follows from this construction that we do not necessarily have $\pi = \rho$. Nonetheless, we have $\pi(y) = \rho(y) = 0 \in L$.
	
	After shrinking $Y$ in the Zariski topology around $y$, we may assume that $Y$ has exactly two irreducible components $Y'$ and $Y''$, and they correspond to the two irreducible components $\{x = 0\}$ and $\{y = 0\}$ of $L$ under the maps $\pi,\rho\colon Y \rightrightarrows L$. However, at this point, it could happen that $Y' = \pi^{-1}(\{x = 0\})$ while $Y' = \rho^{-1}(\{y = 0\})$. If this is the case, we compose (at the beginning of our construction) the map $W \to L$ with the automorphism of $L$ switching $x$ and $y$. This does not affect the commutativity of the left outer pentagon. Thus, we may assume that $Y' = \pi^{-1}(\{x = 0\}) = \rho^{-1}(\{x = 0\})$ and $Y'' = \pi^{-1}(\{y = 0\}) = \rho^{-1}(\{y = 0\})$.
	
	Note that $(\cP,\bar\rho)|_{Y} \cong \rho^{-1}(\overline\cM_{L},\mu^\splus(\bar\tau_0))$ and $(\overline\cM_{V},f^\splus(\bar\tau_0))|_{Y} \cong \pi^{-1}(\overline\cM_{L},\mu^\splus(\bar\tau_0))$ by the very construction of the diagram. Now, we have a map $(\underline P_\ell,(0,1))_{L_0} \to (\overline\cM_{L},\mu^\splus(\bar\tau_0))$ from the constant sheaf that identifies the stalk of $\overline\cM_{L}$ at any point with a quotient of $P_\ell$ by some face. Since the pullback of $(\underline P_\ell,(0,1))_{L_0}$ remains constant, we obtain two maps $(\underline P_\ell,(0,1))_{Y} \to \pi^{-1}(\overline\cM_{L},\mu^\splus(\bar\tau_0))$ and $(\underline P_\ell,(0,1))_{Y} \to \rho^{-1}(\overline\cM_{L},\mu^\splus(\bar\tau_0))$. Now, our assumption on $Y'$ and $Y''$ shows that the identity on $(\underline P_\ell,(0,1))_{Y}$ descends to the desired isomorphism. Namely, on $\{x = 0,y\not= 0\}$, the quotient is always by the face $\langle(-1,\ell)\rangle$, on $\{y = 0,x \not= 0\}$, the quotient is always by the face $\langle (1,0)\rangle$, and on $\{x = y = 0\}$, the quotient is always by the face $\{(0,0)\}$. Here, we have chosen the isomorphism $\kk[P_\ell \oplus \NN^{d - 1}] \cong \kk[t,x,y,z_1,\ldots,z_{d - 1}]$ such that $\m{z}^{(1,0) \oplus 0} \mapsto x$ and $\m{z}^{(-1,\ell) \oplus 0} \mapsto y$.
	
	The isomorphism establishes $f_0\colon V \setminus (\cC_2^\times V \cup Z_V^\circ) \to B_0$ as a section of $\widehat{\cL\cS}_{V}$. It follows from the discussion in \cite[\S9.10]{FeltenGLDT} that this section lies in $\cL\cS_{V} \subseteq \widehat{\cL\cS}_{V}$.
\end{proof}

\subsection{The proof of Proposition~\ref{prop:comparing-sections-of-LS}}

\begin{proof}[Proof of Proposition~\ref{prop:comparing-sections-of-LS}]
	The maps $\xi_{k - 1}$ and $\sigma_{k - 1}$ in the statement of Theorem~\ref{thm:explicit-gls-criterion-kk} depend only on $f_k$. Since $n \geq \ell_f(E)$ and $n \geq \ell_g(E)$ for all $E \in [\cS_1^\times V]$, we find $\ell_f(E) = \ell_g(E)$ for every $E \in [\cS_1^\times V]$. Furthermore, we have $Z_f^{(k)} = Z_g^{(k)}$ as closed subschemes of $X_k \cong Y_k$ for all $k \leq n$. We have $Z_{E;f} = Z_{E;g}$ as closed subsets of $V$, and hence $Z_{V;f} = Z_{V;g} \eqcolon Z_V$.
	
	As we already explained prior to stating Proposition~\ref{prop:comparing-sections-of-LS}, we now have a unique isomorphism of sheaves of monoids with a distinguished section \[(\overline\cM_{X_0 \setminus Z_V},f_0^\splus(\bar\tau_0)) \cong (\overline\cM_{Y_0 \setminus Z_V},g_0^\splus(\bar\tau_0))\]
	from Lemma~\ref{lem:explicit-gls-as-section-of-LS}. Both $f_0$ and $g_0$ induce a section $\lambda_f$ respectively $\lambda_g$ of $\cL\cS_{V \setminus Z_V}$ for the resulting structure of a toroidal crossing scheme $(\cP,\bar\rho)$.
	
	To obtain the desired isomorphism between morphisms of partial log schemes $f_0 \cong g_0$, it is sufficient to show that $\lambda_f = \lambda_g$. At points $p \in V \setminus D$, this is clear because the stalk of $\cL\cS_V$ has only a single element at such points. Now, if $p \in D \setminus Z_V$, there is some $E \in [\cS_1^\times V]$ with $p \in E$. Then, let $\ell = \ell(E)$, so that we have an injective map $\eta_V^{(\ell)}\colon \cL\cS_V \to \cE xt^1(\Omega^1_V,\cO_V)$ in a neighborhood of $p$. The image of $\eta_V^{(\ell)}$ depends only on the flat deformation over $S_\ell$. Since $\ell \leq n$, we have $f_\ell \cong g_\ell$ as flat deformations, and therefore $\eta_V^{(\ell)}(\lambda_f) = \eta_V^{(\ell)}(\lambda_g)$. But then, $\lambda_f = \lambda_g$ near $p$.
\end{proof}

\section{Projective gls families}\label{sect:projective-gls-criterion}\note{sect:projective-gls-criterion}

\begin{proof}[Proof of Proposition~\ref{prop:projective-gls-criterion}]\label{pf:projective-gls-criterion}
	Since $\pi\colon X_\pluscirc \to X$ is smooth, the fiber $X_b$ is normal if and only if the fiber $X_{\pluscirc,b}$ is normal.
	
	By Claim~\ref{clm:intrinsic-dnc-smooth-local}, we have $\pi_0^{-1}(\cU_1^\times V) = \cU_1^\times V_{\pluscirc}$ so that we have $\m{dim}(T) \leq d - 2$ if and only if $\m{dim}(T_{\vartriangle}) \leq d - 1$, where $T_{\vartriangle} = V_{\vartriangle} \setminus \cU_1^\times V_{\vartriangle}$.
	
	We have $\pi_0^{-1}\cS_1^\times V = \cS_1^\times V_{\pluscirc}$. Since $\pi_0$ is smooth and surjective with connected fibers, $\pi_0^{-1}$ induces a bijection between $[\cS_1^\times V]$ and $[\cS_1^\times V_{\pluscirc}]$. Furthermore, $\iota_0^{-1}$ induces a bijection between $[\cS_1^\times V_{\pluscirc}]$ and $[\cS_1^\times V_{\vartriangle}]$. Namely, if $0 \in V_{\vartriangle}$ is not in $\cS_1^\times V_{\vartriangle}$, nothing is to do. If $0 \in \cS_1^\times V_{\vartriangle}$, then $\cS_1^\times V_{\vartriangle}$ is connected and hence irreducible. But then, also $\cS_1^\times V_{\pluscirc}$ is irreducible, and it is non-empty because $\m{dim}(\cS_1^\times V_{\vartriangle}) = d \geq 1$.
	
	Let $\varrho\colon \Omega^1_X \to \Omega^1_{X/B}$ and $\varpi\colon \Omega^1_{X_\pluscirc} \to \Omega^1_{X_\pluscirc/B}$ be the two canonical quotient maps, and consider the map 
	\[\xi_k = \cE xt^1(\varrho|_{X_k},\cO_{X_k})\colon \enspace \cE xt^1(\Omega^1_{X/B}|_{X_k},\cO_{X_k}) \to \cE xt^1(\Omega^1_X|_{X_k},\cO_{X_k}).\]
	Let $E \in [\cS_1^\times V]$. Then, $\xi_k$ is injective at every $p \in E$ if and only if $\pi_k^*\xi_k$ is injective at every $q \in \pi_0^{-1}(E)$ because $\pi_k$ is smooth and surjective. Furthermore, we have 
	\begin{multline*}
		\pi_k^*\xi_k = \cE xt^1((\pi^*\varrho)|_{X_{\pluscirc,k}},\cO_{X_{\pluscirc,k}})\colon \\ \enspace \cE xt^1((\pi^*\Omega^1_{X/B})|_{X_{\pluscirc,k}},\cO_{X_{\pluscirc,k}}) \to \cE xt^1((\pi^*\Omega^1_X)|_{X_{\pluscirc,k}},\cO_{X_{\pluscirc,k}}).
	\end{multline*}

	Since the two relevant short sequences of differential forms are exact and locally split, we have a commutative diagram
	\[
		\xymatrixcolsep{1em}
		\xymatrix{
			\cE xt^1((\pi^*\Omega^1_{X/B})|_{X_{\pluscirc,k}},\cO_{X_{\pluscirc,k}}) \ar@{^(->}[r] \ar[d]^{\pi_k^*\xi_k} & \cE xt^1(\Omega^1_{X_\pluscirc/B}|_{X_{\pluscirc,k}},\cO_{X_{\pluscirc,k}}) \ar@{->>}[r] \ar[d]^{\xi_k'} & \cE xt^1(\Omega^1_{X_\pluscirc/X}|_{X_{\pluscirc,k}},\cO_{X_{\pluscirc,k}}) \ar@{=}[d] \\
			\cE xt^1((\pi^*\Omega^1_X)|_{X_{\pluscirc,k}},\cO_{X_{\pluscirc,k}}) \ar@{^(->}[r] & \cE xt^1(\Omega^1_{X_\pluscirc}|_{X_{\pluscirc,k}},\cO_{X_{\pluscirc,k}}) \ar@{->>}[r] & \cE xt^1(\Omega^1_{X_\pluscirc/X}|_{X_{\pluscirc,k}},\cO_{X_{\pluscirc,k}}) \\
		}
	\]
	with exact and locally split rows. Furthermore, the two terms on the right are zero because $\Omega^1_{X_\pluscirc/X}$ is locally free. Thus, $\pi_k^*\xi_k$ is injective at every point $q \in \pi_0^{-1}(E)$ if and only if $\xi_k'$ is injective at every point $q \in \pi_0^{-1}(E)$.
	
	Let $E_\vartriangle \in [\cS_1^\times V_{\vartriangle}]$ be the connected component with $\iota_0^{-1}(E_\vartriangle) = \pi_0^{-1}(E)$. When the extension of $\xi_k'$ to $E_\vartriangle$ is injective at every point of $E_\vartriangle$, the same is true after restriction to $\iota_0^{-1}(E_\vartriangle)$. Conversely, if it is injective at some point $q \in E_\vartriangle$, then it is injective at every point by Claim~\ref{clm:injectivity-descends}.
	
	Next, we consider the multiplication map
	\[\sigma_{\ell - 1}\colon \quad \cE xt^1(\Omega^1_X|_{X_{\ell - 1}},\cO_{X_{\ell - 1}}) \xrightarrow{(-)\cdot f^\sharp(s)^{\ell - 1}}\cE xt^1(\Omega^1_X|_{X_{\ell - 1}},\cO_{X_{\ell - 1}}).\]
	Let $Z^{(\ell)}$ be the support of its image. Then, the support of the image of $\pi_{\ell - 1}^*\sigma_{\ell - 1}$ is $\pi_{\ell - 1}^{-1}(Z^{(\ell)})$. We have already seen that 
	\[\cE xt^1((\pi^*\Omega^1_X)|_{X_{\pluscirc,\ell - 1}},\cO_{X_{\pluscirc,\ell - 1}}) \xrightarrow{\cong} \cE xt^1(\Omega^1_{X_\pluscirc}|_{X_{\pluscirc,\ell - 1}},\cO_{X_{\pluscirc,\ell - 1}})\]
	so that $\pi_{\ell - 1}^{-1}(Z^{(\ell)})$ is also the support of the image of the multiplication map $\sigma_{\ell - 1}'$ of the family $X_\pluscirc \to B$. When $Z_{\vartriangle}^{(\ell)}$ is the corresponding support of $X_\vartriangle \to B$, then we have $\iota_{\ell - 1}^{-1}(Z_{\vartriangle}^{(\ell)}) = \pi_{\ell - 1}^{-1}(Z^{(\ell)})$. 
	
	If $E \not\subseteq Z^{(\ell)}$, then we also have $E_\vartriangle \not\subseteq Z_{\vartriangle}^{(\ell)}$. Conversely, assume that $E_\vartriangle \not\subseteq Z_{\vartriangle}^{(\ell)}$. If $E \subseteq Z^{(\ell)}$, then we have $\iota_0^{-1}(E_\vartriangle) \subseteq \iota_{\ell - 1}^{-1}(Z_{\vartriangle}^{(\ell)})$. However, $0 \in V_{\vartriangle}$ is contained in the closure of $\iota_0^{-1}(E_\vartriangle)$, so we have $E_\vartriangle \subseteq Z_{\vartriangle}^{(\ell)}$. Thus, $f\colon X \to B$ satisfies the conditions in Theorem~\ref{thm:explicit-gls-criterion-kk} if and only if $f_\vartriangle\colon X_\vartriangle \to B$ satisfies them.
	
	The equality $Z_{X,\vartriangle} \cap X_\pluscirc = \pi^{-1}(Z_X)$ is a straightforward consequence of what we have already established.
\end{proof}

\bibliography{explicit-gls-families.bib}\bibliographystyle{plain}

@book {Sernesi2006,
    AUTHOR = {Sernesi, Edoardo},
     TITLE = {Deformations of algebraic schemes},
    SERIES = {Grundlehren der mathematischen Wissenschaften [Fundamental
              Principles of Mathematical Sciences]},
    VOLUME = {334},
 PUBLISHER = {Springer-Verlag, Berlin},
      YEAR = {2006},
     PAGES = {xii+339},
      ISBN = {978-3-540-30608-5; 3-540-30608-0},
   MRCLASS = {14D15 (14B05 14B10 14B12 14B20 14B25 14D20)},
  MRNUMBER = {2247603},
MRREVIEWER = {Marko\ Roczen},
}

@book {Oda1988,
    AUTHOR = {Oda, Tadao},
     TITLE = {Convex bodies and algebraic geometry},
    SERIES = {Ergebnisse der Mathematik und ihrer Grenzgebiete (3) [Results
              in Mathematics and Related Areas (3)]},
    VOLUME = {15},
      NOTE = {An introduction to the theory of toric varieties,
              Translated from the Japanese},
 PUBLISHER = {Springer-Verlag, Berlin},
      YEAR = {1988},
     PAGES = {viii+212},
      ISBN = {3-540-17600-4},
   MRCLASS = {14L32 (14-02 52A25 52A43)},
  MRNUMBER = {922894},
MRREVIEWER = {I. Dolgachev},
}

@book {LoAG2018,
    AUTHOR = {Ogus, Arthur},
     TITLE = {Lectures on Logarithmic Algebraic Geometry},
    SERIES = {Cambridge Studies in Advanced Mathematics},
    VOLUME = {178},
 PUBLISHER = {Cambridge University Press, Cambridge},
      YEAR = {2018},
     PAGES = {xviii+539},
      ISBN = {978-1-107-18773-3},
   MRCLASS = {14D06 (14A20 14M25)},
  MRNUMBER = {3838359},
       DOI = {10.1017/9781316941614},
       URL = {https://doi.org/10.1017/9781316941614},
}

@Book{Hartshorne1977,
  title     = {Algebraic geometry},
  publisher = {Springer-Verlag, New York-Heidelberg},
  year      = {1977},
  author    = {Hartshorne, Robin},
  isbn      = {0-387-90244-9},
  note      = {Graduate Texts in Mathematics, No. 52},
  mrnumber  = {0463157},
  pages     = {xvi+496},
}

@InCollection{kkatoFI,
  author    = {Kato, Kazuya},
  title     = {Logarithmic structures of {F}ontaine-{I}llusie},
  booktitle = {Algebraic analysis, geometry, and number theory ({B}altimore, {MD}, 1988)},
  publisher = {Johns Hopkins Univ. Press, Baltimore, MD},
  year      = {1989},
  pages     = {191--224},
  mrnumber  = {1463703},
}

@Article{GrossSiebertI,
  author   = {Gross, Mark and Siebert, Bernd},
  title    = {Mirror symmetry via logarithmic degeneration data. {I}},
  journal  = {Journal of Differential Geometry},
  year     = {2006},
  volume   = {72},
  number   = {2},
  pages    = {169--338},
  issn     = {0022-040X},
  mrnumber = {2213573},
}

@Article{GrossSiebertII,
  author   = {Gross, Mark and Siebert, Bernd},
  title    = {Mirror symmetry via logarithmic degeneration data, {II}},
  journal  = {Journal of Algebraic Geometry},
  year     = {2010},
  volume   = {19},
  number   = {4},
  pages    = {679--780},
  issn     = {1056-3911},
  doi      = {10.1090/S1056-3911-2010-00555-3},
  mrnumber = {2669728},
}

@article {FFR2021,
    AUTHOR = {Felten, Simon and Filip, Matej and Ruddat, Helge},
     TITLE = {Smoothing toroidal crossing spaces},
   JOURNAL = {Forum Math. Pi},
  FJOURNAL = {Forum of Mathematics. Pi},
    VOLUME = {9},
      YEAR = {2021},
     PAGES = {Paper No. e7, 36},
   MRCLASS = {14D06 (14D05 14D15 14J32 14J45 32G07)},
  MRNUMBER = {4304077},
MRREVIEWER = {Donatella Iacono},
       DOI = {10.1017/fmp.2021.8},
       URL = {https://doi.org/10.1017/fmp.2021.8},
}

@book {Cox2011,
    AUTHOR = {Cox, David A. and Little, John B. and Schenck, Henry K.},
     TITLE = {Toric varieties},
    SERIES = {Graduate Studies in Mathematics},
    VOLUME = {124},
 PUBLISHER = {American Mathematical Society, Providence, RI},
      YEAR = {2011},
     PAGES = {xxiv+841},
      ISBN = {978-0-8218-4819-7},
   MRCLASS = {14M25 (05A15 05E45 52B12)},
  MRNUMBER = {2810322},
MRREVIEWER = {Ivan Arzhantsev},
       DOI = {10.1090/gsm/124},
       URL = {https://doi.org/10.1090/gsm/124},
}

@Misc{stacks,
  author = {The Stacks Project Authors},
  title  = {Stacks Project},
  note   = {\texttt{http://stacks.math.columbia.edu/}},
  year   = {2025}
}

@article {SchroerSiebert2006,
    AUTHOR = {Schr\"{o}er, Stefan and Siebert, Bernd},
     TITLE = {Toroidal crossings and logarithmic structures},
   JOURNAL = {Adv. Math.},
  FJOURNAL = {Advances in Mathematics},
    VOLUME = {202},
      YEAR = {2006},
    NUMBER = {1},
     PAGES = {189--231},
      ISSN = {0001-8708},
   MRCLASS = {14D06 (14J32)},
  MRNUMBER = {2218822},
MRREVIEWER = {Mark Gross},
       DOI = {10.1016/j.aim.2005.03.006},
       URL = {https://doi.org/10.1016/j.aim.2005.03.006},
}

@book {FeltenGLDT,
    AUTHOR = {Felten, Simon},
     TITLE = {Global logarithmic deformation theory},
    SERIES = {Lecture Notes in Mathematics},
    VOLUME = {2373},
 PUBLISHER = {Springer, Cham},
      YEAR = {[2025] \copyright 2025},
     PAGES = {xlviii+627},
      ISBN = {978-3-031-98750-2; 978-3-031-98751-9},
   MRCLASS = {99-06},
  MRNUMBER = {4999357},
       DOI = {10.1007/978-3-031-98751-9},
       URL = {https://doi.org/10.1007/978-3-031-98751-9},
}

@article {EGAII,
    AUTHOR = {Grothendieck, A.},
     TITLE = {\'El\'ements de g\'eom\'etrie alg\'ebrique. {II}. \'Etude
              globale \'el\'ementaire de quelques classes de morphismes.},
   JOURNAL = {Inst. Hautes \'Etudes Sci. Publ. Math.},
  FJOURNAL = {Institut des Hautes \'Etudes Scientifiques. Publications
              Math\'ematiques},
    NUMBER = {8},
      YEAR = {1961},
     PAGES = {222},
      ISSN = {0073-8301,1618-1913},
   MRCLASS = {14.55},
  MRNUMBER = {217084},
       URL = {http://www.numdam.org/item?id=PMIHES_1961__8__222_0},
}

@misc {FeltenVillegas2026,
  author = {Simon Felten and Andrés Villegas},
  title  = {Deformable log singularities on an affine toroidal crossing space},
  year   = {2026}
}

@misc {FeltenZach2026,
  author = {Simon Felten and Matthias Zach},
  title  = {Some log singular surfaces with unobstructed deformations},
  year   = {2026}
}

@article {EGAIV-2,
    AUTHOR = {Grothendieck, A.},
     TITLE = {\'{E}l\'{e}ments de g\'{e}om\'{e}trie alg\'{e}brique. {IV}. \'{E}tude locale des
              sch\'{e}mas et des morphismes de sch\'{e}mas. {II}},
   JOURNAL = {Inst. Hautes \'{E}tudes Sci. Publ. Math.},
  FJOURNAL = {Institut des Hautes \'{E}tudes Scientifiques. Publications
              Math\'{e}matiques},
    NUMBER = {24},
      YEAR = {1965},
     PAGES = {231},
      ISSN = {0073-8301},
   MRCLASS = {14.00},
  MRNUMBER = {199181},
MRREVIEWER = {H. Hironaka},
       URL = {http://www.numdam.org/item?id=PMIHES_1965__24__231_0},
}

@book{ACGS2025,
 author = {Abramovich, Dan and Chen, Qile and Gross, Mark and Siebert, Bernd},
 title = {Punctured logarithmic maps},
 fseries = {Memoirs of the European Mathematical Society},
 series = {Mem. Eur. Math. Soc.},
 issn = {2747-9080},
 volume = {15},
 isbn = {978-3-98547-086-0; 978-3-98547-586-5},
 year = {2025},
 publisher = {Berlin: European Mathematical Society (EMS)},
 language = {English},
 doi = {10.4171/MEMS/15},
 zbMATH = {7982387}
}

@article {ACGS2020,
    AUTHOR = {Abramovich, Dan and Chen, Qile and Gross, Mark and Siebert,
              Bernd},
     TITLE = {Decomposition of degenerate {G}romov-{W}itten invariants},
   JOURNAL = {Compos. Math.},
  FJOURNAL = {Compositio Mathematica},
    VOLUME = {156},
      YEAR = {2020},
    NUMBER = {10},
     PAGES = {2020--2075},
      ISSN = {0010-437X,1570-5846},
   MRCLASS = {14N35 (14D23 14T20)},
  MRNUMBER = {4177284},
MRREVIEWER = {William\ Liu},
       DOI = {10.1112/s0010437x20007393},
       URL = {https://doi.org/10.1112/s0010437x20007393},
}

@article{Gross2013,
 author = {Gross, Mark and Siebert, Bernd},
 title = {Logarithmic {Gromov}-{Witten} invariants},
 fjournal = {Journal of the American Mathematical Society},
 journal = {J. Am. Math. Soc.},
 issn = {0894-0347},
 volume = {26},
 number = {2},
 pages = {451--510},
 year = {2013},
 language = {English},
 doi = {10.1090/S0894-0347-2012-00757-7},
 zbMATH = {6168513},
 Zbl = {1281.14044}
}

@article {Friedman1983,
    AUTHOR = {Friedman, Robert},
     TITLE = {Global smoothings of varieties with normal crossings},
   JOURNAL = {Ann. of Math. (2)},
  FJOURNAL = {Annals of Mathematics. Second Series},
    VOLUME = {118},
      YEAR = {1983},
    NUMBER = {1},
     PAGES = {75--114},
      ISSN = {0003-486X,1939-8980},
   MRCLASS = {32G11 (14D20 14J28 32G13)},
  MRNUMBER = {707162},
MRREVIEWER = {David\ R.\ Morrison},
       DOI = {10.2307/2006955},
       URL = {https://doi.org/10.2307/2006955},
}

@book {Kollar2023,
    AUTHOR = {Koll\'ar, J\'anos},
     TITLE = {Families of varieties of general type},
    SERIES = {Cambridge Tracts in Mathematics},
    VOLUME = {231},
      NOTE = {With the collaboration of Klaus Altmann and S\'andor J.
              Kov\'acs},
 PUBLISHER = {Cambridge University Press, Cambridge},
      YEAR = {2023},
     PAGES = {xviii+471},
      ISBN = {978-1-009-34610-8},
   MRCLASS = {14J10 (14D20 14E30 14J29)},
  MRNUMBER = {4566297},
MRREVIEWER = {Chenyang\ Xu},
}

@article {Posva2023,
    AUTHOR = {Posva, Quentin},
     TITLE = {Abundance for slc surfaces over arbitrary fields},
   JOURNAL = {\'Epijournal G\'eom. Alg\'ebrique},
  FJOURNAL = {\'Epijournal de G\'eom\'etrie Alg\'ebrique. EPIGA},
    VOLUME = {7},
      YEAR = {2023},
     PAGES = {Art. 5, 23},
      ISSN = {2491-6765},
   MRCLASS = {14E30 (14G17)},
  MRNUMBER = {4582878},
MRREVIEWER = {James\ McKernan},
       DOI = {10.46298/epiga.2023.volume7.8803},
       URL = {https://doi.org/10.46298/epiga.2023.volume7.8803},
}

@misc{Berquist2014,
      title={Embeddings of Demi-Normal Varieties}, 
      author={Jeremy Berquist},
      year={2014},
      eprint={1411.2264},
      archivePrefix={arXiv},
      primaryClass={math.AG},
      url={https://arxiv.org/abs/1411.2264}, 
}

@article {FantechiFranciosiPardini2023,
    AUTHOR = {Fantechi, Barbara and Franciosi, Marco and Pardini, Rita},
     TITLE = {Deformations of semi-smooth varieties},
   JOURNAL = {Int. Math. Res. Not. IMRN},
  FJOURNAL = {International Mathematics Research Notices. IMRN},
      YEAR = {2023},
    NUMBER = {23},
     PAGES = {19827--19856},
      ISSN = {1073-7928,1687-0247},
   MRCLASS = {14D06},
  MRNUMBER = {4675060},
MRREVIEWER = {Luca\ Schaffler},
       DOI = {10.1093/imrn/rnac261},
       URL = {https://doi.org/10.1093/imrn/rnac261},
}

@article {KSB1988,
    AUTHOR = {Koll\'ar, J. and Shepherd-Barron, N. I.},
     TITLE = {Threefolds and deformations of surface singularities},
   JOURNAL = {Invent. Math.},
  FJOURNAL = {Inventiones Mathematicae},
    VOLUME = {91},
      YEAR = {1988},
    NUMBER = {2},
     PAGES = {299--338},
      ISSN = {0020-9910,1432-1297},
   MRCLASS = {14J10 (14D20 14J30 32G10 32G13)},
  MRNUMBER = {922803},
MRREVIEWER = {Yujiro\ Kawamata},
       DOI = {10.1007/BF01389370},
       URL = {https://doi.org/10.1007/BF01389370},
}

@misc{Felten2026-extending,
    AUTHOR = {Felten, Simon},
     TITLE = {Extending log structures across singularities},
      YEAR = {2026},
}

\end{document}